\documentclass[11pt,letterpaper]{article}
\usepackage[T1]{fontenc}
\usepackage[symbol]{footmisc}
\usepackage{amsmath,amssymb,amsfonts,amsthm}
\usepackage{fullpage,authblk,mathabx}
\usepackage{hyperref,cite,color}
\usepackage[capitalise,nameinlink,noabbrev]{cleveref}
\usepackage{enumitem}
\usepackage{graphicx}
\usepackage{tikz}
\usepackage{lineno}
\usepackage{caption}
\usepackage{thm-restate}
\usetikzlibrary{shapes}
\usetikzlibrary{backgrounds}
\usetikzlibrary{positioning}
\usetikzlibrary{decorations.pathreplacing}
\hypersetup{
    colorlinks,
    linkcolor = {red!60!black},
    citecolor = {green!60!black},
    urlcolor  = {blue!60!black},
    pdftitle  = {Erdos-Posa property for induced packings of long S-cycles},
    pdfauthor = {Jungho Ahn, O-joung Kwon}
}
\newtheorem{theorem}{Theorem}[section]
\newtheorem{lemma}[theorem]{Lemma}
\newtheorem{corollary}[theorem]{Corollary}
\newtheorem{proposition}[theorem]{Proposition}
\newtheorem{conjecture}[theorem]{Conjecture}

\newtheorem{observation}[theorem]{Observation}
\crefname{observation}{Observation}{Observations}
\newtheorem{claim}{Claim}
\crefname{claim}{Claim}{Claims}

\newenvironment{subproof}[1][\proofname]{
    
    \begin{proof}[Proof of Claim.]}{\end{proof}
}
\newcommand\abs[1]{\lvert #1\rvert}
\newcommand{\dist}{\mathrm{dist}}

\begin{document}
\date{\today}
\title{Erd\H{o}s--P\'{o}sa property for induced packings of long $S$-cycles}
\author[1]{Jungho~Ahn\thanks{Supported by the National Research Foundation of Korea (NRF) grant funded by the Ministry of Science and ICT (No. RS-2026-25496280).}}
\author[2,3]{O-joung~Kwon\thanks{Supported by the Institute for Basic Science (IBS-R029-C1) and the National Research Foundation of Korea (NRF) grant funded by the Ministry of Science and ICT (No. RS-2023-00211670).}}
\affil[1]{Department of Computer Engineering, Inha University, Incheon, South~Korea}
\affil[2]{Department of Mathematics, Hanyang University, Seoul, South~Korea}
\affil[3]{Discrete Mathematics Group, Institute for Basic Science (IBS), Daejeon, South~Korea}
\affil[ ]{\small\textit{Email addresses:}
    \texttt{junghoahn@inha.ac.kr},
    \texttt{ojoungkwon@hanyang.ac.kr}
}

\maketitle

\begin{abstract}
    The Erd\H{o}s--P\'{o}sa theorem states that for every integer $k\geq1$, every graph contains either $k$ vertex-disjoint cycles or a set of $\mathcal{O}(k\log k)$ vertices meeting all cycles.
    This fundamental min--max duality has been extended to numerous settings, including long cycles, $S$-cycles, that is, cycles containing a vertex in a prescribed set $S$, and cycles satisfying various additional constraints.
    In contrast, much less is known when the packing itself is required to be \emph{induced}, namely, when distinct cycles are vertex-disjoint and have no edges between them.

    We prove that long $S$-cycles admit an induced version of the Erd\H{o}s--P\'{o}sa-type duality.
    More precisely, we show that there exists a polynomial function $f(k,\ell)$ such that for all integers $k\geq1$ and $\ell\geq3$, every graph contains either an induced packing of $k$ $S$-cycles of length at least $\ell$ or a set of at most $f(k,\ell)$ vertices whose closed neighbourhood intersects all $S$-cycles of length at least $\ell$.
    The proof introduces a new ear-decomposition technique based on fragile ears and yields a polynomial-time algorithm for every fixed $\ell$.
\end{abstract}

\section{Introduction}

The classical theorem of Erd\H{o}s and P\'{o}sa~\cite{ErdosP1965} asserts that there exists a function $f(k)=O(k\log k)$ such that every graph contains either~$k$ pairwise vertex-disjoint cycles or a set of at most $f(k)$ vertices intersecting all cycles.
Since its introduction, the Erd\H{o}s--P\'{o}sa theorem has become one of the cornerstones of structural graph theory, inspiring numerous min--max theorems for cycles satisfying additional constraints, such as long cycles, cycles with modularity constraints, labelled cycles, and cycles through prescribed vertices~\cite{Thomassen1988,MoussetNSW17,Reed99,HuynhJW19,PW2012,BruhnJS18,GollinHKOY2025}.

A natural way to strengthen the packing side is to require the packed subgraphs not to be connected by an edge.
For cycles, this means that not only must the cycles be vertex-disjoint, but also no edge is allowed between distinct cycles.
We call such a collection an \emph{induced packing of cycles}.
Induced packings have recently attracted increasing attention through induced analogues of Menger's theorem and related packing problems~\cite{Hendrey_2024,AlbrechtsenHTJKW2024,HickingbothamJ2025,NguyenSS2025,DistelGHLM2026}.
Nevertheless, they remain significantly less understood than ordinary packings.

This additional induced condition fundamentally changes the nature of the problem.
The classical proof of Simonovits~\cite{Simonovits1967}, as well as many subsequent Erd\H{o}s--P\'{o}sa arguments, relies on iteratively extending an ear decomposition.
Each augmentation enlarges the decomposition while preserving all previously constructed cycles.
For induced packings, however, this monotonicity completely disappears.
Adding a new ear may introduce edges towards cycles constructed earlier, preventing them from participating in a future large induced packing.
Consequently, one can no longer argue that a growing ear decomposition automatically yields a growing packing.

In this paper, we prove that this stronger packing notion still satisfies an Erd\H{o}s--P\'{o}sa-type duality for long cycles containing prescribed vertices.
For an integer $\ell\geq3$, a graph~$G$, and a set $S\subseteq V(G)$, an \emph{$(\ell,S)$-cycle} in~$G$ is a cycle of~$G$ of length at least~$\ell$ containing a vertex in~$S$.

\begin{restatable}{theorem}{mainone}\label{thm:main1}
    There exists a function $f(k,\ell)=\mathcal{O}(\ell\cdot k^{13}+\ell^3\cdot k^2)$ such that for all integers $k\geq1$ and $\ell\geq3$, every graph~$G$, and every set $S\subseteq V(G)$, one can find, in $\abs{V(G)}^{\mathcal{O}(\ell)}$ time, either an induced packing of~$k$ $(\ell,S)$-cycles or a set~$X$ of at most $f(k,\ell)$ vertices such that $G-B_G(X,1)$ has no $(\ell,S)$-cycle.
\end{restatable}

Our result generalises the classical Erd\H{o}s--P\'{o}sa theorem~\cite{ErdosP1965}, the Erd\H{o}s--P\'{o}sa theorem for long $S$-cycles~\cite{BruhnJS18}, and the Erd\H{o}s--P\'{o}sa theorem for induced packings of cycles~\cite{coarseEP}, with a weaker bound.
Moreover, our algorithm is constructive and runs in polynomial time for every fixed~$\ell$.

The main challenge in proving \cref{thm:main1} is to control edges between cycles arising from an ear decomposition.
Such edges may prevent vertex-disjoint cycles in the decomposition from forming an induced packing in the original graph.
Our main contribution is a new ear-decomposition framework that restricts the possible positions of these edges.
The framework consists of two related structures, called \emph{$(\ell,S)$-frames} and \emph{$(\ell,S)$-subframes}, which are used at different stages of the proof.
Together, they allow us to obtain either a large induced packing or a bounded set whose closed neighbourhood meets every $(\ell,S)$-cycle.
We describe the main ideas below.

\subsection{Proof overview}

We give an overview of the proof of \cref{thm:main1}.
We omit the precise bounds in this overview and focus on the roles of the main structures.
If~$G$ contains an $(\ell,S)$-cycle of length at most $3(\ell-1)$, we remove its closed neighbourhood and apply induction on $k-1$.
We may therefore assume that every $(\ell,S)$-cycle has length greater than $3(\ell-1)$.
Throughout this paper, we fix an integer $\ell\geq3$.
We say that a \emph{good pair} is a pair $(G,S)$ of a graph~$G$ and a set $S\subseteq V(G)$ such that~$G$ has no $(\ell,S)$-cycle of length at most $3(\ell-1)$.

\paragraph{Frames.}
We first construct a maximal $(\ell,S)$-frame~$\mathcal{H}$, which is an ear decomposition designed to preserve $(\ell,S)$-cycles.
We start with a shortest $(\ell,S)$-cycle and repeatedly add shortest possible $(\ell,S)$-ears.
Here, an $(\ell,S)$-ear is a path with ends in the current frame that contains a vertex of~$S$ or joins distinct components of the frame after deleting~$S$, and its length plus the distance between its ends in the frame is at least~$\ell$.
When no such ear is available, we add a shortest $(\ell,S)$-cycle meeting the current frame in at most one vertex.
If there is no further choice to add, we add a short ear containing a vertex of~$S$ by replacing the subpath of the frame between its ends.
These possibilities form Steps~\ref{item:ear1}--\ref{item:ear4} in \cref{sec:frame}, illustrated in \cref{fig:frame-construction}.

The construction produces two related graphs.
The frame~$\mathcal{H}$ contains every path and cycle added during the construction.
The graph~$\mathcal{H}^-$ is obtained by performing the path replacements used when short ears containing a vertex in~$S$ are added.
We extract cycles from~$\mathcal{H}^-$ rather than directly from~$\mathcal{H}$.
Indeed, \cref{lem:all cycles} shows that every cycle of~$\mathcal{H}^-$ is an $(\ell,S)$-cycle.
Thus, it remains to find vertex-disjoint cycles of~$\mathcal{H}^-$ with no edges between them in~$G$.

\paragraph{Fragile ears and chords.}
The graph~$\mathcal{H}$ need not be an induced subgraph of~$G$.
We call an edge of~$G$ joining two vertices of~$\mathcal{H}$ but not belonging to~$\mathcal{H}$ a \emph{chord} of~$\mathcal{H}$.
A chord may join two cycles that are vertex-disjoint in~$\mathcal{H}^-$ and thereby prevent them from forming an induced packing.
The main purpose of choosing shortest possible ears is to restrict the positions of these chords.

Most chords are local, in the sense that both ends lie within distance~$\ell-1$ from an attachment vertex of an ear.
The main exception occurs on what we call a \emph{fragile ear}.
An ear is fragile when it is an $(\ell,S)$-ear because it contains vertices of~$S$, rather than because its attachment vertices are separated by~$S$ in the preceding frame.
For a fragile ear~$P_{i,j}$, let~$P^*_{i,j}$ be the shortest subpath containing all vertices of~$S$ on~$P_{i,j}$.
We regard~$P^*_{i,j}$ as the central part of the ear and the remaining parts between~$P^*_{i,j}$ and the attachment vertices as its two sides.

The reason for this definition can be seen by considering a short path outside the frame whose ends lie on the same ear.
If its ends are far apart along the ear, replacing the subpath between them produces a shorter candidate ear.
This contradicts the choice of a shortest ear unless the replacement removes every vertex of~$S$ from the ear.
In the exceptional case, the original ear is fragile, its vertices of~$S$ lie between the ends of the short path, and these ends lie on opposite sides of~$P^*_{i,j}$.
This observation is formalised in \cref{lem:short path1}, illustrated in \cref{fig:fragile-ear}.

The resulting description of all chords is given in \cref{prop:induced1,prop:induced2}.
A chord with both ends originating from the same ear is either local to an attachment vertex, or joins the two sides of a fragile ear.
A chord whose ends originate from different ears has a similar form: it is either local to an attachment vertex or involves one side of a fragile ear.
Therefore, every chord that may interfere with cycles extracted from~$\mathcal{H}^-$ is associated with a small neighbourhood of an attachment vertex or with one of the two sides of a fragile ear.

\paragraph{Frames with many branch vertices.}
Suppose that~$\mathcal{H}^-$ has many branch vertices.
The theorem of Simonovits then provides many vertex-disjoint cycles in~$\mathcal{H}^-$.
By \cref{lem:all cycles}, all these cycles are $(\ell,S)$-cycles.
They do not necessarily form an induced packing because chords of~$\mathcal{H}$ may join them.

We select a subcollection of these cycles using the chord description above.
For each selected cycle, the relevant conflicts occur only near its attachment vertices or along the sides of fragile ears that it uses.
The distance conditions in the frame construction ensure that one ear can conflict in this way with only a limited part of the selection.
The selection arguments in \cref{sec:large packings} remove these conflicts while retaining~$k$ cycles.
Consequently, \cref{prop:final} shows that if~$\mathcal{H}^-$ has sufficiently many branch vertices, then~$G$ contains an induced packing of~$k$ $(\ell,S)$-cycles.

\paragraph{From frames to subframes.}
We may now assume that~$\mathcal{H}^-$ has only a bounded number of branch vertices.
We remove the $(\ell-1)$-neighbourhoods of the branch vertices and bounded neighbourhoods of the ends of certain path components of~$\mathcal{H}^-$.
The number of removed vertices is bounded by a polynomial function of~$k$ and~$\ell$ because the frame has few branch vertices.
After this removal, the relevant components of the frame are pairwise disjoint paths.

The maximality of~$\mathcal{H}$ separates these paths from one another.
If a path outside the frame joined one relevant path to another component, it could be extended along the frame to form a new $(\ell,S)$-ear.
This would contradict the maximality of~$\mathcal{H}$.
It follows that every remaining $(\ell,S)$-cycle meets exactly one path component~$F$ of the cleaned frame.
Moreover, distinct path components that meet $(\ell,S)$-cycles lie in distinct components of the graph remaining after the deletion.
We analyse each such component together with its path~$F$ separately.

The path~$F$ has three useful properties within this component.
It is a shortest path between its ends, it admits no $(\ell,S)$-ear, and every $(\ell,S)$-cycle in the component contains at least two vertices of~$F$.
We call a path with these properties an \emph{$(\ell,S)$-central path}.
The first property prevents shortcuts along~$F$, while the other two force every $(\ell,S)$-cycle to repeatedly leave and return to~$F$.
We introduce \emph{$(\ell,S)$-subframes} to describe the cycles arranged around such a central path.

An $(\ell,S)$-subframe starts with an $(\ell,S)$-central path~$F_1$ and records how $(\ell,S)$-cycles intersect this path.
We extend the subframe using short ears containing a vertex of~$S$ whose ends lie on the same path segment of the current subframe.
When such an ear~$R$ is added, we replace the path segment between its ends by~$R$ in an auxiliary graph.
This replacement may allow another path~$R'$ to form an $(\ell,S)$-ear of the modified auxiliary graph.
If such a path~$R'$ exists, we add~$R$ and~$R'$ together and call $(R,R')$ an \emph{admissible pair}.
Otherwise, we add only the short ear~$R$.
The subframe contains all paths added in this process, while the auxiliary graph records the corresponding replacements.
Every cycle of the auxiliary graph is an $(\ell,S)$-cycle by \cref{lem:all cycles-sub}.

The chords of a subframe have a simpler structure than those of a general frame.
By \cref{prop:subframe1}, the two ends of every chord lie within distance~$\ell-1$ from one branch vertex of the subframe.
Therefore, cycles remaining after the deletion of these neighbourhoods have no edges between them in~$G$.
If the subframe has many branch vertices, Simonovits' theorem yields a large induced packing.
If it has few branch vertices, the remaining $(\ell,S)$-cycles are arranged along the central path.
A greedy choice along this path gives either~$k$ pairwise non-adjacent cycles or a bounded set whose closed neighbourhood meets every $(\ell,S)$-cycle.
This argument is summarised in \cref{prop:subframe2}.

In the final proof, we apply \cref{prop:subframe2} to every component associated with a central path.
If at least~$k$ such components contain an $(\ell,S)$-cycle, then choosing one cycle from each component gives an induced packing.
Otherwise, there are fewer than~$k$ components, and we take the union of the hitting sets obtained from \cref{prop:subframe2}.
Together with the vertices removed from the original frame, this union gives the desired set~$X$.

\subsection{Organisation}

The remainder of the paper is organised as follows.
\cref{sec:prelim} introduces notation and preliminary results.
In \cref{sec:frame}, we define $(\ell,S)$-frames and establish their basic properties.
This section also introduces fragile ears.
\cref{sec:chords} analyses the structure of chords and fragile ears in maximal $(\ell,S)$-frames and \cref{sec:large packings} develops several algorithms to find a large induced packing of $(\ell,S)$-cycles.
In \cref{sec:subframe}, we define $(\ell,S)$-subframes which are analogues of $(\ell,S)$-frames under the assumption that an input graph has a specific path intersecting all $(\ell,S)$-cycles.
We show that \cref{thm:main1} holds for the graphs having no short $(\ell,S)$-cycles but admitting $(\ell,S)$-subframes.
\cref{sec:proof} combines these ingredients to prove \cref{thm:main1}.
We leave open problems in \cref{sec:last}.

\section{Preliminaries}\label{sec:prelim}

For an integer~$n$, we denote by~$[n]$ the set of positive integers at most~$n$.
Note that if~$n$ is at most~$0$, then~$[n]$ is empty.
For pairs $(i,j)$ and $(i',j')$ of non-negative integers, we write $(i,j)<_L(i',j')$ if $(i,j)$ is lexicographically smaller than $(i',j')$, and $(i,j)\leq_L(i',j')$ if $(i,j)<_L(i',j')$ or $(i,j)=(i',j')$.
Similarly, we define the relations~$>_L$ and~$\geq_L$.

Every graph in this paper is finite and simple.
For a graph~$G$, we denote by~$V(G)$ and~$E(G)$ its vertex set and edge set, respectively.
A graph is \emph{null} if its vertex set is empty.
For a set $X\subseteq V(G)$, we denote by ${G-X}$ the graph obtained from~$G$ by removing all vertices in~$X$ and all edges incident with vertices in~$X$.
If $X=\{v\}$, then we may write $G-v$ for $G-X$.
The \emph{subgraph of~$G$ induced by~$X$}, denoted by~$G[X]$, is the graph $G-(V(G)\setminus X)$.
For graphs~$G$ and~$H$, the \emph{union} of~$G$ and~$H$, denoted by $G\cup H$, is the graph $(V(G)\cup V(H),E(G)\cup E(H))$.
Similarly, the \emph{intersection} of~$G$ and~$H$, denoted by $G\cap H$, is the graph $(V(G)\cap V(H),E(G)\cap E(H))$.
A \emph{cycle component} of~$G$ is a component of~$G$ which is a cycle.

The \emph{degree} of a vertex~$v$ in~$G$ is denoted by $\deg_G(v)$.
A \emph{branch vertex} of~$G$ is a vertex of degree at least~$3$.
We denote by $V_{\geq3}(G)$ the set of branch vertices of~$G$.
The \emph{average degree} of~$G$ is $\big(\sum_{v\in V(G)}\deg_G(v)\big)/\abs{V(G)}$.

A path is \emph{trivial} if its vertex set is a singleton.
We consider a null graph as a path.
For a path~$P$ in~$G$ and vertices~$a$ and~$b$ of~$P$, we denote by~$aPb$ the subgraph of~$P$ between~$a$ and~$b$.
For sets $X,Y\subseteq V(G)$, an \emph{$(X,Y)$-path} in~$G$ is a path in~$G$ between a vertex in~$X$ and a vertex in~$Y$ such that no internal vertex is contained in $X\cup Y$.
If $X=\{x\}$ or $Y=\{y\}$, then we may write an $(x,Y)$-path, an $(X,y)$-path, or an $(x,y)$-path, as appropriate.
For a subgraph~$H$ of~$G$, an \emph{$H$-path} is a $(V(H),V(H))$-path of length at least one whose edge set is disjoint from~$E(H)$.
A \emph{chord} of~$H$ is an edge $e\in E(G)\setminus E(H)$ between two vertices of~$H$.
We denote by $H+e$ the graph obtained from~$H$ by adding~$e$.
Two paths are \emph{internally disjoint} if no internal vertex of one path is contained in the other path.

An \emph{independent set} of~$G$ is a set of pairwise non-adjacent vertices.
For an integer~$d\geq0$, a graph~$G$ is \emph{$d$-degenerate} if every subgraph of~$G$ has a vertex of degree at most~$d$.
It is well known that every $d$-degenerate graph~$G$ has an independent set of size at least $\abs{V(G)}/(d+1)$, which can be found in linear time.

For vertices~$v$ and~$v'$ of~$G$, the \emph{distance between~$v$ and~$v'$ in~$G$}, denoted by $\dist_G(v,v')$, is the length of a shortest $(v,v')$-path of~$G$.
If~$v$ and~$v'$ are in distinct components of~$G$, then we define $\dist_G(v,v'):=+\infty$.
For sets $X,X'\subseteq V(G)$, let $\dist_G(v,X)=\dist_G(X,v):=\min_{v'\in X}\dist_G(v,v')$ and $\dist_G(X,X'):=\min_{v\in X}\dist_G(v,X')$.
For subgraphs~$H$ and~$H'$ of~$G$, let
\[
    \dist_G(H,H'):=\dist_G(V(H),V(H')).
\]
For an integer $d\geq0$, the \emph{ball of radius~$d$ around~$v$ in~$G$}, denoted by $B_G(v,d)$, is the set of vertices of~$G$ at distance at most~$d$ from~$v$.
Let $B_G(X,d):=\bigcup_{v\in X}B_G(v,d)$ and let $B_G(H,d):=B_G(V(H),d)$.

A vertex is \emph{pendant} if it has degree at most~$1$.
A graph is \emph{subcubic} if every vertex has degree at most~$3$.
We will use the following result of Simonovits~\cite{Simonovits1967}.
For every positive integer~$k$, let
\begin{align*}
    s_k:=
    \begin{cases}
        4k(\log k+\log\log k+4)\quad & \text{if } k>1,\\
        2 & \text{if } k=1.
    \end{cases}
\end{align*}

\begin{theorem}[Simonovits~\cite{Simonovits1967}]
    \label{thm:simonovitz}
    Let~$G$ be a subcubic graph without pendant vertices.
    If $\abs{V_{\geq3}(G)}\geq s_k$ for some positive integer~$k$, then~$G$ contains~$k$ pairwise vertex-disjoint cycles.
    Moreover, these~$k$ cycles can be found in polynomial time.
\end{theorem}

For a partially ordered set~$X$, a \emph{chain} of~$X$ is a set of pairwise comparable elements of~$X$, and an \emph{antichain} of~$X$ is a set of pairwise incomparable elements of~$X$.
A \emph{chain decomposition} of~$X$ is a partition of~$X$ into pairwise disjoint chains.

\begin{theorem}[Dilworth~\cite{Dilworth87}]\label{thm:Dilworth}
    For a finite partially ordered set~$X$, the maximum size of an antichain of~$X$ is equal to the minimum size of a chain decomposition of~$X$.
\end{theorem}

\subsection{Digraphs}

A \emph{digraph} is a pair $D:=(V,A)$ where~$V$ is a finite set and~$A$ is a set of ordered pairs of~$V$.
We call~$V$ and~$A$ the \emph{vertex set} and the \emph{arc set} of~$D$, and denote them by~$V(D)$ and~$A(D)$, respectively.
Each $(u,v)\in A$ is called an \emph{arc from~$u$ to~$v$}.
If $u=v$, then we call it a \emph{loop}.
\emph{Parallel arcs} are distinct arcs starting from and ending with the same vertices.
For a vertex~$v$ of~$D$, the \emph{in-degree} is the number of arcs toward~$v$, and the \emph{out-degree} is the number of arcs from~$v$.

In this paper, every digraph has neither a loop nor parallel arcs.
The \emph{underlying graph} of a digraph~$D$ is a graph~$G$ with vertex set~$V(D)$ such that distinct vertices are adjacent in~$G$ if and only if~$D$ has an arc between them.
An \emph{independent set} of a digraph is an independent set of its underlying graph.

\begin{lemma}\label{lem:digraph}
    Let~$D$ be a digraph where every vertex has out-degree at most~$1$.
    If $\abs{V(D)}\geq3k-2$ for some positive integer $k$, then~$D$ has an independent set of size~$k$.
\end{lemma}
\begin{proof}
    Let~$G$ be the underlying graph of~$D$.
    Since every vertex has out-degree at most~$1$ in~$D$, we have $\abs{A(D)}\leq\abs{V(D)}$.
    Thus, the average degree of~$G$ is at most~$2$.
    By Tur\'{a}n's theorem~\cite{Turan1941} (see also~\cite[Chapter 6]{AlonS2016}), the maximum size of an independent set in~$G$ is at least $\frac{|V(G)|}{2+1}\geq k-\frac{2}{3}$.
    Thus, $D$ has an independent set of size~$k$.
\end{proof}

\section{\texorpdfstring{$(\ell,S)$-frames}{(l,S)-frames}}\label{sec:frame}

To prove \cref{thm:main1}, we introduce $(\ell,S)$-frames in graphs.
We need the following definitions and notations.
Let~$G$ be a graph, let $S:=\{s_1,\ldots,s_m\}$ be a subset of~$V(G)$, and let~$H$ be a subgraph of~$G$.
An \emph{$S$-cycle} is a cycle containing at least one vertex in~$S$.
An \emph{$(\ell,S)$-cycle} is an $S$-cycle of length at least~$\ell$.
An \emph{$S$-ear of~$H$} is an $H$-path~$P$ such that either
\begin{itemize}
    \item $P$ contains a vertex in~$S$, or
    \item the ends of~$P$ are in distinct components of $H-S$.
\end{itemize}
We remark that for an $S$-ear~$P$ of~$H$, every cycle in $H\cup P$ containing at least one edge of~$P$ is an $S$-cycle.
An \emph{$(\ell,S)$-ear of~$H$} is an $S$-ear~$P$ of~$H$ such that for its ends~$a$ and~$b$, it holds that
\[
    \abs{E(P)}+\dist_H(a,b)\geq\ell.
\]
Thus, for an $(\ell,S)$-ear~$P$ of~$H$, every cycle in $H\cup P$ containing at least one edge of~$P$ is an $(\ell,S)$-cycle.
We remark that every $H$-path between distinct components of~$H$ is an $(\ell,S)$-ear of~$H$.

Let $\mathcal{L}_0(H)$ be the graph obtained from $H-V_{\geq3}(H)$ as follows: for each cycle component~$C$ of~$H$, if~$C$ contains a vertex in~$S$, then remove the vertex $s_i\in V(C)\cap S$ with the smallest~$i$, and otherwise remove~$C$.
Note that every component of $\mathcal{L}_0(H)$ is a path.
The graph $\mathcal{L}_0(H)$ is well defined if~$S$ is ordered.
Thus, from now on, when a set $S\subseteq V(G)$ is given, we always assume that it is ordered, without explicitly mentioning it.
Let $\mathcal{L}_1(H)$ be the graph obtained from $\mathcal{L}_0(H)$ by recursively removing a vertex of degree at most~$1$ which is not contained in~$S$.
Note that the ends of each component of $\mathcal{L}_1(H)$ are in~$S$.

A \emph{strict $S$-ear of $H$} is an $H$-path with ends in the same component of $\mathcal{L}_0(H)$ and containing a vertex in $S$.
Let~$P$ be an $H$-path of~$G$ such that the ends of~$P$ are in one component of $\mathcal{L}_0(H)$.
The \emph{base of~$P$ in~$H$} is the shortest path~$Q$ of $\mathcal{L}_0(H)$ containing the ends of~$P$.
We denote by $H\oplus P$ the graph obtained from~$H$ by replacing~$Q$ with~$P$, that is, we remove from~$H$ all edges and internal vertices of~$Q$ and take the union of the resulting graph with~$P$.

We now define an $(\ell,S)$-frame~$\mathcal{H}$ in~$G$.
We begin by considering a null graph~$H_{0,0}$ as a trivial $(\ell,S)$-frame in~$G$.
We define~$H^-_{0,0}$ as a null graph and $\mathcal{P}(H_{0,0})$ as an empty set.
Suppose that we have found an $(\ell,S)$-frame~$H_{i,j}$ in~$G$ for some $i,j\geq0$ with~$H^-_{i,j}$ and $\mathcal{P}(H_{i,j})$.
Let
\begin{align*}
    Y_{i,j}&:=B_{H_{i,j}}(V_{\geq3}(H_{i,j}),\ell-1),\\
    Z_{i,j}&:=B_{G-(V(H_{i,j})\setminus Y_{i,j})}(Y_{i,j},1),\\
    G_{i,j}&:=G-Z_{i,j}.
\end{align*}
We inductively define a larger $(\ell,S)$-frame~$\mathcal{H}$ in~$G$ together with~$\mathcal{H}^-$ and $\mathcal{P}(\mathcal{H})$ as follows; see \cref{fig:frame-construction} for an illustration.

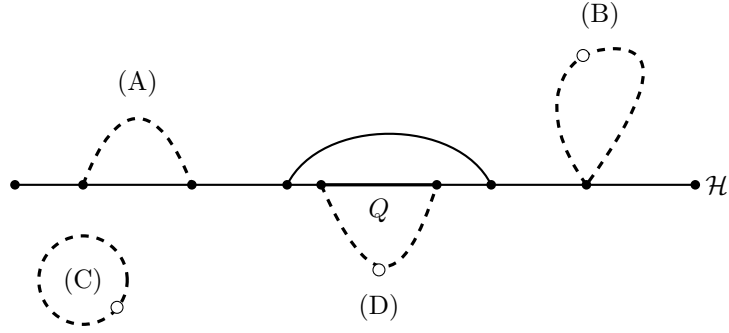
\begin{figure}[t]
    \centering
    \begin{tikzpicture}[scale=0.9,font=\small]
        \draw[thick] (-5,0) -- (-1,0) -- (2,0) -- (5,0);
        \draw[thick] (-1,0) .. controls (-0.5,1.0) and (1.5,1.0) .. (2,0);
        
        \foreach \x/\y in {-5/0,-1/0,2/0,5/0}{
          \fill (\x,\y) circle (2pt);
        }
        
        \fill (-4,0) circle (2pt);
        \fill (-2.4,0) circle (2pt);
        \draw[very thick,dashed]
          (-4,0) .. controls (-3.5,1.3) and (-2.9,1.3) .. (-2.4,0);
        \node[above] at (-3.2,1.15) {(A)};
        
        \fill (3.4,0) circle (2pt);
        \draw[very thick,dashed]
          (3.4,0)
          .. controls (2.6,1.1) and (3.0,2.0) .. (3.8,2.0)
          .. controls (4.6,2.0) and (4.2,1.1) .. (3.4,0);
        \node[above] at (3.6,2.15) {(B)};
        
        \draw[very thick,dashed] (-4,-1.4) circle (0.65);
        \node at (-4,-1.42) {(C)};
        
        \fill (-0.5,0) circle (2pt);
        \fill (1.2,0) circle (2pt);
        
        \draw[line width=1.2pt] (-0.5,0) -- (1.2,0);
        \node[below] at (0.35,-0.05) {$Q$};
        
        \draw[very thick,dashed]
          (-0.5,0) .. controls (0.2,-1.6) and (0.7,-1.6) .. (1.2,0);
        \node[below] at (0.35,-1.45) {(D)};
        
        \filldraw[fill=white] (0.35,-1.25) circle (2.5pt);
        \filldraw[fill=white] (-3.5,-1.8) circle (2.5pt);
        \filldraw[fill=white] (3.35,1.9) circle (2.5pt);

        \node[right] at (5,0) {$\mathcal{H}$};
        
    \end{tikzpicture}
    \caption{The four augmentation steps in the construction of an $(\ell,S)$-frame~$\mathcal{H}$.
    Dashed parts indicate newly added paths or cycles.
    Step~\ref{item:ear1} adds an \((\ell,S)\)-ear, Step~\ref{item:ear2} adds an \((\ell,S)\)-cycle intersecting the current $(\ell,S)$-frame in one vertex, Step~\ref{item:ear3} adds a vertex-disjoint $(\ell,S)$-cycle, and Step~\ref{item:ear4} adds a strict $S$-ear whose base is~$Q$.
    White-filled circles indicates vertices in~$S$.}
\label{fig:frame-construction}
\end{figure}

\begin{enumerate}[label=(\Alph*)]
    \item\label{item:ear1} If~$G_{i,j}$ has an $H_{i,j}$-path which is an $(\ell,S)$-ear of $H^-_{i,j}$, then let~$P_{i,j+1}$ be a shortest such path with ends~$a_{i,j+1}$ and~$b_{i,j+1}$ and let
    \begin{align*}
        \mathcal{H}&:=H_{i,j+1}=H_{i,j}\cup P_{i,j+1},\\
        \mathcal{H}^-&:=H^-_{i,j+1}=H^-_{i,j}\cup P_{i,j+1},\\
        \mathcal{P}(\mathcal{H})&:=\mathcal{P}(H_{i,j})\cup\{(i,j+1)\}.
    \end{align*}
    \item\label{item:ear2} If~\ref{item:ear1} does not hold and~$G_{i,j}$ has an $(\ell,S)$-cycle intersecting~$H_{i,j}$ at exactly one vertex, then let~$P_{i+1,1}$ be a shortest such cycle with intersecting vertex~$c_i$ and let
    \begin{align*}
        \mathcal{H}&:=H_{i+1,1}=H_{i,j}\cup P_{i+1,1},\\
        \mathcal{H}^-&:=H^-_{i+1,1}=H^-_{i,j}\cup P_{i+1,1},\\
        \mathcal{P}(\mathcal{H})&:=\mathcal{P}(H_{i,j})\cup\{(i+1,1)\}.
    \end{align*}
    \item\label{item:ear3} If neither~\ref{item:ear1} nor~\ref{item:ear2} holds and~$G_{i,j}$ has an $(\ell,S)$-cycle vertex-disjoint from~$H_{i,j}$, then let~$P_{i+1,1}$ be a shortest such cycle and let
    \begin{align*}
        \mathcal{H}&:=H_{i+1,1}=H_{i,j}\cup P_{i+1,1},\\
        \mathcal{H}^-&:=H^-_{i+1,1}=H^-_{i,j}\cup P_{i+1,1},\\
        \mathcal{P}(\mathcal{H})&:=\mathcal{P}(H_{i,j})\cup\{(i+1,1)\}.
    \end{align*}
    \item\label{item:ear4} If none of \ref{item:ear1}--\ref{item:ear3} hold and~$G_{i,j}$ has a strict $S$-ear of~$H_{i,j}$ disjoint from $B_{H_{i,j}}(\mathcal{L}_1(H^-_{i,j}),\ell-1)$, then let~$P_{i+1,1}$ be such an ear whose base in~$H_{i,j}$ is closest in~$H_{i,j}$ to the degree-$1$ vertices of $\mathcal{L}_0(H^-_{i,j})$ and let
    \begin{align*}
        \mathcal{H}&:=H_{i+1,1}=H_{i,j}\cup P_{i+1,1},\\
        \mathcal{H}^-&:=H^-_{i+1,1}=H^-_{i,j}\oplus P_{i+1,1},\\
        \mathcal{P}(\mathcal{H})&:=\mathcal{P}(H_{i,j})\cup\{(i+1,1)\}.
    \end{align*}
    We denote by~$Q_{i+1}$ the base of~$P_{i+1,1}$ in~$H_{i,j}$ and by~$a_{i+1,1}$ and~$b_{i+1,1}$ the ends of~$Q_{i+1}$.
\end{enumerate}

The \emph{depth} of~$\mathcal{H}$ is the largest integer~$t$ such that $(t,1)\in\mathcal{P}(\mathcal{H})$; the depth of~$H_{0,0}$ is defined as~$0$.
For each $i\in[t]$, we denote by~$\mu(i)$ the largest integer such that $(i,\mu(i))\in\mathcal{P}(\mathcal{H})$.
We say that~$\mathcal{H}$ is \emph{maximal} if there is neither~$P_{t,\mu(t)+1}$ nor~$P_{t+1,1}$ satisfying one of \ref{item:ear1}--\ref{item:ear4}.
A vertex of~$\mathcal{H}$ is \emph{admissible} if it is not contained in~$Y_{t,\mu(t)}$.
Note that for all $(i,j),(i',j')\in\mathcal{P}(\mathcal{H})$ with $(i,j)\leq_L(i',j')$, we have $Y_{i,j}\subseteq Y_{i',j'}$ and $Z_{i,j}\subseteq Z_{i',j'}$.

We have the following simple observation.

\begin{observation}\label{obs:structure}
    Let~$G$ be a graph.
    For a set $S\subseteq V(G)$, let~$\mathcal{H}$ be an $(\ell,S)$-frame in~$G$.
    For all $(i,j),(i',j')\in\mathcal{P}(\mathcal{H})$ with $(i,j)\leq_L(i',j')$, the following hold.
    \begin{enumerate}[label=(\roman*)]
        \item Every vertex of~$\mathcal{H}$ has degree between~$2$ and~$4$ in~$\mathcal{H}$.
        In particular, every admissible vertex of~$\mathcal{H}$ has degree~$2$ in~$\mathcal{H}$.
        \item\label{item:order} For every branch vertex~$u$ of~$H_{i,j}$, it holds that $B_{H_{i,j}}(u,\ell-1)=B_{H_{i',j'}}(u,\ell-1)$.
    \end{enumerate}
\end{observation}

For each $(i,j)\in\mathcal{P}(\mathcal{H})$, let
\[
    \Gamma_{i,j}:=
    \begin{cases}
        V(P_{i,j})\cap V(H_{i-1,\mu(i-1)}) & \text{if }j=1,\\
        V(P_{i,j})\cap V(H_{i,j-1}) & \text{otherwise}.
    \end{cases}
\]
Note that the vertices of~$\Gamma_{i,1}$ are admissible vertices of~$H_{i-1,\mu(i-1)}$ as~$P_{i,1}$ is disjoint from~$Y_{i-1,\mu(i-1)}$.
Similarly, for $j\geq2$, the vertices of~$\Gamma_{i,j}$ are admissible vertices of~$H_{i,j-1}$.
We remark that~$\Gamma_{i,j}$ is non-empty if and only if ${P_{i,j}-\Gamma_{i,j}}$ is a path.
Moreover, the values of~$j$ and $\abs{\Gamma_{i,j}}$ determine which of \ref{item:ear1}–\ref{item:ear4} is satisfied by $P_{i,j}$.
Thus, the following hold.
\begin{itemize}
    \item $P_{i,j}$ satisfies~\ref{item:ear1} if and only if $j\geq2$, in which case we call~$P_{i,j}$ a \emph{regular $\mathcal{H}$-ear}.
    \item $P_{i,j}$ satisfies~\ref{item:ear2} if and only if $j=1$ and $\abs{\Gamma_{i,j}}=1$, in which case we call~$P_{i,j}$ an \emph{attached $\mathcal{H}$-earring}.
    \item $P_{i,j}$ satisfies~\ref{item:ear3} if and only if $j=1$ and $\abs{\Gamma_{i,j}}=0$, in which case we call~$P_{i,j}$ a \emph{detached $\mathcal{H}$-earring}.
    \item $P_{i,j}$ satisfies~\ref{item:ear4} if and only if $j=1$ and $\abs{\Gamma_{i,j}}=2$, in which case we call~$P_{i,j}$ an \emph{$\mathcal{H}$-detour}.
\end{itemize}
By an \emph{$\mathcal{H}$-earring}, we mean either an attached or a detached $\mathcal{H}$-earring.
For each vertex~$v$ of~$\mathcal{H}$, we denote by~$\sigma_\mathcal{H}(v)$ the lexicographically smallest $(p,q)\in\mathcal{P}(\mathcal{H})$ such that~$v$ is a vertex of~$H_{p,q}$.
Note that $v\in V(P_{p,q})\setminus\Gamma_{p,q}$.
In addition, for every $(i,j)\in\mathcal{P}(\mathcal{H})$ with $v\in V(H_{i,j})$, we have $\sigma_\mathcal{H}(v)=\sigma_{H_{i,j}}(v)$.
We may omit the subscript if it is clear from the context.

In the remainder of this section, we investigate the structure of~$\mathcal{H}$.
We first show that any two nearby admissible vertices of~$\mathcal{H}$ are joined by a unique shortest path.

\begin{lemma}\label{lem:nearby}
    Let~$G$ be a graph.
    For a set $S\subseteq V(G)$, let~$\mathcal{H}$ be an $(\ell,S)$-frame in~$G$.
    If~$\mathcal{H}$ has admissible vertices~$v$ and~$v'$ with $\dist_\mathcal{H}(v,v')\leq2\ell-1$, then every $(v,v')$-path of~$\mathcal{H}$ of length at most $2\ell-1$ is a path of $\mathcal{L}_0(\mathcal{H})$.
    Consequently, such a path is unique and $\sigma(v)=\sigma(v')$.
\end{lemma}
\begin{proof}
    Since~$v$ and~$v'$ are admissible vertices of~$\mathcal{H}$, they are at distance at least~$\ell$ from every branch vertex of~$\mathcal{H}$.
    As $\dist_\mathcal{H}(v,v')\leq2\ell-1$, every $(v,v')$-path of~$\mathcal{H}$ of length at most $2\ell-1$ is a path of $\mathcal{L}_0(\mathcal{H})$ and hence is unique.
    Since~$v$ and~$v'$ are in the same component of $\mathcal{L}_0(\mathcal{H})$, we have $\sigma(v)=\sigma(v')$.
\end{proof}

We provide equivalent conditions for being a branch vertex of~$\mathcal{H}$.

\begin{lemma}\label{lem:degree}
    Let~$G$ be a graph.
    For a set $S\subseteq V(G)$, let~$\mathcal{H}$ be an $(\ell,S)$-frame in~$G$.
    For a vertex~$v$ of~$\mathcal{H}$, the following hold.
    \begin{itemize}
        \item $\deg_\mathcal{H}(v)=3$ if and only if there is a unique $(i,j)\in\mathcal{P}(\mathcal{H})$ such that $v\in\{a_{i,j},b_{i,j}\}$.
        \item $\deg_\mathcal{H}(v)=4$ if and only if there is a unique $i\geq1$ such that $v=c_i$.
    \end{itemize}
\end{lemma}
\begin{proof}
    Let~$t$ be the depth of~$\mathcal{H}$.
    We proceed by induction on $x:=\abs{\mathcal{P}(\mathcal{H})}$ to show that both statements hold.
    The statements obviously hold for~${x\leq1}$.
    Thus, we may assume that~${x\geq2}$.
    Let
    \[
        H':=
        \begin{cases}
            H_{t-1,\mu(t-1)} & \text{if }\mu(t)=1,\\
            H_{t,\mu(t)-1} & \text{otherwise}.
        \end{cases}
    \]

    Suppose that $P_{t,\mu(t)}$ is not an $\mathcal{H}$-earring.
    Note that $\Gamma_{t,\mu(t)}=\{a_{t,\mu(t)},b_{t,\mu(t)}\}$.
    Since~$a_{t,\mu(t)}$ and~$b_{t,\mu(t)}$ are admissible vertices of~$H'$, they have degree~$2$ in~$H'$.
    By the inductive hypothesis, there is no $(i,j)\in\mathcal{P}(H')$ with $\Gamma_{i,j}\cap\Gamma_{t,\mu(t)}\neq\emptyset$.
    Since both~$a_{t,\mu(t)}$ and~$b_{t,\mu(t)}$ have degree~$3$ in~$\mathcal{H}$ and every other vertex of~$P_{t,\mu(t)}$ has degree~$2$ in~$\mathcal{H}$, we have $V_{\geq3}(\mathcal{H})\setminus V_{\geq3}(H')=\Gamma_{t,\mu(t)}$.
    Hence, the statements hold.

    Now, suppose that $P_{t,\mu(t)}$ is an $\mathcal{H}$-earring.
    Note that $\mu(t)=1$.
    If~$P_{t,1}$ is a detached $\mathcal{H}$-earring, then the statements easily follow from the inductive hypothesis as~$P_{t,1}$ and~$H'$ are vertex-disjoint.
    Thus, we may assume that~$P_{t,1}$ is an attached $\mathcal{H}$-earring.
    Since~$c_t$ is an admissible vertex of~$H'$, it has degree~$2$ in~$H'$.
    By the inductive hypothesis, there is no $(i,j)\in\mathcal{P}(H')$ with $c_t\in\Gamma_{i,j}$.
    Since~$c_t$ has degree~$4$ in~$\mathcal{H}$ and every other vertex of~$P_{t,1}$ has degree~$2$ in~$\mathcal{H}$, we have $V_{\geq3}(\mathcal{H})\setminus V_{\geq3}(H')=\{c_t\}$.
    Hence, the statements hold.
    
    This completes the proof by induction.
\end{proof}

As a corollary of \cref{lem:degree}, we show that if two branch vertices of~$\mathcal{H}$ are close to each other in~$\mathcal{H}$, then they are precisely~$a_{i,j}$ and~$b_{i,j}$ for some $(i,j)\in\mathcal{P}(\mathcal{H})$.

\begin{corollary}\label{cor:short branch}
    Let~$G$ be a graph.
    For a set $S\subseteq V(G)$, let~$\mathcal{H}$ be an $(\ell,S)$-frame in~$G$.
    If~$\mathcal{H}$ has distinct branch vertices~$v_1$ and~$v_2$ with $\dist_\mathcal{H}(v_1,v_2)\leq\ell-1$, then there is a unique $(p,q)\in\mathcal{P}(\mathcal{H})$ with $\{v_1,v_2\}=\{a_{p,q},b_{p,q}\}$.
    Consequently, both~$v_1$ and~$v_2$ have degree~$3$ in~$\mathcal{H}$.
\end{corollary}
\begin{proof}
    For each $i\in[2]$, by \cref{lem:degree}, there is a unique $(p_i,q_i)\in\mathcal{P}(\mathcal{H})$ with $v_i\in\Gamma_{p_i,q_i}$.
    Without loss of generality, we may assume that $(p_1,q_1)\leq_L(p_2,q_2)$.
    As
    \[
        v_2\in B_\mathcal{H}(v_1,\ell-1)=B_{H_{p_1,q_1}}(v_1,\ell-1)\subseteq Y_{p_1,q_1},
    \]
    we have $(p_1,q_1)=(p_2,q_2)$.
    Since~$v_1$ and~$v_2$ are distinct, we have $\{v_1,v_2\}=\{a_{p_1,q_1},b_{p_1,q_1}\}$.
    By \cref{lem:degree}, both~$v_1$ and~$v_2$ have degree~$3$ in~$\mathcal{H}$.
\end{proof}

The following lemma provides properties of $\mathcal{H}$-detours.
Especially, it shows that all $\mathcal{H}$-detours and their bases are short.

\begin{lemma}\label{lem:strict1}
    Let~$G$ be a graph.
    For a set $S\subseteq V(G)$, let~$\mathcal{H}$ be an $(\ell,S)$-frame in~$G$ with depth~$t$.
    For each $r\in[t]$, if~$P_{r,1}$ is an $\mathcal{H}$-detour, then $P_{r,1}\cup Q_r$ has length at most $\ell-1$ and its vertex set is a subset of $Y_{r,1}$.
    Consequently, every vertex in $V(P_{r,1}\cup Q_r)\setminus\Gamma_{r,1}$ has degree~$2$ in~$\mathcal{H}$, and every $\mathcal{H}$-path in $G_{t,\mu(t)}$ is an $\mathcal{H}^-$-path.
\end{lemma}
\begin{proof}
    Since~$P_{r,1}$ is a strict $S$-ear, it contains a vertex in~$S$.
    By~\ref{item:ear4}, $P_{r,1}$ is not an $(\ell,S)$-ear of $H^-_{r-1,\mu(r-1)}$.
    Thus,
    \[
        \abs{E(P_{r,1})}+\dist_{H^-_{r-1,\mu(r-1)}}(a_{r,1},b_{r,1})\leq\ell-1.
    \]
    Since~$a_{r,1}$ and~$b_{r,1}$ are admissible vertices of $H_{r-1,\mu(r-1)}$, by \cref{lem:nearby} for $\mathcal{H}:=H_{r-1,\mu(r-1)}$, $Q_r$ is the unique shortest $(a_{r,1},b_{r,1})$-path of $H_{r-1,\mu(r-1)}$.
    Thus, $P_{r,1}\cup Q_r$ has length at most $\ell-1$ and hence $V(P_{r,1}\cup Q_r)\subseteq B_\mathcal{H}(\Gamma_{r,1},\ell-1)\subseteq Y_{r,1}$.

    Note that every vertex $v\in V(P_{r,1}\cup Q_r)\setminus\Gamma_{r,1}$ has degree~$2$ in~$H_{r,1}$.
    As $v\in Y_{r,1}$, $v$ also has degree~$2$ in~$\mathcal{H}$.
    Since~$G_{t,\mu(t)}$ is disjoint from~$Y_{t,\mu(t)}$, every $\mathcal{H}$-path in~$G_{t,\mu(t)}$ is vertex-disjoint from the base of each $\mathcal{H}$-detour and hence is an $\mathcal{H}^-$-path.
\end{proof}

The following lemma shows that every cycle of~$\mathcal{H}^-$ is an $(\ell,S)$-cycle.

\begin{lemma}\label{lem:all cycles}
    Let~$G$ be a graph.
    For a set $S\subseteq V(G)$, let~$\mathcal{H}$ be an $(\ell,S)$-frame in~$G$.
    Then every cycle of~$\mathcal{H}^-$ is an $(\ell,S)$-cycle.
\end{lemma}
\begin{proof}
    We proceed by induction on $x:=\abs{\mathcal{P}(\mathcal{H})}$.
    The statement obviously holds for $x\leq1$.
    Thus, we may assume that $x\geq2$.
    Let~$t$ be the depth of~$\mathcal{H}$ and let~$C$ be a cycle of~$\mathcal{H}^-$.
    By the inductive hypothesis, we may assume that~$C$ contains an edge of~$P_{t,\mu(t)}$.
    If $P_{t,\mu(t)}$ is a regular $\mathcal{H}$-ear, then~$C$ is an $(\ell,S)$-cycle by the definition of an $(\ell,S)$-ear.
    If $P_{t,\mu(t)}$ is an $\mathcal{H}$-earring, then~$C$ is equal to~$P_{t,\mu(t)}$ which is an $(\ell,S)$-cycle.
    Thus, we may assume that~$P_{t,\mu(t)}$ is an $\mathcal{H}$-detour.
    Note that $\mu(t)=1$.
    By \cref{lem:strict1}, every vertex of $P_{t,1}$ has degree~$2$ in~$\mathcal{H}^-$, and therefore~$P_{t,1}$ is a subpath of~$C$.
    Since~$P_{t,1}$ is a strict $S$-ear, $C$ contains a vertex in~$S$.
    Since~$P_{t,1}$ is disjoint from $Y_{t-1,\mu(t-1)}$, $\Gamma_{t,1}$ is at distance at least~$\ell$ from every branch vertex of~$H_{t-1,\mu(t-1)}$.
    Therefore, $C$ is an $(\ell,S)$-cycle.

    This completes the proof by induction.
\end{proof}

For a cycle~$C$ of~$\mathcal{H}$, we denote by~$C^-$ the graph obtained from~$C$ as follows: for every $\mathcal{H}$-detour~$P_{r,1}$ whose base is a path of~$C$, replace~$Q_r$ with~$P_{r,1}$.
By \cref{lem:strict1}, $C^-$ is a cycle of~$\mathcal{H}^-$ if and only if $C\neq P_{r,1}\cup Q_r$ for every $\mathcal{H}$-detour~$P_{r,1}$.

\subsection{\texorpdfstring{Short $\mathcal{H}$-paths}{Short H-paths}}\label{subsec:short paths}

In this subsection, we provide properties of $\mathcal{H}$-paths of length at most $\ell-1$.
We first show that if~$P_{i,1}$ for some $i\geq1$ is a detached $\mathcal{H}$-earring, then every $(\ell,S)$-ear of~$H^-_{i,1}$ in~$G_{i,1}$ of length at most $\ell-1$ has exactly one end in~$P_{i,1}$, unless~$P_{i,1}$ has length at most $2(\ell-1)$.

\begin{lemma}\label{lem:bridge}
    Let~$G$ be a graph.
    For a set $S\subseteq V(G)$, let~$\mathcal{H}$ be an $(\ell,S)$-frame in~$G$.
    For some $i\geq1$, if $P_{i,1}$ is a detached $\mathcal{H}$-earring of length at least $2\ell-1$, then every $(\ell,S)$-ear of~$H^-_{i,1}$ in~$G_{i,1}$ of length at most $\ell-1$ has exactly one end in~$P_{i,1}$.
\end{lemma}
\begin{proof}
    Let~$P$ be an $(\ell,S)$-ear of~$H^-_{i,1}$ in~$G_{i,1}$ of length at most $\ell-1$ and let~$a$ and~$b$ be the ends of~$P$.
    Since~$P_{i,1}$ is vertex-disjoint from $H_{i-1,\mu(i-1)}$, if neither~$a$ nor~$b$ is in~$P_{i,1}$, then~$P$ is an $(\ell,S)$-ear of $H^-_{i-1,\mu(i-1)}$, contradicting~\ref{item:ear3}.
    Thus, at least one of~$a$ and~$b$ is in~$P_{i,1}$.

    Towards a contradiction, suppose that both~$a$ and~$b$ are in~$P_{i,1}$.
    Let~$P_1$ and~$P_2$ be the two $(a,b)$-paths of~$P_{i,1}$.
    Without loss of generality, we may assume that $\abs{E(P_1)}\leq\abs{E(P_2)}$.
    As $\abs{E(P_{i,1})}\geq2\ell-1$, we have $\abs{E(P_2)}\geq\ell$.
    Let~$C$ be the cycle obtained by concatenating~$P_1$ and~$P$.
    Note that~$C$ is a cycle of $G_{i-1,\mu(i-1)}$ vertex-disjoint from $H_{i-1,\mu(i-1)}$ as~$P$ is an $H_{i,1}$-path in~$G_{i,1}$.
    Since~$C$ is a cycle of $P_{i,1}\cup P$, it is an $(\ell,S)$-cycle by the definition of an $(\ell,S)$-ear.
    As $\abs{E(P)}\leq\ell-1<\abs{E(P_2)}$, $C$ is an $(\ell,S)$-cycle shorter than~$P_{i,1}$, contradicting~\ref{item:ear3}.
\end{proof}

In the following two lemmas, we provide properties of $\mathcal{H}$-paths of length at most $\ell-1$ in general cases.
For the ends~$a$ and~$b$ of such a path, the first lemma deals with the case $\sigma(a)=\sigma(b)$ and the second lemma deals with the case $\sigma(a)\neq\sigma(b)$.
We need the following definition.

For $(i,j)\in\mathcal{P}(\mathcal{H})$, we say that~$P_{i,j}$ is \emph{fragile} if it is either
\begin{itemize}
    \item an attached $\mathcal{H}$-earring with $c_i\notin S$, or
    \item a regular $\mathcal{H}$-ear such that~$a_{i,j}$ and~$b_{i,j}$ are in the same component of $H^-_{i,j-1}-S$.
\end{itemize}
By~\ref{item:ear1} or~\ref{item:ear2}, every fragile~$P_{i,j}$ contains a vertex in~$S$.
For a fragile~$P_{i,j}$, we denote by~$P^*_{i,j}$ the shortest path of $P_{i,j}-\Gamma_{i,j}$ containing all vertices in $V(P_{i,j})\cap S$, and by~$a^*_{i,j}$ and~$b^*_{i,j}$ the ends of~$P^*_{i,j}$.
Let $P^{\operatorname{side}}_{i,j}:=P_{i,j}-(V(P^*_{i,j})\cup\Gamma_{i,j})$.
See \cref{fig:fragile-ear} for an illustration of \cref{lem:short path1}\ref{item:short path1-2}.

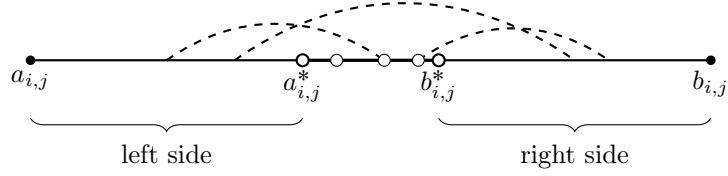
\begin{figure}[t]
    \centering
    \begin{tikzpicture}[scale=0.9, every node/.style={font=\small}]
        \draw[thick] (0,0) -- (10,0);
        
        \fill (0,0) circle (2pt);
        \fill (10,0) circle (2pt);
        
        \node[below] at (0,0) {$a_{i,j}$};
        \node[below] at (10,0) {$b_{i,j}$};

        \draw[very thick] (4,0) -- (6,0);
        
        \filldraw[fill=white, thick] (4,0) circle (2.5pt);
        \filldraw[fill=white, thick] (6,0) circle (2.5pt);

        \node[below] at (4,0) {$a^*_{i,j}$};
        \node[below] at (6,0) {$b^*_{i,j}$};
        
        \draw[decorate,decoration={brace,mirror,amplitude=5pt}]
          (0,-0.85)--(4,-0.85)
          node[midway,below=6pt]{left side};
        
        \draw[decorate,decoration={brace,mirror,amplitude=5pt}]
          (6,-0.85)--(10,-0.85)
          node[midway,below=6pt]{right side};
        
        \draw[dashed,thick,bend left=35] (2,0) to (5.2,0);
        \draw[dashed,thick,bend left=35] (3,0) to (8,0);
        \draw[dashed,thick,bend left=35] (5.7,0) to (8.5,0);

        \foreach \x in {4.5,5.2,5.7}{
          \filldraw[fill=white] (\x,0) circle (2.5pt);
        }
    \end{tikzpicture}
    \caption{A fragile regular $\mathcal{H}$-ear $P_{i,j}$.
    The subpath $P^*_{i,j}$ is the subpath between
    $a^*_{i,j}$ and $b^*_{i,j}$.
    White-filled circles indicates vertices in~$S$.
    Typical chords connect the two sides of the ear or one side to~$P^*_{i,j}$.
    The graph induced by the internal vertices of both sides is~$P^{\operatorname{side}}_{i,j}$.
    For a fragile $\mathcal{H}$-earring, the figure can be interpreted by identifying~$a_{i,j}$ and~$b_{i,j}$ with~$c_i$.}
    \label{fig:fragile-ear}
\end{figure}

\begin{lemma}\label{lem:short path1}
    Let~$G$ be a graph.
    For a set $S\subseteq V(G)$, let~$\mathcal{H}$ be an $(\ell,S)$-frame in~$G$.
    Let~$a$ and~$b$ be vertices of~$\mathcal{H}$ with $\sigma(a)=\sigma(b)=:(p,q)$.
    If~$G_{p,q}$ has an $H_{p,q}$-path~$Q$ of length at most $\ell-1$ between~$a$ and~$b$, then one of the following holds.
    \begin{enumerate}[label=(\roman*)]
        \item\label{item:short path1-1} $\dist_{P_{p,q}}(a,b)\leq\ell-1$.
        \item\label{item:short path1-2} $P_{p,q}$ is fragile, $a$ and~$b$ are in distinct components of~$P^{\operatorname{side}}_{p,q}$, and~$Q$ is not an $S$-ear of~$H^-_{p,q}$.
    \end{enumerate}
\end{lemma}
\begin{proof}
    Suppose that $\dist_{P_{p,q}}(a,b)\geq\ell$.
    Note that~$P_{p,q}$ is not an $\mathcal{H}$-detour by \cref{lem:strict1}.
    We show that~\ref{item:short path1-2} holds.
    Since~$Q$ is disjoint from~$Y_{p,q}$, $a$ and~$b$ are admissible vertices of~$H_{p,q}$.
    We divide into three cases according to the type of~$P_{p,q}$.
    
    \medskip
    \noindent
    \textbf{Case 1.} $P_{p,q}$ is a regular $\mathcal{H}$-ear.
    
    Note that $q\geq2$.
    Let~$P$ be the path obtained from~$P_{p,q}$ by replacing $aP_{p,q}b$ with~$Q$.
    Note that~$P$ is an $H_{p,q-1}$-path in~$G_{p,q-1}$ as~$Q$ is an $H_{p,q}$-path in~$G_{p,q}$.
    As $\abs{E(Q)}\leq\ell-1<\dist_{P_{p,q}}(a,b)$, we have
    \[
        \abs{E(P_{p,q})}>\abs{E(P)}\geq\abs{E(Q)}+2\cdot\dist_{P_{p,q}}(\{a,b\},\Gamma_{p,q})\geq2\ell+1.
    \]
    By~\ref{item:ear1}, $P$ is not an $(\ell,S)$-ear of~$H^-_{p,q-1}$, so~$P$ is disjoint from~$S$, and~$a_{p,q}$ and~$b_{p,q}$ are in the same component of $H^-_{p,q-1}-S$.
    This implies that
    \begin{itemize}
        \item $P_{p,q}$ is fragile,
        \item $P^*_{p,q}$ is a subpath of $aP_{p,q}b-\{a,b\}$,
        \item $Q$ is disjoint from~$S$, and
        \item $a$ and~$b$ are in the same component of $H^-_{p,q}-S$.
    \end{itemize}
    Hence, \ref{item:short path1-2} holds.
    
    \medskip
    \noindent
    \textbf{Case 2.} $P_{p,q}$ is an attached $\mathcal{H}$-earring.

    Note that $q=1$ and $\Gamma_{p,q}=\{c_p\}$.
    Let $P'_{p,1}:=P_{p,1}-c_p$ and let~$P$ be the cycle obtained from~$P_{p,1}$ by replacing $aP'_{p,1}b$ with~$Q$.
    Note that~$P$ is a cycle of~$G_{p-1,\mu(p-1)}$ intersecting $H_{p-1,\mu(p-1)}$ only at~$c_p$ as~$Q$ is an $H_{p,1}$-path in~$G_{p,1}$.
    As $\abs{E(Q)}\leq\ell-1<\dist_{P_{p,1}}(a,b)$, we have
    \[
        \abs{E(P_{p,1})}>\abs{E(P)}\geq\abs{E(Q)}+2\cdot\dist_{P_{p,1}}(\{a,b\},c_p)\geq2\ell+1.
    \]
    By~\ref{item:ear2}, $P$ is not an $(\ell,S)$-cycle, so~$P$ is disjoint from~$S$.
    This implies that
    \begin{itemize}
        \item $P_{p,1}$ is fragile,
        \item $P^*_{p,1}$ is a subpath of $aP'_{p,1}b-\{a,b\}$,
        \item $Q$ is disjoint from~$S$, and
        \item $a$ and~$b$ are in the same component of $H^-_{p,1}-S$.
    \end{itemize}
    Hence, \ref{item:short path1-2} holds.

    \medskip
    \noindent
    \textbf{Case 3.} $P_{p,q}$ is a detached $\mathcal{H}$-earring.

    Note that $q=1$ and $\Gamma_{p,q}=\emptyset$.
    We will derive a contradiction.
    Since~$P_{p,1}$ is an $(\ell,S)$-cycle, it has a subpath between~$a$ and~$b$ containing a vertex in~$S$.
    Let~$P$ be the cycle obtained by concatenating this subpath with~$Q$.
    Note that~$P$ is a cycle of~$G_{p-1,\mu(p-1)}$ vertex-disjoint from $H_{p-1,\mu(p-1)}$ as~$Q$ is an $H_{p,1}$-path in $G_{p,1}$.
    As $\abs{E(Q)}\leq\ell-1<\dist_{P_{p,1}}(a,b)$, we have
    \[
        \abs{E(P_{p,1})}>\abs{E(P)}\geq\abs{E(Q)}+\dist_{P_{p,1}}(a,b)\geq\ell+1.
    \]
    Therefore, $P$ is an $(\ell,S)$-cycle shorter than~$P_{p,1}$, contradicting~\ref{item:ear3}.

    \medskip
    This completes the proof.
\end{proof}

\begin{lemma}\label{lem:short path2}
    Let~$G$ be a graph.
    For a set $S\subseteq V(G)$, let~$\mathcal{H}$ be an $(\ell,S)$-frame in~$G$ with depth~$t$.
    Let~$a$ and~$b$ be vertices of~$\mathcal{H}$ with $\sigma(a)<_L\sigma(b)=:(r,s)$ and let
    \[
        r':=
        \begin{cases}
            \min\{\mu(r),2\} & \text{if }s=1,\\
            s & \text{otherwise}.
        \end{cases}
    \]
    If~$G_{r,r'}$ has an $H_{r,r'}$-path~$Q$ of length at most $\ell-1$ between~$a$ and~$b$, then~$P_{r,s}$ is a detached $\mathcal{H}$-earring such that either
    \begin{enumerate}[label=(\roman*)]
        \item\label{item:short path2-1} $P_{r,s}$ has length at most $2(\ell-1)$, or
        \item\label{item:short path2-2} $(r,s)=(t,1)=(t,\mu(t))$ and~$Q$ is an $(\ell,S)$-ear of~$H^-_{t,1}$.
    \end{enumerate}
\end{lemma}
\begin{proof}
    Since~$Q$ is disjoint from~$Y_{r,r'}$, by \cref{lem:strict1}, $P_{r,1}$ is not an $\mathcal{H}$-detour.
    In addition, $a$ and~$b$ are admissible vertices of~$H_{r,r'}$.
    We divide into three cases according to the type of~$P_{r,s}$.

    \medskip
    \noindent
    \textbf{Case 1.} $P_{r,s}$ is a regular $\mathcal{H}$-ear.
    
    Note that $s\geq2$.
    We will derive a contradiction.
    Since~$P_{r,s}$ is an $(\ell,S)$-ear of~$H^-_{r,s-1}$, without loss of generality, we may assume the following:
    \begin{itemize}
        \item if~$P_{r,s}$ contains a vertex in~$S$, then $bP_{r,s}b_{r,s}$ contains the vertex, and
        \item otherwise either $a\in S$ or~$a$ and~$b_{r,s}$ are in distinct components of $H^-_{r,s-1}-S$.
    \end{itemize}
    Let~$P$ be the path obtained by concatenating $bP_{r,s}b_{r,s}$ with~$Q$.
    Note that~$P$ is an $H_{r,s-1}$-path in~$G_{r,s-1}$ as~$Q$ is an $H_{r,s}$-path in~$G_{r,s}$.
    By the assumption, $P$ is an $S$-ear of~$H^-_{r,s-1}$.
    Since~$b$ is an admissible vertex of~$H_{r,r'}=H_{r,s}$, we have $\dist_{P_{r,s}}(b,\Gamma_{r,s})\geq\ell>\abs{E(Q)}$.
    Thus, we have
    \[
        \abs{E(P_{r,s})}>\abs{E(P)}\geq\abs{E(Q)}+\dist_{P_{r,s}}(b,b_{r,s})\geq\ell+1.
    \]
    Therefore, $P$ is an $(\ell,S)$-ear of~$H^-_{r,s-1}$ shorter than~$P_{r,s}$, contradicting~\ref{item:ear1}.

    \medskip
    \noindent
    \textbf{Case 2.} $P_{r,s}$ is an attached $\mathcal{H}$-earring.
    
    Note that $s=1$ and $\Gamma_{r,s}=\{c_r\}$.
    We will derive a contradiction.
    Since~$P_{r,1}$ is an $(\ell,S)$-cycle, it has a subpath between~$b$ and~$c_r$ containing a vertex in~$S$.
    Let~$P$ be the path obtained by concatenating this subpath with~$Q$.
    Note that~$P$ is an $H_{r-1,\mu(r-1)}$-path in~$G_{r-1,\mu(r-1)}$ as~$Q$ is an $H_{r,1}$-path in~$G_{r,1}$.
    Since~$b$ is an admissible vertex of~$H_{r,1}$, we have $\dist_{P_{r,1}}(b,c_r)\geq\ell$.
    Thus, we have
    \[
        \abs{E(P)}\geq\abs{E(Q)}+\dist_{P_{r,1}}(b,c_r)\geq\ell+1.
    \]
    Therefore, $P$ is an $(\ell,S)$-ear of $H^-_{r-1,\mu(r-1)}$, contradicting~\ref{item:ear2}.

    \medskip
    \noindent
    \textbf{Case 3.} $P_{r,s}$ is a detached $\mathcal{H}$-earring.
    
    Note that $s=1$ and $\Gamma_{r,s}=\emptyset$.
    We may assume that~$P_{r,1}$ has length at least $2\ell-1$, because otherwise~\ref{item:short path2-1} holds.
    We show that~\ref{item:short path2-2} holds.
    Since~$a$ and~$b$ are in distinct components of~$H_{r,1}$, $Q$ is an $(\ell,S)$-ear of~$H^-_{r,1}$.
    Thus, to show that~\ref{item:short path2-2} holds, it suffices to show that $r=t$ and $\mu(t)=1$.
    
    Towards a contradiction, suppose that $r<t$ or $\mu(t)>1$.
    Note that $\mu(r)\geq2$; if $r=t$, then it holds by the assumption, and otherwise it holds by the fact that~$Q$ is an $(\ell,S)$-ear of~$H^-_{r,1}$.
    By~\ref{item:ear1}, $\abs{E(P_{r,2})}\leq\abs{E(Q)}\leq\ell-1$.
    By \cref{lem:bridge}, $P_{r,2}$ has exactly one end, say~$b_{r,2}$, in~$P_{r,1}$.
    Since~$a$ and~$b$ are admissible vertices of $H_{r,r'}=H_{r,2}$, we have $\dist_{H_{r,2}}(\{a,b\},\Gamma_{r,2})\geq\ell$.
    In particular, $a\neq a_{r,2}$ and $b\neq b_{r,2}$.
    Since~$P_{r,1}$ is an $(\ell,S)$-cycle, it has a subpath between~$b$ and~$b_{r,2}$ containing a vertex in~$S$.
    Let~$P$ be the path obtained by concatenating this subpath, $P_{r,2}$, and~$Q$.
    Note that~$P$ is an $H_{r-1,\mu(r-1)}$-path in~$G_{r-1,\mu(r-1)}$ as~$Q$ is an $H_{r,2}$-path in $G_{r,2}$.
    As $\dist_{H_{r,2}}(\{a,b\},\Gamma_{r,2})\geq\ell$, we have
    \[
        \abs{E(P)}\geq\dist_{P_{r,1}}(b,b_{r,2})+\abs{E(P_{r,2})}+\abs{E(Q)}\geq\ell+2.
    \]
    Therefore, $P$ is an $(\ell,S)$-ear of~$H^-_{r-1,\mu(r-1)}$, contradicting~\ref{item:ear3}.
    Hence, \ref{item:short path2-2} holds.
    
    \medskip
    This completes the proof.
\end{proof}

As a corollary of \cref{lem:short path1,lem:short path2}, we characterise all short paths~$P_{i,j}$ of~$\mathcal{H}$ under an assumption that $(G,S)$ is a good pair, that is, $G$ has no $(\ell,S)$-cycle of length at most $3(\ell-1)$.

\begin{corollary}\label{cor:chord}
    Let $(G,S)$ be a good pair and let~$\mathcal{H}$ be an $(\ell,S)$-frame in~$G$.
    For some $(i,j)\in\mathcal{P}(\mathcal{H})$, if $P_{i,j}$ is a regular $\mathcal{H}$-ear of length at most $\ell-1$, then one of the following holds.
    \begin{enumerate}[label=(\roman*)]
        \item\label{item:edge1} $\sigma(a_{i,j})=\sigma(b_{i,j})=:(p,q)$, $P_{p,q}$ is fragile, and~$a_{i,j}$ and~$b_{i,j}$ are in distinct components of~$P^{\operatorname{side}}_{p,q}$.
        \item\label{item:edge2} $j=2$, $P_{i,1}$ is a detached $\mathcal{H}$-earring, and exactly one of~$a_{i,2}$ and~$b_{i,2}$ is in~$P_{i,1}$.
    \end{enumerate}
\end{corollary}

\begin{proof}
    Without loss of generality, we may assume that $(p,q):=\sigma(a_{i,j})\leq_L\sigma(b_{i,j})=:(r,s)$.
    Suppose that $\sigma(a_{i,j})=\sigma(b_{i,j})$.
    We apply \cref{lem:short path1} for $\mathcal{H}:=H_{i,j-1}$ and $Q:=P_{i,j}$.
    We may assume that \cref{lem:short path1}\ref{item:short path1-1} holds, because otherwise~\ref{item:edge1} holds.
    Then $\dist_{P_{p,q}}(a_{i,j},b_{i,j})\leq\ell-1$.
    Since~$a_{i,j}$ and~$b_{i,j}$ are admissible vertices of~$H_{i,j-1}$, every vertex in $B_{H_{i,j-1}}(\Gamma_{i,j},\ell-1)$ has degree~$2$ in~$\mathcal{H}$.
    Hence, a shortest $(a_{i,j},b_{i,j})$-path of~$P_{p,q}$ is a path of~$H^-_{i,j-1}$.
    Then by concatenating this shortest path with~$P_{i,j}$, we obtain a cycle of~$H^-_{i,j}$ of length at most $2(\ell-1)$.
    By \cref{lem:all cycles}, the cycle is an $(\ell,S)$-cycle, contradicting that~$G$ has no $(\ell,S)$-cycle of length at most $2(\ell-1)$.

    Now, suppose that $\sigma(a_{i,j})\neq\sigma(b_{i,j})$.
    Let
    \[
        r':=
        \begin{cases}
            1 & \text{if }(i,j)=(r,2),\\
            \min\{\mu(r),2\} & \text{if }(i,j)\neq(r,2)\text{ and }s=1,\\
            s & \text{otherwise}.
        \end{cases}
    \]
    Note that~$P_{i,j}$ is an $H_{r,r'}$-path in~$G_{r,r'}$ as it is an $H_{i,j-1}$-path in~$G_{i,j-1}$.
    We apply \cref{lem:short path2} for $\mathcal{H}:=H_{i,j-1}$ and $Q:=P_{i,j}$.
    Since $(G,S)$ is a good pair, \cref{lem:short path2}\ref{item:short path2-1} does not hold.
    Thus, \cref{lem:short path2}\ref{item:short path2-2} holds.
    Then $P_{i,1}$ is a detached $\mathcal{H}$-earring, $(r,s)=(i,1)=(i,j-1)$, and~$P_{i,j}$ is an $(\ell,S)$-ear of $H^-_{i,1}$.
    By \cref{lem:bridge}, exactly one of~$a_{i,2}$ and~$b_{i,2}$ is in~$P_{i,1}$, so \ref{item:edge2} holds.
\end{proof}

We now argue the distance in~$\mathcal{H}$ between vertex-disjoint~$P_{i,j}$ and~$P_{i',j'}$.

\begin{lemma}\label{lem:argue distance}
    Let $(G,S)$ be a good pair and let~$\mathcal{H}$ be an $(\ell,S)$-frame in~$G$.
    Let $(i,j)$ and $(i',j')$ be pairs in $\mathcal{P}(\mathcal{H})$ with $(i,j)>_L(i',j')$ such that~$P_{i,j}$ and~$P_{i',j'}$ are vertex-disjoint.
    If~$\mathcal{H}$ has a $(P_{i,j},P_{i',j'})$-path~$Q$ of length at most $\ell-1$, then $P_{i,j}$ is a detached $\mathcal{H}$-earring, $Q=P_{i,2}$, and~$P_{i,2}$ has exactly one end in~$P_{i,1}$.
\end{lemma}
\begin{proof}
    Note that both ends of~$Q$ are branch vertices of~$\mathcal{H}$.
    Since~$Q$ has length at most $\ell-1$, every internal vertex of~$Q$ has degree~$2$ in~$\mathcal{H}$.
    By \cref{cor:short branch}, there is a unique $(p,q)\in\mathcal{P}(\mathcal{H})$ such that~$Q$ is an $(a_{p,q},b_{p,q})$-path and both~$a_{p,q}$ and~$b_{p,q}$ have degree~$3$ in~$\mathcal{H}$.
    Without loss of generality, we may assume that $a_{p,q}\in V(P_{i,j})$ and $b_{p,q}\in V(P_{i',j'})$.
    By \cref{lem:degree}, $\sigma(a_{p,q})$ and $(p,q)$ are the only pairs $(x,y)\in\mathcal{P}(\mathcal{H})$ with $a_{p,q}\in V(P_{x,y})$.
    Similarly, $\sigma(b_{p,q})$ and $(p,q)$ are the only pairs $(x,y)\in\mathcal{P}(\mathcal{H})$ with $b_{p,q}\in V(P_{x,y})$.
    Since~$P_{i,j}$ and~$P_{i',j'}$ are vertex-disjoint, we have $\sigma(a_{p,q})=(i,j)$ and $\sigma(b_{p,q})=(i',j')$.
    This implies that~$P_{p,q}$ is not an $\mathcal{H}$-detour.
    Thus, $P_{p,q}$ is a regular $\mathcal{H}$-ear.
    
    We first show that $Q=P_{p,q}$.
    Suppose not.
    As $B_{H_{p,q}}(\Gamma_{p,q},\ell-1)=B_\mathcal{H}(\Gamma_{p,q},\ell-1)$, $Q$ is a path of~$H_{p,q}$.
    As $Q\neq P_{p,q}$, $Q$ is a path of~$H_{p,q-1}$.
    Since~$a_{p,q}$ and~$b_{p,q}$ are admissible vertices of~$H_{p,q-1}$, by \cref{lem:nearby} for $\mathcal{H}:=H_{p,q-1}$, $Q$ is a path of $\mathcal{L}_0(H_{p,q-1})$.
    Thus, $P_{i,j}$ and~$P_{i',j'}$ intersect, a contradiction.

    We now show that $(p,q)=(i,2)$.
    Suppose not.
    Let
    \[
        i':=
        \begin{cases}
            1 & \text{if }(p,q)=(i,2),\\
            \min\{\mu(i),2\} & \text{if }(p,q)\neq(i,2)\text{ and }j=1,\\
            j & \text{otherwise}.
        \end{cases}
    \]
    Note that~$P_{p,q}$ is an $H_{i,i'}$-path in~$G_{i,i'}$ as it is an $H_{p,q-1}$-path in~$G_{p,q-1}$.
    We apply \cref{lem:short path2} for $\mathcal{H}:=H_{i,i'}$ and $Q:=P_{p,q}$.
    Since $(G,S)$ is a good pair, \cref{lem:short path2}\ref{item:short path2-1} does not hold.
    Thus, \cref{lem:short path2}\ref{item:short path2-2} holds.
    Then~$P_{i,1}$ is a detached $\mathcal{H}$-earring, $(i,j)=(i,1)$, and~$P_{p,q}$ is an $(\ell,S)$-ear of $H^-_{i,1}$.
    As $(i,j)<_L(p,q)$, we have $\mu(i)\geq2$.

    Since~$P_{p,q}$ has length at most $\ell-1$, by~\ref{item:ear1}, $P_{i,2}$ has length at most $\ell-1$.
    By \cref{lem:bridge}, $P_{i,2}$ has exactly one end, say~$b_{i,2}$, in~$P_{i,1}$.
    Since~$a_{p,q}$ and~$b_{p,q}$ are admissible vertices of~$H_{p,q-1}$, we have $\dist_{H_{p,q-1}}(\Gamma_{i,2},\Gamma_{p,q})\geq\ell$.
    Since~$P_{i,1}$ is an $(\ell,S)$-cycle, it has a subpath between~$b_{i,2}$ and~$b_{p,q}$ containing a vertex in~$S$.
    Let~$P$ be the path obtained by concatenating this subpath, $P_{i,2}$, and~$P_{p,q}$.
    Note that~$P$ is an $H_{i-1,\mu(i-1)}$-path in $G_{i-1,\mu(i-1)}$.
    Since $H^-_{i-1,\mu(i-1)}$ is a subgraph of~$H_{p,q-1}$, we have
    \[
        \dist_{H^-_{i-1,\mu(i-1)}}(a_{i,2},a_{p,q})\geq\dist_{H_{p,q-1}}(\Gamma_{i,2},\Gamma_{p,q})\geq\ell.
    \]
    Therefore, $P$ is an $(\ell,S)$-ear of $H^-_{i-1,\mu(i-1)}$, contradicting~\ref{item:ear3}.
    Hence, $(p,q)=(i,2)$.

    By \cref{lem:bridge}, $P_{i,2}$ has exactly one end in~$P_{i,1}$.
\end{proof}

\section{Characterisation of the chords}\label{sec:chords}

Let $(G,S)$ be a good pair and let~$\mathcal{H}$ be an $(\ell,S)$-frame in~$G$.
In this section, we characterise all chords of $\mathcal{H}$.
We first show that if~$P_{i,1}$ for some $i\geq1$ is a detached $\mathcal{H}$-earring, then it has a subpath~$Q_i$ such that every chord of~$P_{i,1}$ is close to the ends of~$Q_i$.

\begin{lemma}\label{lem:cycle chords}
    Let $(G,S)$ be a good pair, and let~$\mathcal{H}$ be an $(\ell,S)$-frame in~$G$.
    For some $i\geq1$, if~$P_{i,1}$ is a detached $\mathcal{H}$-earring, then it has a subpath~$Q_i$ such that every chord of~$P_{i,1}$ has both ends in $B_{P_{i,1}}(Q_i,\ell-1)\setminus V(Q_i)$.
\end{lemma}
\begin{proof}
    If~$P_{i,1}$ has no chord, then we choose~$Q_i$ as an arbitrary subpath of~$P_{i,1}$.
    Thus, we may assume that~$P_{i,1}$ has a chord~$ab$.
    Since $(G,S)$ is a good pair, we have $\abs{E(P_{i,1})}\geq3\ell-2$.
    Hence, if each $(a,b)$-subpath of~$P_{i,1}$ contains a vertex in~$S$, then $P_{i,1}+ab$ has an $(\ell,S)$-cycle shorter than~$P_{i,1}$, contradicting~\ref{item:ear3}.
    Thus, one of the two $(a,b)$-subpaths, say~$P$, contains all vertices in $V(P_{i,1})\cap S$ as internal vertices.

    Let~$Q_i$ be a shortest subpath of~$P$ containing all vertices in $V(P_{i,1})\cap S$.
    Let~$v_1$ and~$v_2$ be the ends of~$Q_i$.
    For each $j\in[2]$, let $A_j:=B_{P_{i,1}}(v_j,\ell-1)\setminus V(Q_i)$.
    Note that $A_1\cup A_2=B_{P_{i,1}}(Q_i,\ell-1)\setminus V(Q_i)$.
    
    To prove the statement, we claim that every chord of $P_{i,1}$ has one end in $A_1$ and the other in $A_2$.
    Suppose, to the contrary, that $P_{i,1}$ has a chord $e$ that does not satisfy this property.
    Without loss of generality, we may assume that~$e$ is disjoint from~$A_1$.
    Then $P_{i,1}+e$ has an $(\ell,S)$-cycle containing $A_1\cup\{v_1\}$ shorter than~$P_{i,1}$, contradicting~\ref{item:ear3}.
    Hence, every chord of $P_{i,1}$ has one end in $A_1$ and the other in $A_2$.
    This implies that every chord of~$P_{i,1}$ has both ends in $B_{P_{i,1}}(Q_i,\ell-1)\setminus V(Q_i)$.
\end{proof}

From now on, for every detached $\mathcal{H}$-earring $P_{i,1}$, we denote by~$Q_i$ the subpath of~$P_{i,1}$ obtained by applying~\cref{lem:cycle chords}.

In the following two propositions, we characterise all chords of~$\mathcal{H}$.
The first proposition characterises all chords~$ab$ with $\sigma(a)=\sigma(b)$.

\begin{proposition}\label{prop:induced1}
    Let $(G,S)$ be a good pair and let~$\mathcal{H}$ be an $(\ell,S)$-frame in~$G$.
    If~$\mathcal{H}$ has a chord~$ab$ with $\sigma(a)=\sigma(b)=:(p,q)$, then one of the following holds.
    \begin{enumerate}[label=(\roman*)]
        \item\label{item:prop1-1} $P_{p,q}$ is a detached $\mathcal{H}$-earring and both~$a$ and~$b$ are in~$B_{P_{p,1}}(Q_p,\ell-1)\setminus V(Q_p)$.
        \item\label{item:prop1-2} There is $x\in\Gamma_{p,q}$ such that both~$a$ and~$b$ are in $B_\mathcal{H}(x,\ell-1)$.
        \item\label{item:prop1-3} $P_{p,q}$ is fragile and~$a$ and~$b$ are in distinct components of~$P^{\operatorname{side}}_{p,q}$.
    \end{enumerate}
\end{proposition}
\begin{proof}
    If~$P_{p,q}$ is an $\mathcal{H}$-detour, then by \cref{lem:strict1}, $P_{p,q}$ has length at most $\ell-2$, and therefore~\ref{item:prop1-2} holds.
    If~$P_{p,q}$ is a detached $\mathcal{H}$-earring, then \ref{item:prop1-1} holds by \cref{lem:cycle chords}.
    Thus, we may assume that~$P_{p,q}$ is either an attached $\mathcal{H}$-earring or a regular $\mathcal{H}$-ear.
    Note that $\Gamma_{p,q}\neq\emptyset$.

    We show that if~\ref{item:prop1-2} does not hold, then~\ref{item:prop1-3} holds.
    Let $P'_{p,q}:=P_{p,q}-\Gamma_{p,q}$ and let~$P$ be the graph obtained from~$P_{p,q}$ by replacing $aP'_{p,q}b$ with~$ab$.
    Note that~$P$ is shorter than~$P_{p,q}$.

    Suppose that~$P_{p,q}$ is an attached $\mathcal{H}$-earring.
    Since~\ref{item:prop1-2} does not hold, $a$ or~$b$ is at distance at least~$\ell$ from~$c_p$ in~$P_{p,q}$.
    Thus, $P$ has length at least $\ell+2$.
    Since~$P$ is shorter than~$P_{p,q}$, by~\ref{item:ear2}, it is not an $(\ell,S)$-cycle.
    Thus, $P$ is disjoint from~$S$.
    This implies that~$P_{p,q}$ is fragile and~$P^*_{p,q}$ is a subpath of $aP'_{p,q}b-\{a,b\}$.
    Hence, \ref{item:prop1-3} holds.

    Now, suppose that~$P_{p,q}$ is a regular $\mathcal{H}$-ear.
    Without loss of generality, we may assume that $a_{p,q}P_{p,q}a$ does not contain~$b$.
    Let~$Q$ be the path obtained by concatenating $a_{p,q}P_{p,q}a$, $b_{p,q}P_{p,q}b$, and a shortest $(a_{p,q},b_{p,q})$-path of~$H^-_{p,q-1}$.
    Since~\ref{item:prop1-2} does not hold, one of~$a$ and~$b$ is at distance at least~$\ell$ from~$a_{p,q}$ in~$\mathcal{H}$.
    This implies that $a_{p,q}Qa$ or $a_{p,q}Qb$ has length at least~$\ell$.
    Thus, we have
    \[
        \abs{E(P)}+\dist_{H^-_{p,q-1}}(a_{p,q},b_{p,q})=1+\abs{E(Q)}\geq\ell+2.
    \]
    Since~$P$ is shorter than~$P_{p,q}$, by~\ref{item:ear1}, it is not an $(\ell,S)$-ear of~$H^-_{p,q-1}$.
    Thus, $P$ is disjoint from~$S$, and~$a_{p,q}$ and~$b_{p,q}$ are in the same component of $H^-_{p,q-1}-S$.
    This implies that~$P_{p,q}$ is fragile and~$P^*_{p,q}$ is a subpath of $aP_{p,q}b-\{a,b\}$.
    Hence, \ref{item:prop1-3} holds.
\end{proof}

The second proposition characterises all chords~$ab$ of~$\mathcal{H}$ with $\sigma(a)\neq\sigma(b)$ under an additional assumption that~$\mathcal{H}$ is maximal.

\begin{proposition}\label{prop:induced2}
    Let $(G,S)$ be a good pair and let~$\mathcal{H}$ be a maximal $(\ell,S)$-frame in~$G$.
    If~$\mathcal{H}$ has a chord~$ab$ with $(p,q):=\sigma(a)<_L\sigma(b)=:(r,s)$, then one of the following holds.
    \begin{enumerate}[label=(\roman*)]
        \item\label{item:prop2-1} There is $x\in\Gamma_{r,s}$ such that either
        \begin{itemize}
            \item both~$a$ and~$b$ are in $B_\mathcal{H}(x,\ell-1)$, or
            \item $P_{r,s}$ is a regular $\mathcal{H}$-ear, $b$ is adjacent to~$x$ in~$\mathcal{H}$, $P_{p,q}$ is fragile, and~$a$ and~$x$ are in distinct components of~$P^{\operatorname{side}}_{p,q}$.
        \end{itemize}
        \item\label{item:prop2-2} $P_{r,s}$ is a detached $\mathcal{H}$-earring, $\mu(r)\geq2$, $\abs{E(P_{r,2})}=1$, and there is $x\in\Gamma_{r,2}$ such that either
        \begin{itemize}
            \item both~$a$ and~$b$ are in $B_\mathcal{H}(x,\ell-1)$, or
            \item $b\in\Gamma_{r,2}\setminus\{x\}$, $P_{p,q}$ is fragile, and~$a$ and~$x$ are in distinct components of~$P^{\operatorname{side}}_{p,q}$.
        \end{itemize}
        \item\label{item:prop2-3} $P_{r,s}$ is a regular $\mathcal{H}$-ear, $a\in\Gamma_{r,s}$, $P_{r,s}$ is fragile, and~$a$ and~$b$ are in distinct components of $P_{r,s}-V(P^*_{r,s})$.
    \end{enumerate}
\end{proposition}

The proof of \cref{prop:induced2} is divided into three lemmas below.
The first lemma shows that if $s\geq2$ and $a\notin\Gamma_{r,s}$, then~$ab$ satisfies \cref{prop:induced2}\ref{item:prop2-1}.

\begin{lemma}\label{lem:first}
    Let $(G,S)$ be a good pair and let~$\mathcal{H}$ be an $(\ell,S)$-frame in~$G$.
    If~$\mathcal{H}$ has a chord~$ab$ with $(p,q):=\sigma(a)<_L\sigma(b)=:(r,s)$ such that $s\geq2$ and $a\notin\Gamma_{r,s}$, then there is $x\in\Gamma_{r,s}$ such that either
    \begin{enumerate}[label=(\roman*)]
        \item\label{item:first1} both~$a$ and~$b$ are in $B_\mathcal{H}(x,\ell-1)$, or
        \item\label{item:first2} $b$ is adjacent to~$x$ in~$\mathcal{H}$, $P_{p,q}$ is fragile, and~$a$ and~$x$ are in distinct components of~$P^{\operatorname{side}}_{p,q}$.
    \end{enumerate}
\end{lemma}
\begin{proof}
    Note that~$a$ is an admissible vertex of~$H_{r,s-1}$, because otherwise $b\in Z_{r,s-1}$.
    Since~$b$ is an internal vertex of~$P_{r,s}$, we have $\abs{E(P_{r,s})}\geq2$.
    We divide into two cases according to $\dist_\mathcal{H}(b,\Gamma_{r,s})$.

    \medskip
    \noindent
    \textbf{Case 1.} $\dist_\mathcal{H}(b,\Gamma_{r,s})\geq2$.

    Since~$P_{r,s}$ is an $(\ell,S)$-ear of~$H^-_{r,s-1}$, without loss of generality, we may assume the following:
    \begin{itemize}
        \item if~$P_{r,s}$ contains a vertex in~$S$, then $bP_{r,s}b_{r,s}$ contains the vertex, and
        \item otherwise either $a\in S$ or~$a$ and~$b_{r,s}$ are in distinct components of $H^-_{r,s-1}-S$.
    \end{itemize}
    Let~$P$ be the path obtained by concatenating $bP_{r,s}b_{r,s}$ and~$ab$.
    Note that~$P$ is an $H_{r,s-1}$-path in~$G_{r,s-1}$.
    By the assumption, $P$ is an $S$-ear of~$H^-_{r,s-1}$.
    As $\dist_\mathcal{H}(b,\Gamma_{r,s})\geq2$, $P$ is shorter than~$P_{r,s}$.
    By~\ref{item:ear1}, $P$ is not an $(\ell,S)$-ear of~$H^-_{r,s-1}$.
    Thus, we have
    \[
        \abs{E(P)}+\dist_{H^-_{r,s-1}}(a,b_{r,s})=1+\dist_{P_{r,s}}(b,b_{r,s})+\dist_{H^-_{r,s-1}}(a,b_{r,s})\leq\ell-1.
    \]
    Therefore, both~$a$ and~$b$ are in $B_{H_{r,s}}(b_{r,s},\ell-1)$.
    By \cref{obs:structure}\ref{item:order}, \ref{item:first1} holds for $x:=b_{r,s}$.

    \medskip
    \noindent
    \textbf{Case 2.} $\dist_\mathcal{H}(b,\Gamma_{r,s})=1$.

    Without loss of generality, we may assume the following:
    \begin{itemize}
        \item if $\abs{E(P_{r,s})}=2$, then $\sigma(a_{r,s})\leq_L\sigma(b_{r,s})$, and
        \item otherwise~$b$ is adjacent to~$a_{r,s}$ in~$\mathcal{H}$.
    \end{itemize}
    Let $P:=aba_{r,s}$.
    If $a\in B_\mathcal{H}(a_{r,s},\ell-1)$, then~\ref{item:first1} holds for $x:=a_{r,s}$.
    Thus, we may assume that $a\notin B_\mathcal{H}(a_{r,s},\ell-1)$.

    Suppose that $\sigma(a)=\sigma(a_{r,s})$.
    Note that~$P$ is an $H_{p,q}$-path in~$G_{p,q}$.
    We apply \cref{lem:short path1} for $\mathcal{H}:=H_{p,q}$ and $Q:=P$.
    As $\dist_{P_{p,q}}(a,a_{r,s})\geq\dist_\mathcal{H}(a,a_{r,s})\geq\ell$, \cref{lem:short path1}\ref{item:short path1-1} does not hold.
    Thus, \cref{lem:short path1}\ref{item:short path1-2} holds, and therefore \ref{item:first2} holds for $x:=a_{r,s}$.

    Now, suppose that $\sigma(a)\neq\sigma(a_{r,s})$.
    Let~$\alpha$ and~$\beta$ be the vertices in $\{a,a_{r,s}\}$ such that $\sigma(\alpha)<_L\sigma(\beta)=:(i,j)$.
    Note that $(i,j)<_L(r,s)$.
    Let
    \[
        i':=
        \begin{cases}
            1 & \text{if }(r,s)=(i,2),\\
            \min\{\mu(i),2\} & \text{if }(r,s)\neq(i,2)\text{ and }j=1,\\
            j & \text{otherwise}.
        \end{cases}
    \]
    Note that~$P$ is an $H_{i,i'}$-path in~$G_{i,i'}$ as~$P_{r,s}$ is an $H_{r,s-1}$-path in~$G_{r,s-1}$.
    We apply \cref{lem:short path2} for $\mathcal{H}:=H_{i,i'}$ and $Q:=P$.
    Since $(G,S)$ is a good pair, \cref{lem:short path2}\ref{item:short path2-1} does not hold.
    Thus, \cref{lem:short path2}\ref{item:short path2-2} holds.
    Then~$P_{i,1}$ is a detached $\mathcal{H}$-earring, $(i,j)=(i,1)=(i,i')$, and~$P$ is an $(\ell,S)$-ear of~$H^-_{i,1}$.
    As $(i,j)<_L(r,s)$, we have $\mu(i)\geq2$.
    Thus, by the construction of~$i'$, we have $(r,s)=(i,2)$.
    By~\ref{item:ear1}, $\abs{E(P_{i,2})}=\abs{E(P)}=2$.
    By the assumption, $\sigma(a_{i,2})\leq_L\sigma(b_{i,2})$.
    By \cref{lem:bridge}, $b_{i,2}$ is the unique end of~$P_{i,2}$ in~$P_{i,1}$.
    Thus, $\alpha=a_{i,2}$, $\beta=a$, and $\sigma(a)=(i,1)$.
    Let $P':=abb_{i,2}$.
    Note that~$P'$ is an $H_{i,1}$-path in~$G_{i,1}$
    We apply \cref{lem:short path1} for $\mathcal{H}:=H_{i,1}$ and $Q:=P'$.
    Since~$P_{i,1}$ is a detached $\mathcal{H}$-earring, \cref{lem:short path1}\ref{item:short path1-2} does not hold.
    Thus, \cref{lem:short path1}\ref{item:short path1-1} holds, and therefore \ref{item:first1} holds for $x:=b_{i,2}$.

    \medskip
    This completes the proof.
\end{proof}

The second lemma shows that if $s=1$ and $\Gamma_{r,1}\neq\emptyset$, then~$ab$ satisfies the first item of \cref{prop:induced2}\ref{item:prop2-1}, that is, there is $x\in\Gamma_{r,1}$ such that both~$a$ and~$b$ are in $B_\mathcal{H}(x,\ell-1)$.
We remark that this lemma does not require that~$G$ has no short $(\ell,S)$-cycle.

\begin{lemma}\label{lem:second}
    Let~$G$ be a graph.
    For a set $S\subseteq V(G)$, let~$\mathcal{H}$ be an $(\ell,S)$-frame in~$G$.
    If~$\mathcal{H}$ has a chord~$ab$ with $\sigma(a)<_L\sigma(b)=:(r,1)$ such that $\Gamma_{r,1}\neq\emptyset$, then there is $x\in\Gamma_{r,1}$ such that both~$a$ and~$b$ are in~$B_\mathcal{H}(x,\ell-1)$.
\end{lemma}
\begin{proof}
    Note that~$a$ is an admissible vertex of~$H_{r-1,\mu(r-1)}$, because otherwise $b\in Z_{r-1,\mu(r-1)}$.
    Suppose that~$P_{r,1}$ is an $\mathcal{H}$-detour.
    Without loss of generality, we may assume that $bP_{r,1}b_{r,1}$ contains a vertex in~$S$.
    By \cref{lem:strict1}, $b\in B_\mathcal{H}(b_{r,1},\ell-1)$.
    We may assume that $a\notin B_\mathcal{H}(b_{r,1},\ell-1)$, because otherwise the statement holds for $x:=b_{r,1}$.
    Let~$P$ be the path obtained by concatenating $bP_{r,1}b_{r,1}$ and~$ab$.
    Note that~$P$ is an $H_{r-1,\mu(r-1)}$-path in~$G_{r-1,\mu(r-1)}$.
    As $a\notin B_\mathcal{H}(b_{r,1},\ell-1)$, we have
    \[
        \dist_{H^-_{r-1,\mu(r-1)}}(a,b_{r,1})\geq\dist_\mathcal{H}(a,b_{r,1})\geq\ell.
    \]
    Therefore, $P$ is an $(\ell,S)$-ear of~$H^-_{r-1,\mu(r-1)}$, contradicting~\ref{item:ear4}.

    Now, suppose that~$P_{r,1}$ is an attached $\mathcal{H}$-earring.
    We may assume that~$a$ or~$b$ is not in $B_\mathcal{H}(c_r,\ell-1)$, because otherwise the statement holds for $x:=c_r$.
    If $a=c_r$, then $P_{r,1}+ab$ has an $(\ell,S)$-cycle shorter than~$P_{r,1}$, contradicting~\ref{item:ear2}.
    Thus, $a\neq c_r$.
    Since~$P_{r,1}$ is an $(\ell,S)$-cycle, it has a subpath between~$b$ and~$c_r$ containing a vertex in~$S$.
    Let~$P$ be the path obtained by concatenating this subpath with~$ab$.
    Note that~$P$ is an $H_{r-1,\mu(r-1)}$-path in~$G_{r-1,\mu(r-1)}$.
    Since~$a$ or~$b$ is not in $B_\mathcal{H}(c_r,\ell-1)$, we have
    \[
        \abs{E(P)}+\dist_{H^-_{r-1,\mu(r-1)}}(a,c_r)\geq1+\dist_{P_{r,1}}(b,c_r)+\dist_{H^-_{r-1,\mu(r-1)}}(a,c_r)\geq\ell+2.
    \]
    Therefore, $P$ is an $(\ell,S)$-ear of $H^-_{r-1,\mu(r-1)}$, contradicting~\ref{item:ear2}.
\end{proof}

The last lemma shows that if $s=1$ and $\Gamma_{r,1}=\emptyset$, then~$ab$ satisfies \cref{prop:induced2}\ref{item:prop2-2} under an additional assumption that~$G_{r,\mu(r)}$ has no $(\ell,S)$-ear of $H^-_{r,\mu(r)}$.

\begin{lemma}\label{lem:third}
    Let $(G,S)$ be a good pair and let~$\mathcal{H}$ be an $(\ell,S)$-frame in~$G$.
    If~$\mathcal{H}$ has a chord~$ab$ with $(p,q):=\sigma(a)<_L\sigma(b)=:(r,1)$ such that $\Gamma_{r,1}=\emptyset$ and~$G_{r,\mu(r)}$ has no $(\ell,S)$-ear of~$H^-_{r,\mu(r)}$, then $\mu(r)\geq2$, $\abs{E(P_{r,2})}=1$, and there is $x\in\Gamma_{r,2}$ such that either
    \begin{enumerate}[label=(\roman*)]
        \item\label{item:third1} both~$a$ and~$b$ are in~$B_\mathcal{H}(x,\ell-1)$, or
        \item\label{item:third2} $b\in\Gamma_{r,2}\setminus\{x\}$, $P_{p,q}$ is fragile, and~$a$ and~$x$ are in distinct components of~$P^{\operatorname{side}}_{p,q}$.
    \end{enumerate}
\end{lemma}
\begin{proof}
    Note that~$a$ is an admissible vertex of~$H_{r-1,\mu(r-1)}$, because otherwise $b\in Z_{r-1,\mu(r-1)}$.
    As $\Gamma_{r,1}=\emptyset$, $a$ and~$b$ are admissible vertices of~$H_{r,1}$.
    Since~$a$ and~$b$ are in distinct components of~$H_{r,1}$, the edge~$ab$ forms an $(\ell,S)$-ear of~$H^-_{r,1}$ in~$G_{r,1}$.
    Since $G_{r,\mu(r)}$ has no $(\ell,S)$-ear of~$H_{r,\mu(r)}$, we have $\mu(r)\geq2$.
    By~\ref{item:ear1}, $\abs{E(P_{r,2})}=1$.
    By \cref{lem:bridge}, $P_{r,2}$ has exactly one end, say~$b_{r,2}$, in~$P_{r,1}$.
    Note that $\{a,b\}\neq\Gamma_{r,2}$.
    We divide into three cases according to $\{a,b\}\cap\Gamma_{r,2}$.

    \medskip
    \noindent
    \textbf{Case 1.} $a=a_{r,2}$.

    We may assume that $b\notin B_\mathcal{H}(b_{r,2},\ell-1)$, because otherwise~\ref{item:third1} holds for $x:=b_{r,2}$.
    Since~$P_{r,1}$ is an $(\ell,S)$-cycle, it has a subpath between~$b$ and~$b_{r,2}$ containing a vertex in~$S$.
    Let~$P$ be the cycle obtained by concatenating this subpath with $ba_{r,2}b_{r,2}$.
    Note that~$P$ is a cycle of~$G_{r-1,\mu(r-1)}$ intersecting~$H_{r-1,\mu(r-1)}$ only at~$a_{r,2}$.
    As $b\notin B_\mathcal{H}(b_{r,2},\ell-1)$, $P$ is an $(\ell,S)$-cycle, contradicting~\ref{item:ear3}.

    \medskip
    \noindent
    \textbf{Case 2.} $a\neq a_{r,2}$ and $b=b_{r,2}$.

    Suppose that $\sigma(a)=\sigma(a_{r,2})$.
    We apply \cref{lem:short path1} for $\mathcal{H}:=H_{p,q}$ and $Q:=ab_{r,2}a_{r,2}$.
    As $a\notin B_\mathcal{H}(a_{r,2},\ell-1)$, \cref{lem:short path1}\ref{item:short path1-1} does not hold.
    Thus, \cref{lem:short path1}\ref{item:short path1-2} holds, and therefore~\ref{item:third2} holds for $x:=a_{r,2}$.

    Now, suppose that $\sigma(a)\neq\sigma(a_{r,2})$.
    Let $(i,j)$ be the lexicographically larger one between~$\sigma(a)$ and $\sigma(a_{r,2})$.
    Note that $(i,j)<_L(r,1)$.
    In particular, we have $i<r$.
    Let
    \[
        i':=
        \begin{cases}
            \min\{\mu(i),2\} & \text{if }j=1,\\
            j & \text{otherwise}.
        \end{cases}
    \]
    As $(i,j)<_L(r,1)$, $ab_{r,2}a_{r,2}$ is an $H_{i,i'}$-path in~$G_{i,i'}$.
    We apply \cref{lem:short path2} for $\mathcal{H}:=H_{i,\mu(i)}$ and $Q:=ab_{r,2}a_{r,2}$.
    Since $(G,S)$ is a good pair, \cref{lem:short path2}\ref{item:short path2-1} does not hold.
    Thus, \cref{lem:short path2}\ref{item:short path2-2} holds.
    Then $P_{i,1}$ is a detached $\mathcal{H}$-earring, $(i,j)=(i,1)=(i,\mu(i))$, and this path~$Q$ is an $(\ell,S)$-ear of~$H^-_{i,1}$.
    This contradicts one of \ref{item:ear2}--\ref{item:ear4} as $i<r$.
    
    \medskip
    \noindent
    \textbf{Case 3.} $\{a,b\}\cap\Gamma_{r,2}=\emptyset$.

    Since~$P_{r,1}$ is an $(\ell,S)$-cycle, it has a subpath between~$b$ and~$b_{r,2}$ containing a vertex in~$S$.
    Let~$P$ be the path obtained by concatenating this subpath, $P_{r,2}$, and~$ab$.
    Note that~$P$ is an $H_{r-1,\mu(r-1)}$-path in~$G_{r-1,\mu(r-1)}$.
    As $\abs{E(P_{r,2})}=1$, if both~$a$ and~$b$ are in $B_\mathcal{H}(\Gamma_{r,2},\ell-2)$, then~\ref{item:third1} holds with some $x\in\Gamma_{r,2}$.
    Thus, we may assume that~$a$ or~$b$ is not in $B_\mathcal{H}(\Gamma_{r,2},\ell-2)$.
    We have
    \[
        \abs{E(P)}+\dist_{H^-_{r-1,\mu(r-1)}}(a,a_{r,2})\geq2+\dist_{P_{r,1}}(b,b_{r,2})+\dist_\mathcal{H}(a,a_{r,2})\geq\ell+2.
    \]
    Therefore, $P$ is an $(\ell,S)$-ear of~$H^-_{r-1,\mu(r-1)}$, contradicting~\ref{item:ear3}.
    
    \medskip
    This completes the proof.
\end{proof}

We now prove \cref{prop:induced2}.

\begin{proof}[Proof of \cref{prop:induced2}]
    By the maximality of~$\mathcal{H}$, $G_{r,\mu(r)}$ has no $(\ell,S)$-ear of $H^-_{r,\mu(r)}$.
    By \cref{lem:second,lem:third}, we may assume that $s\geq2$.
    By \cref{lem:first}, we may assume that $a\in\Gamma_{r,s}$.
    Without loss of generality, we may assume that $a=a_{r,s}$.
    
    Let~$P$ be the path obtained by concatenating~$bP_{r,s}b_{r,s}$ and~$ab$.
    Note that~$P$ is an $H_{r,s-1}$-path in~$G_{r,s-1}$ shorter than~$P_{r,s}$.
    By~\ref{item:ear1}, $P$ is not an $(\ell,S)$-ear of~$H^-_{r,s-1}$.
    If~$P$ is not an $S$-ear of~$H^-_{r,s-1}$, then~$P_{r,s}$ is fragile and~$P^*_{r,s}$ is a subpath of $aP_{r,s}b-\{a,b\}$, and therefore \ref{item:prop2-3} holds.
    Thus, we may assume that~$P$ is an $S$-ear of~$H^-_{r,s-1}$.
    Then
    \[
        \abs{E(P)}+\dist_{H^-_{r,s-1}}(a_{r,s},b_{r,s})\leq\ell-1,
    \]
    and therefore~\ref{item:prop2-1} holds for each $x\in\Gamma_{r,s}$.
\end{proof}

In the remainder of this section, we present three corollaries of \cref{prop:induced1,prop:induced2}, which will be used in the next section.
For each $(i,j)\in\mathcal{P}(\mathcal{H})$, let $\Sigma_{i,j}:=\{\sigma(x):x\in\Gamma_{i,j}\}$.
The first corollary characterises all edges of~$G$ between distinct~$P_{i,j}$ and~$P_{i',j'}$ of a maximal~$\mathcal{H}$ such that $\Sigma_{i,j}=\Sigma_{i',j'}\neq\emptyset$.

\begin{corollary}\label{cor:induced1}
    Let $(G,S)$ be a good pair and let~$\mathcal{H}$ be a maximal $(\ell,S)$-frame in~$G$.
    Let $(i,j)$ and $(i',j')$ be distinct pairs in $\mathcal{P}(\mathcal{H})$ with $\Sigma_{i,j}=\Sigma_{i',j'}\neq\emptyset$.
    If~$G$ has an edge $ab$ between~$P_{i,j}$ and~$P_{i',j'}$, then~$ab$ is a chord of~$\mathcal{H}$ and there are $x\in\Gamma_{i,j}$ and $x'\in\Gamma_{i',j'}$ with $\sigma(x)=\sigma(x')=:(i'',j'')$ such that
    \begin{enumerate}[label=(\roman*)]
        \item\label{item:induced1-1} $\{a,b\}\subseteq B_{P_{i,j}}(x,1)\cup B_{P_{i',j'}}(x',1)$, and
        \item\label{item:induced1-2} if~$P_{i'',j''}$ is not a detached $\mathcal{H}$-earring, then~$P_{i'',j''}$ is fragile and~$x$ and~$x'$ are in distinct components of~$P^{\operatorname{side}}_{i'',j''}$.
    \end{enumerate}
\end{corollary}
\begin{proof}
    Without loss of generality, we may assume that $a\in V(P_{i,j})$, $b\in V(P_{i',j'})$, and $(p,q):=\sigma(a)\leq_L\sigma(b)=:(r,s)$.
    As $\Sigma_{i,j}=\Sigma_{i',j'}\neq\emptyset$, neither~$P_{i,j}$ nor~$P_{i',j'}$ is a detached $\mathcal{H}$-earring.
    In addition, by \cref{lem:degree}, $P_{i,j}$ and~$P_{i',j'}$ are vertex-disjoint.
    By \cref{lem:argue distance}, $\dist_\mathcal{H}(P_{i,j},P_{i',j'})\geq\ell$, and therefore~$ab$ is a chord of~$\mathcal{H}$.

    Suppose that $\sigma(a)=\sigma(b)$.
    Since~$P_{i,j}$ and~$P_{i',j'}$ are vertex-disjoint, we have $a\in\Gamma_{i,j}$ and $b\in\Gamma_{i',j'}$.
    We show that the statement holds for $x:=a$ and $x':=b$.
    Note that~\ref{item:induced1-1} holds.
    We may assume that~$P_{p,q}$ is not a detached $\mathcal{H}$-earring, because otherwise~\ref{item:induced1-2} obviously holds.
    We apply \cref{prop:induced1} for~$\mathcal{H}$ and~$ab$.
    Since~$P_{p,q}$ is not a detached $\mathcal{H}$-earring, \cref{prop:induced1}\ref{item:prop1-1} does not hold.
    Since~$a$ and~$b$ are admissible vertices of~$H_{p,q}$, \cref{prop:induced1}\ref{item:prop1-2} does not hold.
    Thus, \cref{prop:induced1}\ref{item:prop1-3} holds, and therefore the statement holds for $x:=a$ and $x':=b$.

    Now, suppose that $\sigma(a)\neq\sigma(b)$.
    We divide into two cases according to whether $\sigma(b)=(i',j')$.

    \medskip
    \noindent
    \textbf{Case 1.} $\sigma(b)=(i',j')$.

    Let
    \[
        (i^*,j^*):=
        \begin{cases}
            (i'-1,\mu(i'-1)) & \text{if }j'=1,\\
            (i',j'-1) & \text{otherwise}.
        \end{cases}
    \]
    Note that~$a$ is an admissible vertex of~$H_{i^*,j^*}$, because otherwise $b\in Z_{i^*,j^*}$.
    We apply \cref{prop:induced2} for~$\mathcal{H}$ and~$ab$.
    Since~$P_{i',j'}$ is not a detached $\mathcal{H}$-earring, \cref{prop:induced2}\ref{item:prop2-2} does not hold.
    Since~$P_{i,j}$ and~$P_{i',j'}$ are vertex-disjoint, \cref{prop:induced2}\ref{item:prop2-3} does not hold.
    Thus, \cref{prop:induced2}\ref{item:prop2-1} holds for some $x'\in\Gamma_{i',j'}$.
    As $\dist_\mathcal{H}(P_{i,j},P_{i',j'})\geq\ell$, we have $a\notin B_\mathcal{H}(x',\ell-1)$.
    Thus, $b$ is adjacent to~$x'$ in~$\mathcal{H}$, $P_{p,q}$ is fragile, and~$a$ and~$x'$ are in distinct components of~$P^{\operatorname{side}}_{p,q}$.
    Note that $(p,q)\neq(i,j)$ as~$P_{i,j}$ and~$P_{i',j'}$ are vertex-disjoint.
    Thus, $a\in\Gamma_{i,j}$, and therefore~\ref{item:induced1-1} and~\ref{item:induced1-2} hold for $x:=a$ and~$x'$.

    \medskip
    \noindent
    \textbf{Case 2.} $\sigma(b)\neq(i',j')$.
    
    Note that $b\in\Gamma_{i',j'}$.
    We apply \cref{prop:induced2} for~$\mathcal{H}$ and~$ab$.
    Since~$b$ is an admissible vertex of~$H_{r,s}$, \cref{prop:induced2}\ref{item:prop2-1} does not hold.
    Suppose that \cref{prop:induced2}\ref{item:prop2-2} holds.
    Then~$b$ is not an admissible vertex of~$H_{r,2}$.
    Since~$b$ is an admissible vertex of~$H_{r,1}$, we deduce that $(r,2)=(i',j')$.
    Let~$x'$ be the vertex in $\Gamma_{r,2}\setminus\{b\}$.
    As $\dist_\mathcal{H}(P_{i,j},P_{i',j'})\geq\ell$, we have $a\notin B_\mathcal{H}(x',\ell-1)$.
    Thus, $P_{p,q}$ is fragile and~$a$ and~$x'$ are in distinct components of~$P^{\operatorname{side}}_{p,q}$.
    Note that $(p,q)\neq(i,j)$ as~$P_{i,j}$ and~$P_{i',j'}$ are vertex-disjoint.
    Thus, $a\in\Gamma_{i,j}$, and therefore~\ref{item:induced1-1} and~\ref{item:induced1-2} hold for $x:=a$ and~$x'$.

    Now, suppose that \cref{prop:induced2}\ref{item:prop2-3} holds.
    Then $a\in\Gamma_{r,s}$.
    Note that $(r,s)\neq(i,j)$ as~$P_{i,j}$ and~$P_{i',j'}$ are vertex-disjoint.
    By \cref{lem:degree}, $a\notin\Gamma_{i,j}$.
    Thus, $\sigma(a)=(i,j)<_L(r,s)$, and therefore $(r,s)\notin\Sigma_{i,j}$, contradicting that $\Sigma_{i,j}=\Sigma_{i',j'}$.

    \medskip
    This completes the proof.
\end{proof}

The second corollary shows that if~$G$ has an induced packing $\{C_1,C_2\}$ of cycles of~$\mathcal{H}$, then $\{C^-_1,C^-_2\}$ is an induced packing in~$G$ as well.

\begin{corollary}\label{cor:induced2}
    Let $(G,S)$ be a good pair and let~$\mathcal{H}$ be a maximal $(\ell,S)$-frame in~$G$.
    If~$G$ has an induced packing $\{C_1,C_2\}$ of cycles of~$\mathcal{H}$, then $\{C^-_1,C^-_2\}$ is an induced packing in~$G$.
\end{corollary}
\begin{proof}
    Note that~$C^-_1$ and~$C^-_2$ are vertex-disjoint.
    Towards a contradiction, suppose that~$G$ has an edge~$ab$ with $a\in V(C^-_1)$ and $b\in V(C^-_2)$.
    Since $\{C_1,C_2\}$ is an induced packing in~$G$, for some $i\geq2$, $P_{i,1}$ is an $\mathcal{H}$-detour containing~$a$ or~$b$ as an internal vertex such that every vertex of~$Q_i$ has degree~$2$ in~$C_1$ or~$C_2$, respectively.
    Without loss of generality, we may assume that~$b$ is an internal vertex of~$P_{i,1}$.
    Note that $\sigma(b)=(i,1)\neq\sigma(a)$.
    We apply \cref{prop:induced2} for~$\mathcal{H}$ and~$ab$.
    By \cref{lem:strict1}, we have $B_\mathcal{H}(b,\ell-1)\subseteq B_\mathcal{H}(\Gamma_{i,1},\ell-1)\subseteq V(C^-_2\cup C_2)$.
    Thus, $\mathcal{H}$ has no branch vertex~$x$ such that both~$a$ and~$b$ are in $B_\mathcal{H}(x,\ell-1)$.

    Suppose that $\sigma(a)<_L\sigma(b)$.
    Note that \cref{prop:induced2}\ref{item:prop2-1} does not hold.
    Since~$P_{i,1}$ is an $\mathcal{H}$-detour, \cref{prop:induced2}\ref{item:prop2-2} does not hold.
    Since~$C^-_1$ and~$C^-_2$ are vertex-disjoint, \cref{prop:induced2}\ref{item:prop2-3} does not hold.
    Hence, \cref{prop:induced2} does not hold, a contradiction.

    Now, suppose that $\sigma(a)>_L\sigma(b)$.
    Since~$P_{i,1}$ is not fragile, neither \cref{prop:induced2}\ref{item:prop2-1} nor \cref{prop:induced2}\ref{item:prop2-2} holds.
    Since~$b$ has degree~$2$ in~$\mathcal{H}$, \cref{prop:induced2}\ref{item:prop2-3} does not hold.
    Hence, \cref{prop:induced2} does not hold, a contradiction.
\end{proof}

The last corollary characterises all chords of a maximal $\mathcal{H}$ between distinct cycles of~$\mathcal{H}$ which form an induced packing in~$\mathcal{H}$.

\begin{corollary}\label{cor:induced3}
    Let $(G,S)$ be a good pair and let~$\mathcal{H}$ be a maximal $(\ell,S)$-frame in~$G$.
    If~$\mathcal{H}$ has an induced packing $\{C_1,C_2\}$ of cycles and a chord~$ab$ between~$C_1$ and~$C_2$ with $(p,q):=\sigma(a)\leq_L\sigma(b)=:(r,s)$, then one of the following holds.
    \begin{enumerate}[label=(\Roman*)]
        \item\label{item:packing1} $(p,q)=(r,s)$, $P_{p,q}$ is a detached $\mathcal{H}$-earring, and both~$a$ and~$b$ are in $B_{P_{p,1}}(Q_p,\ell-1)\setminus V(Q_p)$.
        \item\label{item:packing2} $(p,q)=(r,s)$, $P_{p,q}$ is fragile, and~$a$ and~$b$ are in distinct components of~$P^{\operatorname{side}}_{p,q}$.
        \item\label{item:packing3} $(p,q)\neq(r,s)$, $P_{r,s}$ is a regular $\mathcal{H}$-ear, and there is $x\in\Gamma_{r,s}$ such that~$b$ is adjacent to~$x$ in~$\mathcal{H}$, $P_{p,q}$ is fragile, and~$a$ and~$x$ are in distinct components of~$P^{\operatorname{side}}_{p,q}$.
        \item\label{item:packing4} $(p,q)\neq(r,s)$, $P_{r,s}$ is a detached $\mathcal{H}$-earring, $\mu(r)\geq2$, $\abs{E(P_{r,2})}=1$, and there is $x\in\Gamma_{r,2}$ such that $b\in\Gamma_{r,2}\setminus\{x\}$, $P_{p,q}$ is fragile, and~$a$ and~$x$ are in distinct components of~$P^{\operatorname{side}}_{p,q}$.
        \item\label{item:packing5} $(p,q)\neq(r,s)$, $P_{r,s}$ is a regular $\mathcal{H}$-ear, $a\in\Gamma_{r,s}$, $P_{r,s}$ is fragile, and~$a$ and~$b$ are in distinct components of ${P_{r,s}-V(P^*_{r,s})}$.
        \item\label{item:packing6} There is a regular $\mathcal{H}$-ear~$P_{i,j}$ such that~$H_{i,j}$ has an $(a_{i,j},b_{i,j})$-path~$Q$ of length at most $\ell-1$ and for each $h\in[2]$, there are $x_h\in\Gamma_{i,j}$ and $y_h\in\{a,b\}$ with $y_h\in B_{\mathcal{H}-E(Q)}(x_h,\ell-1)\subseteq V(C_h)$.
    \end{enumerate}
\end{corollary}
\begin{proof}
    Suppose that none of~\ref{item:packing1}--\ref{item:packing5} hold.
    We show that~\ref{item:packing6} holds.
    By \cref{prop:induced1} or \cref{prop:induced2}, $\mathcal{H}$ has a branch vertex~$x$ such that both~$a$ and~$b$ are in $B_\mathcal{H}(x,\ell-1)$.
    If~$\mathcal{H}$ has no other branch vertex in $B_\mathcal{H}(x,\ell-1)$, then both~$C_1$ and~$C_2$ contain~$x$, a contradiction.
    Thus, $\mathcal{H}$ has another branch vertex~$x'$ in $B_\mathcal{H}(x,\ell-1)$.
    As $\dist_\mathcal{H}(x,x')\leq\ell-1$, by \cref{cor:short branch}, there is a unique $(i,j)\in\mathcal{P}(\mathcal{H})$ such that $\{x,x'\}=\{a_{i,j},b_{i,j}\}$ and both~$a_{i,j}$ and~$b_{i,j}$ have degree~$3$ in~$\mathcal{H}$.
    Note that~$a_{i,j}$ and~$b_{i,j}$ are the only branch vertices of~$\mathcal{H}$ in $B_\mathcal{H}(\Gamma_{i,j},\ell-1)$.
    Thus, each of~$C_1$ and~$C_2$ contains one of~$a_{i,j}$ and~$b_{i,j}$.
    Without loss of generality, we may assume that $\{a,a_{i,j}\}\subseteq V(C_1)$ and $\{b,b_{i,j}\}\subseteq V(C_2)$.

    By \cref{obs:structure}\ref{item:order}, $B_{H_{i,j}}(\Gamma_{i,j},\ell-1)=B_\mathcal{H}(\Gamma_{i,j},\ell-1)$.
    Thus, $\dist_{H_{i,j}}(a_{i,j},b_{i,j})\leq\ell-1$.
    Let~$Q$ be a shortest $(a_{i,j},b_{i,j})$-path of~$H_{i,j}$.
    If~$P_{i,j}$ is an $\mathcal{H}$-detour, then by \cref{lem:strict1}, $Q$ is either~$P_{i,1}$ or~$Q_i$ and hence each of~$C_1$ and~$C_2$ contains both~$a_{i,j}$ and~$b_{i,j}$, a contradiction.
    Thus, $P_{i,j}$ is a regular $\mathcal{H}$-ear.
    Since~$Q$ is a path of length at most $\ell-1$ between branch vertices of~$\mathcal{H}$, every internal vertex of~$Q$ has degree~$2$ in~$\mathcal{H}$.
    Thus, $Q$ is edge-disjoint from $C_1\cup C_2$.
    Since both~$a_{i,j}$ and~$b_{i,j}$ have degree~$3$ in~$\mathcal{H}$, we have $a\in B_{\mathcal{H}-E(Q)}(a_{i,j},\ell-1)\subseteq V(C_1)$ and $b\in B_{\mathcal{H}-E(Q)}(b_{i,j},\ell-1)\subseteq V(C_2)$.
\end{proof}

\section{Large induced packings}\label{sec:large packings}

Throughout this section, we assume that~$(G,S)$ is a good pair.
Let~$\mathcal{H}$ be a maximal $(\ell,S)$-frame in~$G$.
The goal of this section is to prove the following proposition showing that if~$\mathcal{H}^-$ has many branch vertices, then~$G$ has an induced packing of~$k$ $(\ell,S)$-cycles.
We refer to \cref{thm:simonovitz} for the value~$s_k$.

\begin{proposition}\label{prop:final}
    Let $(G,S)$ be a good pair and let~$\mathcal{H}$ be a maximal $(\ell,S)$-frame in~$G$.
    If~$\mathcal{H}^-$ has at least
    \[
        10^{20}\cdot k^9(2k-1)^4+2s_k
    \]
    branch vertices for some positive integer~$k$, then~$G$ has an induced packing of~$k$ $(\ell,S)$-cycles.
\end{proposition}

To prove \cref{prop:final}, we will use the following lemmas.
For any pairs $(p,q),(r,s)\in\mathcal{P}(\mathcal{H})$, consider the set~$\mathcal{Q}$ of pairs $(i,j)\in\mathcal{P}(\mathcal{H})$ with $\Sigma_{i,j}=\{(p,q),(r,s)\}$.
The following two lemmas show that if~$\mathcal{Q}$ is large, then~$G$ has an induced packing of~$k$ $(\ell,S)$-cycles.
The first lemma deals with the case $(p,q)\neq(r,s)$ and the second lemma deals with the case $(p,q)=(r,s)$.

\begin{lemma}\label{lem:jump1}
    Let $(G,S)$ be a good pair and let $\mathcal{H}$ be a maximal $(\ell,S)$-frame in~$G$.
    For distinct pairs $(p,q),(r,s)\in\mathcal{P}(\mathcal{H})$, let~$\mathcal{Q}$ be the set of pairs $(i,j)\in\mathcal{P}(\mathcal{H})$ with $\Sigma_{i,j}=\{(p,q),(r,s)\}$.
    If $\abs{\mathcal{Q}}\geq16(k+8)^2+16$ for some positive integer~$k$, then~$G$ has an induced packing of~$k$ $(\ell,S)$-cycles.
\end{lemma}
\begin{proof}
    As $(p,q)\neq(r,s)$, for every $(i,j)\in\mathcal{Q}$, $P_{i,j}$ is a regular $\mathcal{H}$-ear.
    Without loss of generality, we may assume that $a_{i,j}\in V(P_{p,q})$ and $b_{i,j}\in V(P_{r,s})$.
    We denote by~$a'_{i,j}$ and $b'_{i,j}$ the neighbours of~$a_{i,j}$ and~$b_{i,j}$ in~$P_{i,j}$, respectively.
    Let $\mathcal{C}_0:=\{P_{i,j}:(i,j)\in\mathcal{Q}\}$.
    By \cref{lem:degree}, the paths in~$\mathcal{C}_0$ are pairwise vertex-disjoint.
    In addition, by \cref{lem:argue distance}, they are pairwise at distance at least~$\ell$ in~$\mathcal{H}$.
    This obviously implies that~$\mathcal{C}_0$ is an induced packing in~$\mathcal{H}$.

    We define a subpath~$P'_{p,q}$ of~$P_{p,q}$ as follows: if~$P_{p,q}$ is a detached $\mathcal{H}$-earring, then for an arbitrary vertex~$c_{p,q}$ of~$P_{p,q}$, let $P'_{p,q}:=P_{p,q}-c_{p,q}$, and otherwise let $P'_{p,q}:=P_{p,q}-\Gamma_{p,q}$.
    Moreover, we define two subpaths~$P^1_{p,q}$ and~$P^2_{p,q}$ of~$P_{p,q}$ as follows:
    \begin{itemize}
        \item if~$P_{p,q}$ is fragile, then let~$P^1_{p,q}$ and~$P^2_{p,q}$ be the two subpaths of~$P_{p,q}$ between~$a^*_{p,q}$ and~$\Gamma_{p,q}$, and
        \item otherwise let $P^1_{p,q}:=P'_{p,q}$ and let~$P^2_{p,q}$ be a null graph.
    \end{itemize}
    Similarly, we define three subpaths~$P'_{r,s}$, $P^1_{r,s}$, and~$P^2_{r,s}$ of~$P_{r,s}$.
    For each $h\in[2]$, let $\mathcal{Q}_h$ be the set of pairs $(i,j)\in\mathcal{Q}$ such that $a_{i,j}\in V(P^h_{p,q})$.
    We have
    \[
        \abs{\mathcal{Q}_1}+\abs{\mathcal{Q}_2}\geq\abs{\mathcal{Q}}-1\geq16(k+8)^2+15.
    \]
    Thus, for some $h\in[2]$, $\abs{\mathcal{Q}_h}\geq8(k+8)^2+8$.
    Again, for each $m\in[2]$, let $\mathcal{Q}'_m$ be the set of pairs $(i,j)\in\mathcal{Q}_h$ such that $b_{i,j}\in V(P^m_{r,s})$.
    We have
    \[
        \abs{\mathcal{Q}'_1}+\abs{\mathcal{Q}'_2}\geq\abs{\mathcal{Q}_h}-1\geq8(k+8)^2+7.
    \]
    Thus, for some $m\in[2]$, $\abs{\mathcal{Q}'_m}\geq4(k+8)^2+4=(2k+16)^2+4$.
    Let~$\mathcal{Q}''$ be the set of pairs $(i,j)\in\mathcal{Q}'_m$ such that for each $(x,y)\in\{(p,q),(r,s)\}$, if $P_{x,y}$ is a detached $\mathcal{H}$-earring, then~$P_{i,j}$ is disjoint from $B_{P_{x,y}}(Q_x,\ell-1)\setminus V(Q_x)$.
    Since the paths in~$\mathcal{C}_0$ are pairwise at distance at least~$\ell$ in~$\mathcal{H}$, we have $\abs{\mathcal{Q}''}\geq\abs{\mathcal{Q}'_m}-4\geq(2k+16)^2$.
    By the choice of~$\mathcal{Q}'_m$ and \cref{cor:induced1}, $\{P_{i,j}:(i,j)\in\mathcal{Q}''\}$ is an induced packing in~$G$.

    Let~$\preceq_{p,q}$ and~$\preceq_{r,s}$ be the linear orders on the vertices of $P'_{p,q}$ and $P'_{r,s}$, respectively, induced by the order in which the vertices appear along the corresponding paths.
    Let~$\preceq_1$ and $\preceq_2$ be partial orders on~$\mathcal{Q}''$ such that for distinct $(i,j),(i',j')\in\mathcal{Q}''$,
    \begin{itemize}
        \item $(i,j)\prec_1(i',j')$ if and only if $a_{i,j}\prec_{p,q} a_{i',j'}$ and $b_{i,j}\prec_{r,s}b_{i',j'}$, and
        \item $(i,j)\prec_2(i',j')$ if and only if $a_{i,j}\prec_{p,q} a_{i',j'}$ and $b_{i',j'}\prec_{r,s}b_{i,j}$.
    \end{itemize}
    We remark that every chain of $(\mathcal{Q}'',\preceq_1)$ is an antichain of $(\mathcal{Q}'',\preceq_2)$, and vice versa.
    
    As $\abs{\mathcal{Q}''}\geq(2k+16)^2$, by \cref{thm:Dilworth}, $(\mathcal{Q}'',\preceq_1)$ or $(\mathcal{Q}'',\preceq_2)$ has a chain~$\mathcal{Q}^*$ of size $2k+16$.
    Let $(i_1,j_1),\ldots,(i_{2k+16},j_{2k+16})$ be the pairs in~$\mathcal{Q}^*$ with $a_{i_1,j_1}\prec_{p,q}\cdots\prec_{p,q}a_{i_{2k+16},j_{2k+16}}$.
    For each $h\in[k+8]$,
    \begin{itemize}
        \item let~$A_h$ be the subpath of $P'_{p,q}$ between~$a_{i_{2h-1},j_{2h-1}}$ and $a_{i_{2h},j_{2h}}$,
        \item let~$B_h$ be the subpath of $P'_{r,s}$ between~$b_{i_{2h-1},j_{2h-1}}$ and $b_{i_{2h},j_{2h}}$,
        \item let $P_h:=P_{i_{2h-1},j_{2h-1}}\cup P_{i_{2h},j_{2h}}$, and
        \item let~$C_h$ be the cycle obtained by concatenating~$A_h$, $B_h$, and~$P_h$.
    \end{itemize}
    Since~$\mathcal{Q}^*$ is a chain of $(\mathcal{Q}'',\preceq_1)$ or $(\mathcal{Q}'',\preceq_2)$, the cycles in $\{C_h:h\in[k+8]\}$ are pairwise vertex-disjoint.
    Let~$\mathcal{C}_1$ be the set of cycles $C_h$ for $h\in[k+8]$ disjoint from $\Gamma_{p,q}\cup\Gamma_{r,s}$ such that for each $(i,j)\in\{(p,q),(r,s)\}$, if~$P_{i,j}$ is a detached $\mathcal{H}$-earring and $\mu(i)\geq2$, then~$C_h$ is disjoint from~$\Gamma_{i,2}$.
    Note that $\abs{\mathcal{C}_1}\geq k$.

    We show that~$\mathcal{C}_1$ is an induced packing in~$G$.
    Towards a contradiction, suppose that~$G$ has an edge~$ab$ between distinct~$C_h$ and~$C_{h'}$ in~$\mathcal{C}_1$, where $a\in V(C_h)$ and $b\in V(C_{h'})$.
    Since $\{P_{i,j}:(i,j)\in\mathcal{Q}''\}$ is an induced packing in~$G$, $a$ or~$b$, say~$a$, is in $P_{p,q}\cup P_{r,s}$.
    We remark that every $(i,j)\in\mathcal{Q}''$ is lexicographically larger than both $(p,q)$ and $(r,s)$, and has non-empty~$\Gamma_{i,j}$.

    Suppose that~$ab$ is an edge of~$\mathcal{H}$.
    Note that~$a$ and~$b$ are branch vertices of~$\mathcal{H}$.
    By \cref{cor:short branch}, there is a unique $(x,y)\in\mathcal{P}(\mathcal{H})$ such that $\{a,b\}=\{a_{x,y},b_{x,y}\}$.
    If~$P_{x,y}$ is an $\mathcal{H}$-detour, then $\sigma(a)=\sigma(b)$ and hence either~$P_{p,q}$ or~$P_{r,s}$, say~$P_{p,q}$, contains both~$a$ and~$b$.
    By \cref{lem:strict1}, $\dist_{P_{p,q}}(P_h,P_{h'})\leq\dist_{P_{p,q}}(a,b)\leq\ell-2$, contradicting that the paths in~$\mathcal{C}_0$ are pairwise at distance at least~$\ell$ in~$\mathcal{H}$.
    Thus, $P_{x,y}$ is a regular $\mathcal{H}$-ear.
    By the construction of~$\mathcal{C}_1$ and \cref{cor:chord}, we have $\abs{E(P_{x,y})}\geq\ell$.
    Thus, $ab$ is an edge of~$H_{x,y-1}$.
    This implies that $\sigma(a)=\sigma(b)$, and hence either~$P_{p,q}$ or~$P_{r,s}$ contains both~$a$ and~$b$.
    Then~$ab$ is an edge between~$P_h$ and~$P_{h'}$, a contradiction.

    Hence, $ab$ is a chord of~$\mathcal{H}$.
    We divide into three cases according to whether~$b$ lies in~$P_{p,q}\cup P_{r,s}$.

    \medskip
    \noindent
    \textbf{Case 1.} $P_{p,q}$ or~$P_{r,s}$ contains both~$a$ and~$b$.

    Without loss of generality, we may assume that~$P_{p,q}$ contains both~$a$ and~$b$.
    We apply \cref{prop:induced1} for~$\mathcal{H}$ and~$ab$.
    By the construction of~$\mathcal{Q}''$, neither \cref{prop:induced1}\ref{item:prop1-1} nor \cref{prop:induced1}\ref{item:prop1-3} holds.
    Since~$a$ and~$b$ are admissible vertices of~$H_{p,q}$, \cref{prop:induced1}\ref{item:prop1-2} does not hold.
    Hence, \cref{prop:induced1} does not hold, a contradiction.

    \medskip
    \noindent
    \textbf{Case 2.} Each of $P_{p,q}$ and~$P_{r,s}$ contains~$a$ or~$b$.

    Without loss of generality, we may assume that $a\in V(P_{p,q})$, $b\in V(P_{r,s})$, and $(p,q)<_L(r,s)$.
    We apply \cref{prop:induced2} for~$\mathcal{H}$ and~$ab$.
    Since~$b$ is an admissible vertex of~$H_{r,s}$, \cref{prop:induced2}\ref{item:prop2-1} does not hold.
    By the construction of~$\mathcal{Q}''$, \cref{prop:induced2}\ref{item:prop2-3} does not hold.
    Thus, \cref{prop:induced2}\ref{item:prop2-2} holds.
    Then~$P_{r,s}$ is a detached $\mathcal{H}$-earring, $\mu(r)\geq2$, and~$b$ is not an admissible vertex of~$H_{r,2}$.
    Since~$b$ is an admissible vertex of~$H_{r,1}$, we have $\dist_{P_{r,1}}(b,\Gamma_{r,2})\leq\ell-1$.
    Since~$C_{h'}$ is disjoint from~$\Gamma_{r,2}$, we have $\dist_{P_{r,1}}(P_{h'},\Gamma_{r,2})\leq\ell-1$.
    Hence, $P_{h'}$ contains a vertex in~$Y_{r,2}$, a contradiction.
    
    \medskip
    \noindent
    \textbf{Case 3.} $b\notin V(P_{p,q}\cup P_{r,s})$.

    Recall that $a$ is in $P_{p,q}\cup P_{r,s}$.
    Without loss of generality, we may assume that $a\in V(P_{p,q})$.
    Note that $(p,q)=\sigma(a)<_L\sigma(b)$.
    Let $\sigma(b):=(i'',j'')$.
    We apply \cref{prop:induced2} for~$\mathcal{H}$ and~$ab$.
    Since~$P_{i'',j''}$ is not a detached $\mathcal{H}$-earring, \cref{prop:induced2}\ref{item:prop2-2} does not hold.
    Since~$C_h$ and~$C_{h'}$ are vertex-disjoint, \cref{prop:induced2}\ref{item:prop2-3} does not hold.
    Thus, \cref{prop:induced2}\ref{item:prop2-1} holds.
    By the construction of~$\mathcal{Q}''$, there is no $x\in\Gamma_{i'',j''}$ such that~$a$ and~$x$ are in distinct components of~$P^{\operatorname{side}}_{p,q}$.
    Thus, for some $x\in\Gamma_{i'',j''}$, both~$a$ and~$b$ are in $B_\mathcal{H}(x,\ell-1)$.
    As $B_{\mathcal{H}-E(P_{i'',j''})}(x,\ell-1)\subseteq V(P_{p,q})$, we have $\dist_{P_{p,q}}(a,x)\leq\ell-1$.
    Since~$a$ is in~$A_h\cup B_h$, we have $\dist_{P_{p,q}}(x,P_h)\leq\ell-1$, contradicting that the paths in~$\mathcal{C}_0$ are pairwise at distance at least~$\ell$ in~$\mathcal{H}$.

    \medskip
    Therefore, $\mathcal{C}_1$ is an induced packing in~$G$.
    By \cref{lem:all cycles}, for each $C\in\mathcal{C}_1$, $C^-$ is an $(\ell,S)$-cycle.
    Thus, by \cref{cor:induced2}, $\{C^-:C\in\mathcal{C}_1\}$ is an induced packing in~$G$ of at least~$k$ $(\ell,S)$-cycles.
\end{proof}

\begin{lemma}\label{lem:jump2}
    Let $(G,S)$ be a good pair and let~$\mathcal{H}$ be a maximal $(\ell,S)$-frame in~$G$.
    For a pair $(p,q)\in\mathcal{P}(\mathcal{H})$, let~$\mathcal{Q}$ be the set of pairs $(i,j)\in\mathcal{P}(\mathcal{H})$ with $j\geq2$ and $\Sigma_{i,j}=\{(p,q)\}$.
    If $\abs{\mathcal{Q}}\geq(2k+5)(26k-19)^2$ for some positive integer~$k$, then~$G$ has an induced packing of~$k$ $(\ell,S)$-cycles.
\end{lemma}
\begin{proof}
    For each $(i,j)\in\mathcal{Q}$, we denote by~$a'_{i,j}$ and~$b'_{i,j}$ the neighbours of~$a_{i,j}$ and~$b_{i,j}$ in~$P_{i,j}$, respectively.
    Let $\mathcal{C}_0:=\{P_{i,j}:(i,j)\in\mathcal{Q}\}$.
    By \cref{lem:degree}, the paths in~$\mathcal{C}_0$ are pairwise vertex-disjoint.
    In addition, by \cref{lem:argue distance}, they are pairwise at distance at least~$\ell$ in~$\mathcal{H}$.
    This obviously implies that~$\mathcal{C}_0$ is an induced packing in~$\mathcal{H}$.
    
    We define a subpath~$P'_{p,q}$ of~$P_{p,q}$ as follows: if~$P_{p,q}$ is a detached $\mathcal{H}$-earring, then for an arbitrary vertex~$c_{p,q}$ of~$P_{p,q}$, let $P'_{p,q}:=P_{p,q}-c_{p,q}$, and otherwise let $P'_{p,q}:=P_{p,q}-\Gamma_{p,q}$.
    Let~$\preceq$ be the linear order on the vertices of $P'_{p,q}$ induced by the order in which the vertices appear along the path.
    Without loss of generality, for every $(i,j)\in\mathcal{Q}$, we may assume that $a_{i,j}\prec b_{i,j}$.
    Let~$\preceq_0$, $\preceq_1$, and $\preceq_2$ be partial orders on~$\mathcal{Q}$ such that for distinct $(i,j),(i',j')\in\mathcal{Q}$,
    \begin{itemize}
        \item $(i,j)\prec_0(i',j')$ if and only if $a_{i,j}\prec b_{i,j}\prec a_{i',j'}\prec b_{i',j'}$,
        \item $(i,j)\prec_1(i',j')$ if and only if $a_{i,j}\prec a_{i',j'}\prec b_{i,j}\prec b_{i',j'}$, and
        \item $(i,j)\prec_2(i',j')$ if and only if $a_{i,j}\prec a_{i',j'}\prec b_{i',j'}\prec b_{i,j}$.
    \end{itemize}
    We remark that every chain of $(\mathcal{Q},\preceq_1)$ is an antichain of $(\mathcal{Q},\preceq_2)$, and vice versa.
    We divide into two cases according to whether $(\mathcal{Q},\preceq_0)$ has a chain of size $2k+6$.

    \medskip
    \noindent
    \textbf{Case 1.} $(\mathcal{Q},\preceq_0)$ has a chain~$\mathcal{Q}'$ of size $2k+6$.

    Let $(i_1,j_1),\ldots,(i_{2k+5},j_{2k+5})$ be distinct pairs in~$\mathcal{Q}'$ such that if~$P_{p,q}$ is a detached $\mathcal{H}$-earring, then $c_{p,q}\notin\Gamma_{i_h,j_h}$ for each $h\in[2k+5]$.
    Without loss of generality, we may assume that $a_{i_1,j_1}\prec\cdots\prec a_{i_{2k+5},j_{2k+5}}$.
    For each $h\in[2k+5]$,
    \begin{itemize}
        \item let~$A_h$ be the subpath of $P'_{p,q}$ between~$a_{i_h,j_h}$ and~$b_{i_h,j_h}$, and
        \item let~$C_h$ be the cycle obtained by concatenating~$A_h$ and $P_{i_h,j_h}$.
    \end{itemize}
    Since~$\mathcal{Q}'$ is a chain of $(\mathcal{Q},\preceq_0)$, the cycles in $\{C_h:h\in[2k+5]\}$ are pairwise vertex-disjoint.

    We define two subpaths~$P^1_{p,q}$ and~$P^2_{p,q}$ of~$P_{p,q}$ as follows:
    \begin{itemize}
        \item if~$P_{p,q}$ is fragile, then let~$P^1_{p,q}$ and~$P^2_{p,q}$ be the two subpaths of~$P_{p,q}$ between~$a^*_{p,q}$ and~$\Gamma_{p,q}$, and
        \item otherwise let $P^1_{p,q}:=P'_{p,q}$ and let~$P^2_{p,q}$ be a null graph.
    \end{itemize}
    For each $m\in[2]$, let~$\mathcal{Q}_m$ be the set of pairs $(i_h,j_h)$ for $h\in[2k+5]$ such that~$A_h$ intersects~$P^m_{p,q}$.
    Note that for some $m\in[2]$, $\abs{\mathcal{Q}_m}\geq k+3$.
    Let $\mathcal{Q}''$ be the set of pairs $(i_h,j_h)\in\mathcal{Q}_m$ such that
    \begin{itemize}
        \item if~$P_{p,q}$ is fragile, then~$A_h$ does not contain~$a^*_{p,q}$, and
        \item if~$P_{p,q}$ is a detached $\mathcal{H}$-earring, then~$A_h$ is disjoint from $B_{P_{p,q}}(Q_p,\ell-1)\setminus V(Q_p)$.
    \end{itemize}
    Since the paths in~$\mathcal{C}_0$ are pairwise at distance at least~$\ell$ in~$\mathcal{H}$, we have $\abs{\mathcal{Q}''}\geq\abs{\mathcal{Q}_m}-2\geq k+1$.
    By \cref{cor:induced1}, $\{P_{i,j}:(i,j)\in\mathcal{Q}''\}$ is an induced packing in~$G$.

    We show that~$\mathcal{C}_1:=\{C_h:(i_h,j_h)\in\mathcal{Q}''\}$ is an induced packing in~$G$.
    Towards a contradiction, suppose that~$G$ has an edge~$ab$ between distinct~$C_h$ and~$C_{h'}$ in~$\mathcal{C}_1$, where $a\in V(C_h)$ and $b\in V(C_{h'})$.
    Since $\{P_{i,j}:(i,j)\in\mathcal{Q}''\}$ is an induced packing in~$G$, $a$ or~$b$, say~$a$, is in~$P_{p,q}$.
    We remark that every $(i,j)\in\mathcal{Q}'$ is lexicographically larger than $(p,q)$, and has non-empty $\Gamma_{i,j}$.

    Suppose that~$ab$ is an edge of~$\mathcal{H}$.
    Note that~$a$ and~$b$ are branch vertices of~$\mathcal{H}$.
    By \cref{cor:short branch}, there is a unique $(x,y)\in\mathcal{P}(\mathcal{H})$ such that $\{a,b\}=\{a_{x,y},b_{x,y}\}$.
    If~$P_{x,y}$ is an $\mathcal{H}$-detour, then $\sigma(a)=\sigma(b)$ and hence~$P_{p,q}$ contains both~$a$ and~$b$.
    By \cref{lem:strict1}, $\dist_{P_{p,q}}(P_{i_h,j_h},P_{i_{h'},j_{h'}})\leq\dist_{P_{p,q}}(a,b)\leq\ell-2$, contradicting that the paths in~$\mathcal{C}_0$ are pairwise at distance at least~$\ell$ in~$\mathcal{H}$.
    Thus, $P_{x,y}$ is a regular $\mathcal{H}$-ear.
    By the construction of~$\mathcal{Q}''$ and \cref{cor:chord}, we have $\abs{E(P_{x,y})}\geq\ell$.
    Thus, $ab$ is an edge of~$H_{x,y-1}$.
    This implies that $\sigma(a)=\sigma(b)$, and hence~$P_{p,q}$ contains both~$a$ and~$b$.
    Then~$ab$ is an edge between~$A_h$ and~$A_{h'}$, a contradiction.
    
    Hence, $ab$ is a chord of~$\mathcal{H}$.
    We divide into two subcases according to whether~$b$ lies in~$P_{p,q}$.

    \medskip
    \noindent
    \textbf{Case 1-1.} $b\in V(P_{p,q})$.

    We apply \cref{prop:induced1} for~$\mathcal{H}$ and~$ab$.
    By the construction of~$\mathcal{Q}''$, neither \cref{prop:induced1}\ref{item:prop1-1} nor \cref{prop:induced1}\ref{item:prop1-3} holds.
    Since~$a$ and~$b$ are admissible vertices of~$H_{p,q}$, \cref{prop:induced1}\ref{item:prop1-2} does not hold.
    Hence, \cref{prop:induced1} does not hold, a contradiction.

    \medskip
    \noindent
    \textbf{Case 1-2.} $b\notin V(P_{p,q})$.
    
    Note that $(p,q)=\sigma(a)<_L\sigma(b)=(i_{h'},j_{h'})$.
    We apply \cref{prop:induced2} for~$\mathcal{H}$ and~$ab$.
    Since~$P_{i_{h'},j_{h'}}$ is not a detached $\mathcal{H}$-earring, \cref{prop:induced2}\ref{item:prop2-2} does not hold.
    Since~$C_h$ and~$C_{h'}$ are vertex-disjoint, \cref{prop:induced2}\ref{item:prop2-3} does not hold.
    Thus, \cref{prop:induced2}\ref{item:prop2-1} holds.
    By the construction of~$\mathcal{Q}''$, there is no $x\in\Gamma_{i_{h'},j_{h'}}$ such that~$a$ and~$x$ are in distinct components of~$P^{\operatorname{side}}_{p,q}$.
    Thus, for some $x\in\Gamma_{i_{h'},j_{h'}}$, both~$a$ and~$b$ are in $B_\mathcal{H}(x,\ell-1)$.
    As $B_{\mathcal{H}-E(P_{i_{h'},j_{h'}})}(x,\ell-1)\subseteq V(P_{p,q})$, we have $\dist_{P_{p,q}}(a,x)\leq\ell-1$.
    Since~$a$ is in~$A_h$, we have $\dist_{P_{p,q}}(x,P_{i_h,j_h})\leq\ell-1$, contradicting that the paths in~$\mathcal{C}_0$ are pairwise at distance at least~$\ell$ in~$\mathcal{H}$.

    \medskip
    Therefore, $\mathcal{C}_1$ is an induced packing in~$G$.
    By \cref{lem:all cycles}, for each $C\in\mathcal{C}_1$, $C^-$ is an $(\ell,S)$-cycle.
    Thus, by \cref{cor:induced2}, $\{C^-:C\in\mathcal{C}_1\}$ is an induced packing in~$G$ of at least~$k$ $(\ell,S)$-cycles.

    \medskip
    \noindent
    \textbf{Case 2.} $(\mathcal{Q},\preceq_0)$ has no chain of size $2k+6$.

    As $\abs{\mathcal{Q}}\geq(2k+5)(26k-19)^2$, by \cref{thm:Dilworth}, $(\mathcal{Q},\preceq_0)$ has an antichain~$\mathcal{Q}'$ of size $(26k-19)^2$.
    Recall that every chain of $(\mathcal{Q},\preceq_1)$ is an antichain of $(\mathcal{Q},\preceq_2)$, and vice versa.
    By \cref{thm:Dilworth}, $(\mathcal{Q}',\preceq_1)$ or $(\mathcal{Q}',\preceq_2)$ has a chain~$\mathcal{Q}''$ of size $26k-19$.
    Let $(i_1,j_1),\ldots,(i_{26k-20},j_{26k-20})$ be distinct pairs in~$\mathcal{Q}''$ such that if~$P_{p,q}$ is a detached $\mathcal{H}$-earring, then $c_{p,q}\notin\Gamma_{i_h,j_h}$ for each $h\in[26k-20]$.
    Without loss of generality, we may assume that $a_{i_1,j_1}\prec\cdots\prec a_{i_{26k-20},j_{26k-20}}$.
    For each $h\in[13k-10]$,
    \begin{itemize}
        \item let~$A_h$ be the subpath of $P'_{p,q}$ between~$a_{i_{2h-1},j_{2h-1}}$ and $a_{i_{2h},j_{2h}}$,
        \item let~$B_h$ be the subpath of $P'_{p,q}$ between~$b_{i_{2h-1},j_{2h-1}}$ and $b_{i_{2h},j_{2h}}$,
        \item let $P_h:=P_{i_{2h-1},j_{2h-1}}\cup P_{i_{2h},j_{2h}}$, and
        \item let~$C_h$ be the cycle obtained by concatenating~$A_h$, $B_h$, and~$P_h$.
    \end{itemize}
    Since~$\mathcal{Q}''$ is a chain of $(\mathcal{Q}',\preceq_1)$ or $(\mathcal{Q}',\preceq_2)$, the cycles in $\{C_h:h\in[13k-10]\}$ are pairwise vertex-disjoint.

    We define two subpaths~$P^1_{p,q}$ and~$P^2_{p,q}$ of~$P_{p,q}$ as follows:
    \begin{itemize}
        \item if~$P_{p,q}$ is fragile, then let~$P^1_{p,q}$ and~$P^2_{p,q}$ be the two subpaths of~$P_{p,q}$ between~$\Gamma_{p,q}$ and~$\{a^*_{p,q},b^*_{p,q}\}$ which are edge-disjoint from~$P^*_{p,q}$, and
        \item otherwise let $P^1_{p,q}:=P'_{p,q}$ and let~$P^2_{p,q}$ be a null graph.
    \end{itemize}
    For each $m\in[2]$, let~$I_m$ be the set of integers $h\in[13k-10]$ such that~$C_h$ is vertex-disjoint from~$P^m_{p,q}$ and if~$P_{p,q}$ is a detached $\mathcal{H}$-earring, then~$C_h$ is disjoint from $B_{P_{p,q}}(Q_p,\ell-1)\setminus V(Q_p)$ as well.
    Let $\mathcal{D}_m:=\{C_h:h\in I_m\}$.
    We will use the following two claims.

    \begin{claim}\label{clm:jump2-1}
        For each $m\in[2]$, the cycles in~$\mathcal{D}_m$ are pairwise at distance at least~$\ell$ in~$\mathcal{H}$.
    \end{claim}
    \begin{subproof}
        Towards a contradiction, suppose that~$\mathcal{D}_m$ contains distinct cycles~$C_h$ and~$C_{h'}$ with $\dist_\mathcal{H}(C_h,C_{h'})\leq\ell-1$.
        Let~$Q$ be a shortest path of~$\mathcal{H}$ between~$C_h$ and~$C_{h'}$ and let~$a$ and~$b$ be the ends of~$Q$ in~$C_h$ and $C_{h'}$, respectively.
        Note that~$a$ and~$b$ are branch vertices of~$\mathcal{H}$.
        By \cref{cor:short branch}, there is a unique $(i,j)\in\mathcal{P}(\mathcal{H})$ such that $\{a,b\}=\{a_{i,j},b_{i,j}\}$ and both~$a_{i,j}$ and~$b_{i,j}$ have degree~$3$ in~$\mathcal{H}$.
        Since the paths in~$\mathcal{C}_0$ are pairwise at distance at least~$\ell$ in~$\mathcal{H}$, $a$ or~$b$ is in~$P_{p,q}$.
        Without loss of generality, we may assume that~$a$ is in~$P_{p,q}$.

        Since~$C_h$ and~$C_{h'}$ are vertex-disjoint, for each $h''\in\{h,h'\}$, we have
        \[
            (i,j)\notin\{(i_{2h''-1},j_{2h''-1}),(i_{2h''},j_{2h''})\}.
        \]
        Thus, $P_{i,j}$ is edge-disjoint from $A_h\cup A_{h'}$.
        As $\sigma(a)=(p,q)$, we have $(i,j)>_L(p,q)$ and hence~$P_{i,j}$ is edge-disjoint from~$P_{p,q}$.
        This implies that $P_{i,j}\cup Q$ is edge-disjoint from $C_h\cup C_{h'}$.
        Since both~$a$ and~$b$ have degree~$2$ in $C_h\cup C_{h'}$ and degree~$3$ in~$\mathcal{H}$, $Q$ uses all edges of~$P_{i,j}$ incident with~$\Gamma_{i,j}$.
        Since~$Q$ is a path of length at most $\ell-1$ between branch vertices of~$\mathcal{H}$, every internal vertex of~$Q$ has degree~$2$ in~$\mathcal{H}$.
        Therefore, $Q=P_{i,j}$.
        However, by the construction of~$\mathcal{D}_m$, \cref{cor:chord} does not hold for~$P_{i,j}$, a contradiction.
    \end{subproof}

    \begin{claim}\label{clm:jump2-2}
        For each $m\in[2]$, $\mathcal{D}_m$ is an induced packing in~$G$.
    \end{claim}
    \begin{subproof}
        Towards a contradiction, suppose that~$G$ has an edge~$ab$ between distinct~$C_h$ and~$C_{h'}$ in~$\mathcal{D}_m$, where $a\in V(C_h)$ and $b\in V(C_{h'})$.
        By \cref{clm:jump2-1}, $ab$ is a chord of~$\mathcal{H}$.
        By the construction of~$I_m$ and \cref{cor:induced1}, $\{P_{h''}:h''\in I_m\}$ is an induced packing in~$G$.
        Thus, $a$ or~$b$, say~$a$, is in~$P_{p,q}$.

        Suppose that $b\in V(P_{p,q})$.
        We apply \cref{prop:induced1} for~$\mathcal{H}$ and~$ab$.
        By the construction of~$I_m$, neither \cref{prop:induced1}\ref{item:prop1-1} nor \cref{prop:induced1}\ref{item:prop1-3} holds.
        Since~$a$ and~$b$ are admissible vertices of~$H_{p,q}$, \cref{prop:induced1}\ref{item:prop1-2} does not hold.
        Hence, \cref{prop:induced1} does not hold, a contradiction.
        
        Now, suppose that $b\notin V(P_{p,q})$.
        Note that $(p,q)=\sigma(a)<_L\sigma(b)$.
        We apply \cref{prop:induced2} for~$\mathcal{H}$ and~$ab$.
        Let $\sigma(b):=(i'',j'')$.
        Since~$P_{i'',j''}$ is not a detached $\mathcal{H}$-earring, \cref{prop:induced2}\ref{item:prop2-2} does not hold.
        Since~$C_h$ and~$C_{h'}$ are vertex-disjoint, \cref{prop:induced2}\ref{item:prop2-3} does not hold.
        Thus, \cref{prop:induced2}\ref{item:prop2-1} holds.
        By the construction of~$I_m$, there is no $x\in\Gamma_{i'',j''}$ such that~$a$ and~$x$ are in distinct components of~$P^{\operatorname{side}}_{p,q}$.
        Thus, for some $x\in\Gamma_{i'',j''}$, both~$a$ and~$b$ are in $B_\mathcal{H}(x,\ell-1)$.
        As $B_{\mathcal{H}-E(P_{i'',j''})}(x,\ell-1)\subseteq V(P_{p,q})$, we have $\dist_{P_{p,q}}(a,x)\leq\ell-1$.
        Since~$a$ is in~$A_h\cup B_h$, we have $\dist_{P_{p,q}}(x,P_h)\leq\ell-1$, contradicting \cref{clm:jump2-1}.
    \end{subproof}

    For each $m\in[2]$ and each $C\in\mathcal{D}_m$, by \cref{lem:all cycles}, $C^-$ is an $(\ell,S)$-cycle.
    Thus, by \cref{clm:jump2-2} and \cref{cor:induced2}, $\{C^-:C\in\mathcal{D}_m\}$ is an induced packing in~$G$ of at least~$k$ $(\ell,S)$-cycles.
    Hence, we may assume that each~$I_m$ is of size at most $k-1$.
    Then for some $h\in[13k-10]$, $C_h$ intersects both~$P^1_{p,q}$ and~$P^2_{p,q}$.
    This implies that~$P^2_{p,q}$ is not a null graph, and therefore~$P_{p,q}$ is fragile.
    Let~$I$ be the set of integers $h\in[13k-10]\setminus(I_1\cup I_2)$ such that~$C_h$ contains neither~$a^*_{p,q}$ nor~$b^*_{p,q}$.
    Note that
    \[
        \eta:=\abs{I}\geq(13k-10)-2(k-1)-2=11k-10.
    \]
    We remark that for each $h\in I$, $A_h$ and~$B_h$ are subpaths of~$P^1_{p,q}$ and~$P^2_{p,q}$, respectively.
    Let
    \[
        \tau:=A^I_1,A^I_2,\ldots,A^I_\eta,B^I_\eta,\ldots,B^I_2,B^I_1
    \]
    be the linear order on the paths~$A_h$ and~$B_h$ for $h\in I$ arranged so that their ends form an increasing sequence with respect to the linear order~$\preceq$.
    Let $\mathcal{A}:=\{A^I_h:h\in[\eta]\}$ and $\mathcal{B}:=\{B^I_h:h\in[\eta]\}$.
    Let~$H$ be the bipartite graph with bipartition $(\mathcal{A},\mathcal{B})$ such that for all $h,h'\in[\eta]$, $A^I_h$ and~$B^I_{h'}$ are adjacent if and only if~$\mathcal{H}$ has a chord~$xx'$ such that either
    \begin{itemize}
        \item $xx'$ is between~$A^I_h$ and~$B^I_{h'}$, or
        \item $x\in V(A^I_h)$ and $x'\in\{b'_{i_{2h'-1},j_{2h'-1}},b'_{i_{2h'},j_{2h'}}\}$, or
        \item $x\in V(B^I_{h'})$ and $x'\in\{a'_{i_{2h-1},j_{2h-1}},a'_{i_{2h},j_{2h}}\}$.
    \end{itemize}
    We will use the following two claims.

    \begin{claim}\label{clm:jump2-3}
        For every increasing subsequence $A^I_g,A^I_h,B^I_{h'},B^I_{g'}$ of~$\tau$, at most one of~$A^I_gB^I_{h'}$ and~$A^I_hB^I_{g'}$ is an edge of~$H$.
    \end{claim}
    \begin{subproof}
        Suppose not.
        Then~$G_{p,q}$ has $H_{p,q}$-paths~$R_1$ and~$R_2$ of length at most~$2$ such that~$R_1$ is between~$A^I_g$ and~$B^I_{h'}$ and $R_2$ is between~$A^I_h$ and~$B^I_{g'}$.
        For each $m\in[2]$, let~$a_m$ and~$b_m$ be the ends of~$R_m$.
        Without loss of generality, we may assume that $a_1,a_2\in V(P_{p,q}^1)$ and $b_1,b_2\in V(P_{p,q}^2)$.
        Let~$P$ be the graph obtained from $P_{p,q}\cup R_1\cup R_2$ by removing all internal vertices of $a_1P'_{p,q}a_2$ and $b_1P'_{p,q}b_2$.
        Since~$P_{p,q}$ is fragile, it is either an attached $\mathcal{H}$-earring or a regular $\mathcal{H}$-ear.
        Since~$R_1$ and~$R_2$ are $H_{p,q}$-paths in~$G_{p,q}$,
        \begin{itemize}
            \item if~$P_{p,q}$ is an attached $\mathcal{H}$-earring, then~$P$ is a cycle of~$G_{p-1,\mu(p-1)}$ intersecting~$H_{p-1,\mu(p-1)}$ only at~$c_p$, and
            \item otherwise~$P$ is an $H_{p,q-1}$-path in~$G_{p,q-1}$.
        \end{itemize}
        Since the paths in~$\mathcal{C}_0$ are pairwise at distance at least~$\ell$ in~$\mathcal{H}$, we have
        \[
            \min\{\dist_\mathcal{H}(a_1,a_2),\dist_\mathcal{H}(b_1,b_2)\}\geq\ell>\max\{\abs{E(P_1)},\abs{E(P_2)}\}.
        \]
        Thus, $P$ is shorter than~$P_{p,q}$.
        Since~$a_1$ and~$b_2$ are admissible vertices of~$H_{p,q}$, we have
        \[
            \abs{E(P)}\geq\dist_{P_{p,q}}(a_1,\Gamma_{p,q})+\dist_{P_{p,q}}(b_2,\Gamma_{p,q})+\dist_{P_{p,q}}(a_2,b_1)+2\geq2\ell+3.
        \]
        Note that~$P$ contains a vertex in~$S$ as~$P^*_{p,q}$ is its subpath.
        Therefore, $P$ is a better choice than~$P_{p,q}$, contradicting~\ref{item:ear1} or~\ref{item:ear3}.
    \end{subproof}
    
    \begin{claim}\label{clm:jump2-4}
        If $\abs{V_{\geq3}(H)}\geq6k-5$, then~$G$ has an induced packing of~$k$ $(\ell,S)$-cycles.
    \end{claim}
    \begin{subproof}
        Note that $\mathcal{A}\cap V_{\geq3}(H)$ or $\mathcal{B}\cap V_{\geq3}(H)$ is of size at least $3k-2$.
        By symmetry, we may assume that $\abs{\mathcal{A}\cap V_{\geq3}(H)}\geq3k-2$.
        Let $A^*_1,\ldots,A^*_{3k-2}$ be distinct paths in $\mathcal{A}\cap V_{\geq3}(H)$ which form an increasing subsequence of~$\tau$.
        For each $i\in[3k-2]$, let $B^*_{i,1}$, $B^*_{i,2}$, and $B^*_{i,3}$ denote the three neighbours of~$A^*_i$ in~$H$ that appear first, second, and last in~$\tau$, respectively.
        By \cref{clm:jump2-3}, the sequence
        \[
            B^*_{3k-2,1},\,B^*_{3k-2,2},\,B^*_{3k-2,3},\,
            B^*_{3k-3,1},\,B^*_{3k-3,2},\,B^*_{3k-3,3},\,
            \ldots,\,B^*_{1,1},\,B^*_{1,2},\,B^*_{1,3}
        \]
        is an increasing subsequence of~$\tau$.
    
        Let~$J$ be the set of integers $h\in[\eta]$ such that $B^I_h=B^*_{i,2}$ for some $i\in[3k-2]$.
        By \cref{clm:jump2-3}, the subgraph of~$H$ induced by $\{A^I_h:h\in J\}\cup\{B^I_h:h\in J\}$ has maximum degree at most~$1$.
        Let~$H_J$ be the graph with vertex set $\{A^I_h\cup B^I_h:h\in J\}$ such that distinct ${A^I_h\cup B^I_h}$ and ${A^I_{h'}\cup B^I_{h'}}$ are adjacent if and only if one of~$A^I_hB^I_{h'}$ and~$A^I_{h'}B^I_h$ is an edge of~$H$.
        Note that~$H_J$ has maximum degree at most~$2$.
        As $\abs{J}=3k-2$, $J$ has a subset~$J'$ of size~$k$ such that $\{A^I_h\cup B^I_h:h\in J'\}$ is an independent set of~$H_J$.

        We show that $\mathcal{C}_1:=\{C_h:h\in J'\}$ is an induced packing in~$G$.
        Towards a contradiction, suppose that~$G$ has an edge~$ab$ between distinct~$C_h$ and~$C_{h'}$ in~$\mathcal{C}_1$, where $a\in V(C_h)$ and $b\in V(C_{h'})$.
        By the construction of~$J'$ and \cref{cor:induced1}, $\{P_{h''}:h''\in J'\}$ is an induced packing in~$G$.
        Thus, $a$ or~$b$, say~$a$, is in~$P_{p,q}$.
        We remark that every $(i,j)\in\mathcal{Q}''$ is lexicographically larger than $(p,q)$, and has non-empty~$\Gamma_{i,j}$.

        Suppose that~$ab$ is an edge of~$\mathcal{H}$.
        Note that~$a$ and~$b$ are branch vertices of~$\mathcal{H}$.
        By \cref{cor:short branch}, there is a unique $(x,y)\in\mathcal{P}(\mathcal{H})$ such that $\{a,b\}=\{a_{x,y},b_{x,y}\}$.
        If~$P_{x,y}$ is an $\mathcal{H}$-detour, then $\sigma(a)=\sigma(b)$ and hence~$P_{p,q}$ contains both~$a$ and~$b$.
        By \cref{lem:strict1}, $\dist_{P_{p,q}}(P_h,P_{h'})\leq\dist_{P_{p,q}}(a,b)\leq\ell-2$, contradicting that the paths in~$\mathcal{C}_0$ are pairwise at distance at least~$\ell$ in~$\mathcal{H}$.
        Thus, $P_{x,y}$ is a regular $\mathcal{H}$-ear.
        By the construction of~$J'$ and \cref{cor:chord}, we have $\abs{E(P_{x,y})}\geq\ell$.
        Thus, $ab$ is an edge of~$H_{x,y-1}$.
        This implies that $\sigma(a)=\sigma(b)$, and hence~$P_{p,q}$ contains both~$a$ and~$b$.
        Then~$ab$ is an edge between~$P_h$ and~$P_{h'}$, a contradiction.

        Hence, $ab$ is a chord of~$\mathcal{H}$.
        We divide into two subcases according to whether~$b$ lies in~$P_{p,q}$.

        \medskip
        \noindent
        \textbf{Case I.} $b\in V(P_{p,q})$.

        We apply \cref{prop:induced1} for~$\mathcal{H}$ and~$ab$.
        Since~$P_{p,q}$ is fragile, \cref{prop:induced1}\ref{item:prop1-1} does not hold.
        Since~$a$ and~$b$ are admissible vertices of~$H_{p,q}$, \cref{prop:induced1}\ref{item:prop1-2} does not hold.
        By the construction of~$J'$, \cref{prop:induced1}\ref{item:prop1-3} does not hold.
        Hence, \cref{prop:induced1} does not hold, a contradiction.

        \medskip
        \noindent
        \textbf{Case II.} $b\notin V(P_{p,q})$.
    
        Note that $(p,q)=\sigma(a)<_L\sigma(b)$.
        Let $\sigma(b):=(i'',j'')$.
        We apply \cref{prop:induced2} for~$\mathcal{H}$ and~$ab$.
        Since~$P_{i'',j''}$ is not a detached $\mathcal{H}$-earring, \cref{prop:induced2}\ref{item:prop2-2} does not hold.
        Since~$C_h$ and~$C_{h'}$ are vertex-disjoint, \cref{prop:induced2}\ref{item:prop2-3} does not hold.
        Thus, \cref{prop:induced2}\ref{item:prop2-1} holds.
        By the construction of~$J'$, there is no $x\in\Gamma_{i'',j''}$ such that~$a$ and~$x$ are in distinct components of~$P^{\operatorname{side}}_{p,q}$.
        Thus, for some $x\in\Gamma_{i'',j''}$, both~$a$ and~$b$ are in $B_\mathcal{H}(x,\ell-1)$.
        As $B_{\mathcal{H}-E(P_{i'',j''})}(x,\ell-1)\subseteq V(P_{p,q})$, we have $\dist_{P_{p,q}}(a,x)\leq\ell-1$.
        Since~$a$ is in~$A_h\cup B_h$, we have $\dist_{P_{p,q}}(x,P_h)\leq\ell-1$, contradicting that the paths in~$\mathcal{C}_0$ are pairwise at distance at least~$\ell$ in~$\mathcal{H}$.

        \medskip
        Therefore, $\mathcal{C}_1$ is an induced packing in~$G$.
        By \cref{lem:all cycles}, for each $C\in\mathcal{C}_1$, $C^-$ is an $(\ell,S)$-cycle.
        Thus, by \cref{cor:induced2}, $\{C^-:C\in\mathcal{C}_1\}$ is an induced packing in~$G$ of at least~$k$ $(\ell,S)$-cycles.
    \end{subproof}

    By \cref{clm:jump2-4}, we may assume that $\abs{V_{\geq3}(H)}\leq6(k-1)$.
    Let
    \[
        L:=\{h\in[\eta]:\max\{\deg_H(A^I_h),\deg_H(B^I_h)\}\leq2\}.
    \]
    As $\eta\geq11k-10$, we have $\abs{L}\geq5k-4$.
    Let~$H_L$ be the graph with vertex set $\{A^I_h\cup B^I_h:h\in L\}$ such that distinct ${A^I_h\cup B^I_h}$ and ${A^I_{h'}\cup B^I_{h'}}$ are adjacent if and only if $A^I_hB^I_{h'}$ or $A^I_{h'}B^I_h$ is an edge of~$H$.
    Note that~$H_L$ has maximum degree at most~$4$.
    As $\abs{L}\geq5k-4$, $L$ has a subset~$L'$ of size~$k$ such that $\{A^I_h\cup B^I_h:h\in L'\}$ is an independent set of~$H_L$.

    We show that $\mathcal{C}_1:=\{C_h:h\in L'\}$ is an induced packing in~$G$.
    Towards a contradiction, suppose that~$G$ has an edge~$ab$ between distinct~$C_h$ and~$C_{h'}$ in~$\mathcal{C}_1$, where $a\in V(C_h)$ and $b\in V(C_{h'})$.
    Without loss of generality, we may assume that $h=1$ and $h'=2$. 
    By the construction of~$L'$ and \cref{cor:induced1}, $\{P_{h''}:h''\in L'\}$ is an induced packing in~$G$.
    Thus, $a$ or~$b$ is in~$P_{p,q}$.
    We remark that every $(i,j)\in\mathcal{Q}''$ is lexicographically larger than $(p,q)$, and has non-empty~$\Gamma_{i,j}$.

    Suppose that~$ab$ is an edge of~$\mathcal{H}$.
    Note that~$a$ and~$b$ are branch vertices of~$\mathcal{H}$.
    By \cref{cor:short branch}, there is a unique $(x,y)\in\mathcal{P}(\mathcal{H})$ such that $\{a,b\}=\{a_{x,y},b_{x,y}\}$.
    If~$P_{x,y}$ is an $\mathcal{H}$-detour, then $\sigma(a)=\sigma(b)$ and hence~$P_{p,q}$ contains both~$a$ and~$b$.
    By \cref{lem:strict1}, $\dist_{P_{p,q}}(P_1,P_2)\leq\dist_{P_{p,q}}(a,b)\leq\ell-2$, contradicting that the paths in~$\mathcal{C}_0$ are pairwise at distance at least~$\ell$ in~$\mathcal{H}$.
    Thus, $P_{x,y}$ is a regular $\mathcal{H}$-ear.
    By the construction of~$L'$ and \cref{cor:chord}, we have $\abs{E(P_{x,y})}\geq\ell$.
    Thus, $ab$ is an edge of~$H_{x,y-1}$.
    This implies that $\sigma(a)=\sigma(b)$, and hence~$P_{p,q}$ contains both~$a$ and~$b$.
    Then~$ab$ is an edge between~$P_1$ and~$P_2$, a contradiction.

    Hence, $ab$ is a chord of~$\mathcal{H}$.
    We apply \cref{cor:induced3} for~$\mathcal{H}$, $\{C_1,C_2\}$, and~$ab$.
    Without loss of generality, we may assume that $\sigma(a)\leq_L\sigma(b)$.
    Suppose that~\ref{item:packing6} holds.
    Then there is $(i,j)\in\mathcal{P}(\mathcal{H})$ with $j\geq2$ such that~$H_{i,j}$ has an $(a_{i,j},b_{i,j})$-path~$Q$ of length at most $\ell-1$ and for each $h''\in[2]$, there is $x_{h''}\in\Gamma_{i,j}$ with $B_{\mathcal{H}-E(Q)}(x_{h''},\ell-1)\subseteq V(C_{h''})$.
    By the construction of~$L'$ and \cref{cor:chord}, we have $\abs{E(P_{i,j})}\geq\ell$.
    Thus, $\dist_{H_{i,j-1}}(a_{i,j},b_{i,j})\leq\ell-1$.
    Since~$a_{i,j}$ and~$b_{i,j}$ are admissible vertices of~$H_{i,j-1}$, by \cref{lem:nearby} for $\mathcal{H}:=H_{i,j-1}$, we have $\sigma(a_{i,j})=\sigma(b_{i,j})$.
    Since~$C_1$ and~$C_2$ are vertex-disjoint, $P_{p,q}$ contains both~$a_{i,j}$ and~$b_{i,j}$.
    Note that $a_{i,j}P_{p,q}b_{i,j}$ is the unique shortest $(a_{i,j},b_{i,j})$-path of~$H_{i,j-1}$.
    Thus, $\dist_{P_{p,q}}(P_1,P_2)\leq\dist_{P_{p,q}}(a_{i,j},b_{i,j})\leq\ell-1$, a contradiction.
    
    Hence, \ref{item:packing6} does not hold.
    We divide into two subcases according to whether $\sigma(a)=\sigma(b)$.

    \medskip
    \noindent
    \textbf{Case 2-1.} $\sigma(a)=\sigma(b)$.

    Since~$C_1$ and~$C_2$ are vertex-disjoint, $P_{p,q}$ contains both~$a$ and~$b$.
    Since~$P_{p,q}$ is fragile, \ref{item:packing1} does not hold.
    By the construction of~$L'$, \ref{item:packing2} does not hold.
    Hence, \cref{cor:induced3} does not hold, a contradiction.

    \medskip
    \noindent
    \textbf{Case 2-2.} $\sigma(a)\neq\sigma(b)$.

    Note that $(p,q)=\sigma(a)<_L\sigma(b)$.
    Let $\sigma(b):=(i',j')$.
    By the construction of~$L'$, \ref{item:packing3} does not hold.
    Since~$P_{i',j'}$ is not a detached $\mathcal{H}$-earring, \ref{item:packing4} does not hold.
    Since~$C_1$ and~$C_2$ are vertex-disjoint, \ref{item:packing5} does not hold.
    Hence, \cref{cor:induced3} does not hold, a contradiction.

    \medskip
    Therefore, $\mathcal{C}_1$ is an induced packing in~$G$.
    By \cref{lem:all cycles}, for each $C\in\mathcal{C}_1$, $C^-$ is an $(\ell,S)$-cycle.
    Thus, by \cref{cor:induced2}, $\{C^-:C\in\mathcal{C}_1\}$ is an induced packing in~$G$ of at least~$k$ $(\ell,S)$-cycles.

    \medskip
    This completes the proof.
\end{proof}

For a pair $(p,q)\in\mathcal{P}(\mathcal{H})$, consider any set~$\mathcal{Q}$ of pairs $(i,j)\in\mathcal{P}(\mathcal{H})$ with $\abs{\Sigma_{i,j}}=2$ such that $\Sigma_{i,j}\cap\Sigma_{i',j'}\subseteq\{(p,q)\}$ for all distinct $(i,j),(i',j')\in\mathcal{Q}$.
We show that if~$\abs{\mathcal{Q}}\geq6k$, then~$\mathcal{Q}$ has a subset~$\mathcal{Q}'$ of size~$k$ such that $\{P_{i,j}:(i,j)\in\mathcal{Q}'\}$ is an induced packing in~$G$.

\begin{lemma}\label{lem:jump3}
    Let $(G,S)$ be a good pair and let~$\mathcal{H}$ be a maximal $(\ell,S)$-frame in~$G$.
    For a pair $(p,q)\in\mathcal{P}(\mathcal{H})$, let~$\mathcal{Q}$ be a set of pairs $(i,j)\in\mathcal{P}(\mathcal{H})$ with $\abs{\Sigma_{i,j}}=2$ such that $\Sigma_{i,j}\cap\Sigma_{i',j'}\subseteq\{(p,q)\}$ for all distinct $(i,j),(i',j')\in\mathcal{Q}$.
    If~$\abs{\mathcal{Q}}\geq6k$ for some positive integer~$k$, then~$\mathcal{Q}$ has a subset~$\mathcal{Q}'$ of size~$k$ such that $\{P_{i,j}:(i,j)\in\mathcal{Q}'\}$ is an induced packing in~$G$.
\end{lemma}
\begin{proof}
    For every $(i,j)\in\mathcal{Q}$, as $\abs{\Sigma_{i,j}}=2$, $P_{i,j}$ is a regular $\mathcal{H}$-ear.
    Without loss of generality, we may assume that $a_{i,j}\in V(P_{p,q})$.
    We denote by~$a'_{i,j}$ and~$b'_{i,j}$ the neighbours of~$a_{i,j}$ and~$b_{i,j}$ in~$P_{i,j}$, respectively.
    Let $\mathcal{C}_0:=\{P_{i,j}:(i,j)\in\mathcal{Q}\}$.
    By \cref{lem:degree}, the paths in~$\mathcal{C}_0$ are pairwise vertex-disjoint.
    In addition, by \cref{lem:argue distance}, they are pairwise at distance at least~$\ell$ in~$\mathcal{H}$.
    This obviously implies that~$\mathcal{C}_0$ is an induced packing in~$\mathcal{H}$.

    We define a subpath~$P'_{p,q}$ of~$P_{p,q}$ as follows: if~$P_{p,q}$ is a detached $\mathcal{H}$-earring, then for an arbitrary vertex~$c_{p,q}$ of~$P_{p,q}$, let $P'_{p,q}:=P_{p,q}-c_{p,q}$, and otherwise let $P'_{p,q}:=P_{p,q}-\Gamma_{p,q}$.
    Moreover, we define two subpaths~$P^1_{p,q}$ and~$P^2_{p,q}$ of~$P_{p,q}$ as follows:
    \begin{itemize}
        \item if~$P_{p,q}$ is fragile, then let~$P^1_{p,q}$ and~$P^2_{p,q}$ be the two subpaths of~$P_{p,q}$ between~$a^*_{p,q}$ and~$\Gamma_{p,q}$, and
        \item otherwise let $P^1_{p,q}:=P'_{p,q}$ and let~$P^2_{p,q}$ be a null graph.
    \end{itemize}
    For each $m\in[2]$, let~$\mathcal{Q}_m$ be the set of pairs $(i,j)\in\mathcal{Q}$ such that $a_{i,j}\in V(P^m_{p,q})$.
    We have
    \[
        \abs{\mathcal{Q}_1}+\abs{\mathcal{Q}_2}\geq\abs{\mathcal{Q}}-1\geq6k-1.
    \]
    Thus, for some $m\in[2]$, $\abs{\mathcal{Q}_m}\geq3k$.
    Let~$\mathcal{Q}'_m$ be the set of pairs $(i,j)\in\mathcal{Q}_m$ such that if~$P_{p,q}$ is a detached $\mathcal{H}$-earring, then~$P_{i,j}$ is disjoint from $B_{P_{p,q}}(Q_p,\ell-1)\setminus V(Q_p)$.
    Since the paths in~$\mathcal{C}_0$ are pairwise at distance at least~$\ell$ in~$\mathcal{H}$, we have $\abs{\mathcal{Q}'_m}\geq\abs{\mathcal{Q}_m}-2\geq3k-2$.

    Let $\preceq$ be the linear order on the vertices of~$P'_{p,q}$ induced by the order in which vertices appear along the path.
    Let~$H$ be the graph with vertex set $\{w_{i,j}:(i,j)\in\mathcal{Q}'_m\}$ such that~$w_{i,j}$ and~$w_{i',j'}$ with $(i,j)<_L(i',j')$ are adjacent if and only if~$P_{i,j}$ contains a vertex in~$\Gamma_{i'',j''}$ for $(i'',j''):=\sigma(b_{i',j'})$.
    Note that~$H$ is $2$-degenerate.
    As $\abs{\mathcal{Q}'_m}\geq3k-2$, $H$ has an independent set~$I$ of size~$k$.
    
    We show that~$\mathcal{C}_1:=\{P_{i,j}:w_{i,j}\in I\}$ is an induced packing in~$G$.
    Towards a contradiction, suppose that~$G$ has an edge~$ab$ between distinct~$P_{i,j}$ and~$P_{i',j'}$ in~$\mathcal{C}_1$, where $a\in V(P_{i,j})$ and $b\in V(P_{i',j'})$.
    Since~$\mathcal{C}_0$ is an induced packing in~$\mathcal{H}$, $ab$ is a chord of~$\mathcal{H}$.
    We remark that every $(i'',j'')\in\mathcal{Q}'_m$ is lexicographically larger than $(p,q)$, and has non-empty~$\Gamma_{i'',j''}$.
    
    Suppose that $\sigma(a)=\sigma(b)$.
    Since~$P_{i,j}$ and~$P_{i',j'}$ are vertex-disjoint, $P_{p,q}$ contains both~$a$ and~$b$.
    We apply \cref{prop:induced1} for~$\mathcal{H}$ and~$ab$.
    By the construction of~$\mathcal{Q}'_m$, neither \cref{prop:induced1}\ref{item:prop1-1} nor \cref{prop:induced1}\ref{item:prop1-3} holds.
    Since~$a$ and~$b$ are admissible vertices of~$H_{p,q}$, \cref{prop:induced1}\ref{item:prop1-2} does not hold.
    Hence, \cref{prop:induced1} does not hold, a contradiction.

    Now, suppose that $\sigma(a)\neq\sigma(b)$.
    Without loss of generality, we may assume that $(i_1,j_1):=\sigma(a)<_L\sigma(b)=:(i_2,j_2)$.
    We apply \cref{prop:induced2} for~$\mathcal{H}$ and~$ab$.
    We divide into two cases according to whether $\sigma(b)=(i',j')$.

    \medskip
    \noindent
    \textbf{Case 1.} $\sigma(b)=(i',j')$.
    
    Since~$P_{i',j'}$ is not a detached $\mathcal{H}$-earring, \cref{prop:induced2}\ref{item:prop2-2} does not hold.
    Since~$P_{i,j}$ and~$P_{i',j'}$ are vertex-disjoint, \cref{prop:induced2}\ref{item:prop2-3} does not hold, and therefore \cref{prop:induced2}\ref{item:prop2-1} holds.
    As $\dist_{\mathcal{H}}(P_{i,j},P_{i',j'})\geq\ell$, we have $a\notin B_\mathcal{H}(\Gamma_{i',j'},\ell-1)$.
    Hence, $P_{i_1,j_1}$ is fragile and there is $x\in\Gamma_{i',j'}$ such that~$a$ and~$x$ are in distinct components of~$P^{\operatorname{side}}_{i_1,j_1}$.
    Since~$P_{i,j}$ and~$P_{i',j'}$ are vertex-disjoint, we have $\sigma(a)\neq(i,j)$.
    By the construction of~$\mathcal{Q}'_m$, $\sigma(a)\neq(p,q)$.
    Then $\Sigma_{i,j}=\Sigma_{i',j'}=\{(p,q),(i_1,j_1)\}$, a contradiction.
    
    \medskip
    \noindent
    \textbf{Case 2.} $\sigma(b)\neq(i',j')$.

    Note that $b\in\Gamma_{i',j'}$.
    Since~$b$ is an admissible vertex of~$H_{i_2,j_2}$, \cref{prop:induced2}\ref{item:prop2-1} does not hold.
    Suppose that \cref{prop:induced2}\ref{item:prop2-2} holds.
    Then~$b$ is not an admissible vertex of~$H_{i_2,2}$.
    Since~$b$ is an admissible vertex of~$H_{i_2,1}$, we have $(i_2,2)=(i',j')$.
    Let~$x$ be the vertex in $\Gamma_{i_2,2}\setminus\{b\}$.
    As $\dist_\mathcal{H}(P_{i,j},P_{i',j'})\geq\ell$, we have $a\notin B_\mathcal{H}(x,\ell-1)$.
    Hence, $P_{i_1,j_1}$ is fragile and~$a$ and~$x$ are in distinct components of~$P^{\operatorname{side}}_{i_1,j_1}$.
    Since~$P_{i,j}$ and~$P_{i',j'}$ are vertex-disjoint, we have $\sigma(a)\neq(i,j)$.
    Thus, $a\in\Gamma_{i,j}$.
    As $\Sigma_{i,j}\cap\Sigma_{i',j'}\subseteq\{(p,q)\}$, we have $\sigma(a)=\sigma(x)=(p,q)$, contradicting the construction of~$\mathcal{Q}'_m$.
    Thus, \cref{prop:induced2}\ref{item:prop2-3} holds.
    Then $a\in\Gamma_{i_2,j_2}$.
    Since~$I$ is an independent set of~$H$, we have $b\neq b_{i',j'}$ and hence $b=a_{i',j'}$ and $\sigma(b)=(p,q)$.
    Thus, $a\in\Gamma_{p,q}$, contradicting that~$P_{i,j}$ is disjoint from~$Y_{p,q}$.

    \medskip
    This completes the proof.
\end{proof}

We now construct two auxiliary graphs from~$\mathcal{H}$, and show that if any of them has many vertices, then~$G$ has an induced packing of~$k$ $(\ell,S)$-cycles.
We construct the first auxiliary graph as follows.
Let $\mathcal{P}_1(\mathcal{H})$ be the set of pairs $(i,j)\in\mathcal{P}(\mathcal{H})$ such that~$P_{i,j}$ is either
\begin{itemize}
    \item an $\mathcal{H}$-earring, or
    \item a regular $\mathcal{H}$-ear with $\dist_{H_{i,j-1}}(a_{i,j},b_{i,j})\leq\ell$.
\end{itemize}
For each $(i,j)\in\mathcal{P}_1(\mathcal{H})$, we define a path~$Q_{i,j}$ as follows:
\begin{itemize}
    \item if~$P_{i,j}$ is a detached $\mathcal{H}$-earring, then~$Q_{i,j}$ is a null graph,
    \item if~$P_{i,j}$ is an attached $\mathcal{H}$-earring, then~$Q_{i,j}$ is a trivial path at~$c_i$, and
    \item otherwise~$Q_{i,j}$ is a shortest $(a_{i,j},b_{i,j})$-path of~$H_{i,j-1}$.
\end{itemize}
For every $(i,j)\in\mathcal{P}_1(\mathcal{H})$ with $j\geq2$, since~$a_{i,j}$ and~$b_{i,j}$ are admissible vertices of~$H_{i,j-1}$, by \cref{lem:nearby} for $\mathcal{H}:=H_{i,j-1}$, we have $\sigma(a_{i,j})=\sigma(b_{i,j})$ and~$Q_{i,j}$ is a unique shortest $(a_{i,j},b_{i,j})$-path of~$H_{i,j-1}$.
We remark that every internal vertex of~$Q_{i,j}$ has degree~$2$ in~$\mathcal{H}$ as~$Q_{i,j}$ is a path of length at most~$\ell$ between branch vertices of~$\mathcal{H}$.
This implies that~$Q_{i,j}$ is a path of~$\mathcal{H}^-$.
Let~$A_1(\mathcal{H})$ be the graph with vertex set $\{w_{i,j}:(i,j)\in\mathcal{P}_1(\mathcal{H})\}$ such that~$w_{i,j}$ and~$w_{i',j'}$ with $(i,j)>_L(i',j')$ are adjacent if and only if either $(i',j')\in\Sigma_{i,j}$ or $\mu(i)\geq2$ and $\Sigma_{i,2}=\{(i,j),(i',j')\}$.
Note that~$A_1(\mathcal{H})$ is $2$-degenerate, because $\abs{\Sigma_{i,j}}\leq1$ for every $(i,j)\in\mathcal{P}_1(\mathcal{H})$.

We show that if $\abs{\mathcal{P}_1(\mathcal{H})}\geq3k(6k-5)$, then~$G$ has an induced packing of~$k$ $(\ell,S)$-cycles.

\begin{lemma}\label{lem:aux2}
    Let $(G,S)$ be a good pair and let~$\mathcal{H}$ be a maximal $(\ell,S)$-frame in~$G$.
    If $\abs{\mathcal{P}_1(\mathcal{H})}\geq3k(6k-5)$ for some positive integer~$k$, then~$G$ has an induced packing of~$k$ $(\ell,S)$-cycles.
\end{lemma}
\begin{proof}
    Since~$A_1(\mathcal{H})$ is $2$-degenerate and has at least $3k(6k-5)$ vertices, it has an independent set~$I$ of size $k(6k-5)$.
    For each $w_{i,j}\in I$, let ${C_{i,j}:=P_{i,j}\cup Q_{i,j}}$ which is a cycle of~$H^-_{i,j}$.
    By \cref{lem:all cycles}, $C_{i,j}$ is an $(\ell,S)$-cycle.
    Let $\mathcal{C}_0:=\{C_{i,j}:w_{i,j}\in I\}$.
    We will use the following four claims.

    \begin{claim}\label{clm:aux2-1}
        The cycles in~$\mathcal{C}_0$ are pairwise at distance at least~$\ell$ in~$\mathcal{H}$.
    \end{claim}
    \begin{subproof}
        Let~$C_{i,j}$ and~$C_{i',j'}$ be cycles in~$\mathcal{C}_0$ with $(i,j)>_L(i',j')$.
        We first show that~$C_{i,j}$ and~$C_{i',j'}$ are vertex-disjoint.
        Suppose not.
        Let~$z$ be a vertex in $V(C_{i,j}\cap C_{i',j'})$.
        Since $V(P_{i,j})\setminus\Gamma_{i,j}$ is disjoint from $V(H_{i',j'})$, $z$ is in~$Q_{i,j}$.
        Thus, $z$ is an admissible vertex of~$H_{i,j-1}$.
        Then every vertex in $B_{H_{i,j-1}}(z,\ell-1)$ has degree~$2$ in~$H_{i,j-1}$.
        Thus, $z$ is not in~$Q_{i',j'}$ and hence is an internal vertex of~$P_{i',j'}$.
        Then~$Q_{i,j}$ is a subpath of $P_{i',j'}-\Gamma_{i',j'}$.
        Thus, $(i',j')\in\Sigma_{i,j}$, contradicting that~$w_{i,j}$ and~$w_{i',j'}$ are not adjacent in $A_1(\mathcal{H})$.

        We now show that $\dist_\mathcal{H}(C_{i,j},C_{i',j'})\geq\ell$.
        Since every internal vertex of~$Q_{i,j}$ and~$Q_{i',j'}$ has degree~$2$ in~$\mathcal{H}$, every shortest path of~$\mathcal{H}$ between~$C_{i,j}$ and~$C_{i',j'}$ is a $(P_{i,j},P_{i',j'})$-path.
        Since~$I$ is an independent set of~$A_1(\mathcal{H})$, by \cref{lem:argue distance}, we have $\dist_\mathcal{H}(C_{i,j},C_{i',j'})=\dist_\mathcal{H}(P_{i,j},P_{i',j'})\geq\ell$.
    \end{subproof}
    
    \begin{claim}\label{clm:aux2-2}
        If~$\mathcal{H}$ has a chord~$ab$ between distinct~$C_{i,j}$ and~$C_{i',j'}$ in~$\mathcal{C}_0$, then~$a$ or~$b$ is in $V(Q_{i,j}\cup Q_{i',j'})$.
    \end{claim}
    \begin{subproof}
        Suppose not.
        Without loss of generality, we may assume that $a\in V(P_{i,j})$ and $b\in V(P_{i',j'})$.
        Note that $a\notin\Gamma_{i,j}$ and $b\notin\Gamma_{i',j'}$.
        Thus, $\sigma(a)=(i,j)$ and $\sigma(b)=(i',j')$.
        We apply \cref{cor:induced3} for~$\mathcal{H}$, $\{C_1,C_2\}:=\{C_{i,j},C_{i',j'}\}$, and~$ab$.
        As $\sigma(a)\neq\sigma(b)$, neither~\ref{item:packing1} nor~\ref{item:packing2} holds.
        Since~$P_{i,j}$ and~$P_{i',j'}$ are vertex-disjoint, neither~\ref{item:packing3} nor~\ref{item:packing5} holds.
        Since~$I$ is an independent set of~$A_1(\mathcal{H})$, \ref{item:packing4} does not hold.
        By \cref{clm:aux2-1}, \ref{item:packing6} does not hold.
        Hence, \cref{cor:induced3} does not hold, a contradiction.
    \end{subproof}

    \begin{claim}\label{clm:aux2-3}
        The set of~$C_{i,j}\in\mathcal{C}_0$ with $\Gamma_{i,j}=\emptyset$ is an induced packing in~$G$.
    \end{claim}
    \begin{subproof}
        Towards a contradiction, suppose that~$G$ has an edge~$ab$ between distinct~$C_{i,j}$ and~$C_{i',j'}$ in~$\mathcal{C}_0$ with $\Gamma_{i,j}=\Gamma_{i',j'}=\emptyset$.
        By \cref{clm:aux2-1}, $ab$ is a chord of~$\mathcal{H}$.
        We apply \cref{cor:induced3} for~$\mathcal{H}$, $\{C_1,C_2\}:=\{C_{i,j},C_{i',j'}\}$, and~$ab$.
        As $\sigma(a)\neq\sigma(b)$, neither~\ref{item:packing1} nor~\ref{item:packing2} holds.
        As $\Gamma_{i,j}=\Gamma_{i',j'}=\emptyset$, none of~\ref{item:packing3}--\ref{item:packing5} hold.
        By \cref{clm:aux2-1}, \ref{item:packing6} does not hold.
        Hence, \cref{cor:induced3} does not hold, a contradiction.
    \end{subproof}
    
    \begin{claim}\label{clm:aux2-4}
        For a pair $(p,q)\in\mathcal{P}(\mathcal{H})$, if~$I$ has a subset~$I'$ of size $2k+1$ such that $\Sigma_{i,j}=\{(p,q)\}$ for every $w_{i,j}\in I'$, then~$G$ has an induced packing of~$k$ $(\ell,S)$-cycles.
    \end{claim}
    \begin{subproof}
        Note that for every $w_{i,j}\in I'$, $Q_{i,j}$ is a path of $P_{p,q}-B_{P_{p,q}}(\Gamma_{p,q},\ell-1)$.
        By \cref{lem:strict1}, $P_{p,q}$ is not an $\mathcal{H}$-detour.
        We divide into two cases according to whether~$P_{p,q}$ is fragile.
        
        \medskip
        \noindent
        \textbf{Case 1.} $P_{p,q}$ is fragile.

        Let~$P^1_{p,q}$ and~$P^2_{p,q}$ be the two subpaths of~$P_{p,q}$ between~$a^*_{p,q}$ and~$\Gamma_{p,q}$.
        For each $m\in[2]$, let~$I_m$ be the set of vertices $w_{i,j}\in I'$ such that~$Q_{i,j}$ is a subpath of~$P^m_{p,q}$.
        We have
        \[
            \abs{I_1}+\abs{I_2}\geq\abs{I'}-1=2k.
        \]
        Thus, for some $m\in[2]$, $\abs{I_m}\geq k$.

        We show that $\mathcal{C}_1:=\{C_{i,j}:w_{i,j}\in I_m\}$ is an induced packing in~$G$.
        Towards a contradiction, suppose that~$G$ has an edge~$ab$ between distinct~$C_{i,j}$ and~$C_{i',j'}$ in~$\mathcal{C}_1$ with $(i,j)>_L(i',j')$, where $a\in V(C_{i,j})$ and $b\in V(C_{i',j'})$.
        By \cref{clm:aux2-1}, $ab$ is a chord of~$\mathcal{H}$.
        We apply \cref{cor:induced3} for~$\mathcal{H}$, $\{C_1,C_2\}:=\{C_{i,j},C_{i',j'}\}$, and~$ab$.
        By \cref{clm:aux2-1}, \ref{item:packing6} does not hold.

        Suppose that $b\in V(Q_{i',j'})$.
        As $V(Q_{i',j'})\subseteq Y_{i',j'}$, we have $\sigma(a)\neq(i,j)$ and hence~$a$ is in~$Q_{i,j}$.
        Thus, $\sigma(a)=\sigma(b)=(p,q)$, and therefore none of~\ref{item:packing3}--\ref{item:packing5} hold.
        Since~$P_{p,q}$ is fragile, \ref{item:packing1} does not hold.
        By the construction of~$I_m$, \ref{item:packing2} does not hold.
        Hence, \cref{cor:induced3} does not hold, a contradiction.

        Now, suppose that $b\notin V(Q_{i',j'})$.
        By \cref{clm:aux2-2}, we have $a\in V(Q_{i,j})$ and hence $(p,q)=\sigma(a)<_L\sigma(b)=(i',j')$.
        Thus, neither~\ref{item:packing1} nor~\ref{item:packing2} holds.
        By the construction of~$I_m$, \ref{item:packing3} does not hold.
        Since~$P_{i',j'}$ is not a detached $\mathcal{H}$-earring, \ref{item:packing4} does not hold.
        Since~$C_{i,j}$ and~$C_{i',j'}$ are vertex-disjoint, \ref{item:packing5} does not hold.
        Hence, \cref{cor:induced3} does not hold, a contradiction.

        \medskip
        \noindent
        \textbf{Case 2.} $P_{p,q}$ is not fragile.

        By \cref{clm:aux2-1}, there are at most two vertices $w_{i,j}\in I'$ such that if~$P_{p,q}$ is a detached $\mathcal{H}$-earring, then~$Q_{i,j}$ contains a vertex in $B_{P_{p,q}}(Q_p,\ell-1)\setminus V(Q_p)$.
        Let~$I''$ be the set of the other vertices in~$I'$.
        Note that $\abs{I''}\geq2k-1\geq k$.

        We show that $\mathcal{C}_1:=\{C_{i,j}:w_{i,j}\in I''\}$ is an induced packing in~$G$.
        Towards a contradiction, suppose that~$G$ has an edge~$ab$ between distinct~$C_{i,j}$ and~$C_{i',j'}$ in~$\mathcal{C}_1$ with $(i,j)>_L(i',j')$, where $a\in V(C_{i,j})$ and $b\in V(C_{i',j'})$.
        By \cref{clm:aux2-1}, $ab$ is a chord of~$\mathcal{H}$.
        We apply \cref{cor:induced3} for~$\mathcal{H}$, $\{C_1,C_2\}:=\{C_{i,j},C_{i',j'}\}$, and~$ab$.
        By \cref{clm:aux2-1}, \ref{item:packing6} does not hold.
        
        Suppose that $b\in V(Q_{i',j'})$.
        As $V(Q_{i',j'})\subseteq Y_{i',j'}$, we have $\sigma(a)\neq(i,j)$ and hence~$a$ is in~$Q_{i,j}$.
        Thus, $\sigma(a)=\sigma(b)=(p,q)$, and therefore none of~\ref{item:packing3}--\ref{item:packing5} hold.
        By the construction of~$I''$, \ref{item:packing1} does not hold.
        Since~$P_{p,q}$ is not fragile, \ref{item:packing2} does not hold.
        Hence, \cref{cor:induced3} does not hold, a contradiction.

        Now, suppose that $b\notin V(Q_{i',j'})$.
        By \cref{clm:aux2-2}, we have $a\in V(Q_{i,j})$ and hence $(p,q)=\sigma(a)<_L\sigma(b)=(i',j')$.
        Thus, neither~\ref{item:packing1} nor~\ref{item:packing2} holds.
        Since~$P_{p,q}$ is not fragile, neither~\ref{item:packing3} nor~\ref{item:packing4} holds.
        Since~$C_{i,j}$ and~$C_{i',j'}$ are vertex-disjoint, \ref{item:packing5} does not hold.
        Hence, \cref{cor:induced3} does not hold, a contradiction.

        \medskip
        This proves the claim.
    \end{subproof}

    By \cref{clm:aux2-3}, we may assume that there are at most $k-1$ vertices $w_{i,j}\in I$ with $\Gamma_{i,j}=\emptyset$.
    Let~$I_1$ be the set of vertices $w_{i,j}\in I$ with $\Gamma_{i,j}\neq\emptyset$.
    Note that $\abs{I_1}\geq\abs{I}-(k-1)=6k(k-1)+1$.
    By \cref{clm:aux2-4}, for each $(p,q)\in\mathcal{P}(\mathcal{H})$, we may assume that there are at most~$2k$ vertices $w_{i,j}\in I$ with $\Sigma_{i,j}=\{(p,q)\}$.
    As $\abs{I_1}\geq6k(k-1)+1$, $I_1$ has a subset~$I_2$ of size~$3k-2$ such that $\Sigma_{i,j}\cap\Sigma_{i',j'}=\emptyset$ for all distinct $w_{i,j},w_{i',j'}\in I_2$.
    Let~$H$ be the graph with vertex set~$I_2$ such that~$w_{i,j}$ and~$w_{i',j'}$ with $(i,j)>_L(i',j')$ are adjacent if and only if~$P_{i',j'}$ contains a vertex in~$\Gamma_{i'',j''}$ for $(i'',j'')\in\Sigma_{i,j}$.
    Note that~$H$ is $2$-degenerate.
    Thus, $H$ has an independent set~$I_3$ of size~$k$.
    
    We show that $\mathcal{C}_1:=\{C_{i,j}:w_{i,j}\in I_3\}$ is an induced packing in~$G$.
    Towards a contradiction, suppose that~$G$ has an edge~$ab$ between distinct~$C_{i,j}$ and~$C_{i',j'}$ in~$\mathcal{C}_1$ with $(i,j)>_L(i',j')$, where $a\in V(C_{i,j})$ and $b\in V(C_{i',j'})$.
    By \cref{clm:aux2-1}, $ab$ is a chord of~$\mathcal{H}$.
    We apply \cref{cor:induced3} for~$\mathcal{H}$, $\{C_1,C_2\}:=\{C_{i,j},C_{i',j'}\}$, and~$ab$.
    By \cref{clm:aux2-1}, \ref{item:packing6} does not hold.
    We divide into five cases according to the comparison between~$\sigma(a)$ and~$\sigma(b)$ and whether $b\in V(Q_{i',j'})$.

    \medskip
    \noindent
    \textbf{Case 1.} $\sigma(a)=\sigma(b)$.

    Since~$C_{i,j}$ and~$C_{i',j'}$ are vertex-disjoint, we have $a\notin V(P_{i,j})$ and $b\notin V(P_{i',j'})$.
    Thus, $\Sigma_{i,j}=\Sigma_{i',j'}=\{\sigma(a)\}$, contradicting the construction of~$I_2$.

    \medskip
    \noindent
    \textbf{Case 2.} $\sigma(a)<_L\sigma(b)$ and $b\in V(Q_{i',j'})$.

    Let $\sigma(b):=(i'',j'')$.
    Since~$b$ is an admissible vertex of~$H_{i'',j''}$, \ref{item:packing3} does not hold.
    Suppose that~\ref{item:packing4} holds.
    Then for $\sigma(b):=(i'',1)$, $P_{i'',1}$ is a detached $\mathcal{H}$-earring, $\mu(i'')\geq2$, $\abs{E(P_{i'',2})}=1$, and $b\in\Gamma_{i'',2}$.
    This implies that $(i',j')=(i'',2)$.
    By \cref{lem:bridge}, $P_{i'',2}$ has exactly one end in~$P_{i'',1}$, contradicting that $\sigma(a_{i,j})=\sigma(b_{i,j})$.
    Thus, \ref{item:packing5} holds.
    Then $a\in\Gamma_{i'',j''}$.
    As $V(Q_{i',j'})\subseteq Y_{i',j'}$, we have $\sigma(a)\neq(i,j)$ and hence~$a$ is in~$Q_{i,j}$.
    This implies that $(i'',j'')=(i,j)$, contradicting that $(i,j)>_L(i',j')$.
    Hence, \cref{cor:induced3} does not hold, a contradiction.

    \medskip
    \noindent
    \textbf{Case 3.} $\sigma(a)<_L\sigma(b)$ and $b\notin V(Q_{i',j'})$.

    Note that $\sigma(b)=(i',j')$.
    By \cref{clm:aux2-2}, $a$ is in~$Q_{i,j}$.
    By the construction of~$I_2$, \ref{item:packing3} does not hold.
    Since~$P_{i',j'}$ is not a detached $\mathcal{H}$-earring, \ref{item:packing4} does not hold.
    Since~$C_{i,j}$ and~$C_{i',j'}$ are vertex-disjoint, \ref{item:packing5} does not hold.
    Hence, \cref{cor:induced3} does not hold, a contradiction.

    \medskip
    \noindent
    \textbf{Case 4.} $\sigma(b)<_L\sigma(a)$ and $b\in V(Q_{i',j'})$.

    As $V(Q_{i',j'})\subseteq Y_{i',j'}$, we have $\sigma(a)\neq(i,j)$ and hence~$a$ is in~$Q_{i,j}$.
    Let $\sigma(a):=(i'',j'')$.
    Since~$a$ is an admissible vertex of~$H_{i'',j''}$, \ref{item:packing3} does not hold.
    Suppose that~\ref{item:packing4} holds.
    Then~$P_{i'',j''}$ is a detached $\mathcal{H}$-earring, $\mu(i'')\geq2$, $\abs{E(P_{i'',2})}=1$, and $a\in\Gamma_{i'',2}$.
    This implies that $(i,j)=(i'',2)$.
    By \cref{lem:bridge}, $P_{i'',2}$ has exactly one end in~$P_{i'',1}$, contradicting that $\sigma(a_{i,j})=\sigma(b_{i,j})$.
    Thus, \ref{item:packing5} holds.
    Then $b\in\Gamma_{i'',j''}$.
    This implies that $(i'',j'')=(i',j')$ and hence~$P_{i',j'}$ contains a vertex in~$\Gamma_{i'',j''}$, contradicting that~$I_3$ is an independent set of~$H$.
    Hence, \cref{cor:induced3} does not hold, a contradiction.

    \medskip
    \noindent
    \textbf{Case 5.} $\sigma(b)<_L\sigma(a)$ and $b\notin V(Q_{i',j'})$.

    Note that $\sigma(b)=(i',j')$.
    By \cref{clm:aux2-2}, $a$ is in~$Q_{i,j}$.
    Let $\sigma(a):=(i'',j'')$.
    Since~$a$ is an admissible vertex of~$H_{i'',j''}$, \ref{item:packing3} does not hold.
    By the construction of~$I_2$, \ref{item:packing5} does not hold.
    Thus~\ref{item:packing4} holds.
    Then~$P_{i'',j''}$ is a detached $\mathcal{H}$-earring, $\mu(i'')\geq2$, $\abs{E(P_{i'',2})}=1$, and $a\in\Gamma_{i'',2}$.
    This implies that $(i,j)=(i'',2)$.
    By \cref{lem:bridge}, $P_{i'',2}$ has exactly one end in~$P_{i'',1}$, contradicting the fact that $\sigma(a_{i,j})=\sigma(b_{i,j})$.
    Hence, \cref{cor:induced3} does not hold, a contradiction.

    \medskip
    This completes the proof.
\end{proof}

We construct the second auxiliary graph as follows.
Let $\mathcal{P}_2(\mathcal{H})$ be the set of pairs $(i,j)\in\mathcal{P}(\mathcal{H})$ satisfying one of the following.
\begin{enumerate}[label=(\alph*)]
    \item\label{item:aux3-1} $\mathcal{H}$ has a chord~$e$ between vertices of~$P_{i,j}$ such that $P_{i,j}+e$ has an $(\ell,S)$-cycle containing~$e$.
    \item\label{item:aux3-3} $P_{i,j}$ is fragile and~$\mathcal{H}$ has distinct vertices $u$, $u'$, and~$v$ such that
    \begin{itemize}
        \item $\sigma(v):=(i',j')\neq(i,j)$ for some $j'\geq2$,
        \item $u$ and~$u'$ are in distinct components of~$P^{\operatorname{side}}_{i,j}$,
        \item $uv$ is a chord of~$\mathcal{H}$, $u'v$ is an edge of~$\mathcal{H}$, and
        \item the cycle obtained by concatenating $uvu'$ and the $(u,u')$-path of $P_{i,j}-\Gamma_{i,j}$ is an $(\ell,S)$-cycle.
    \end{itemize}
\end{enumerate}
Let~$A_2(\mathcal{H})$ be the graph with vertex set $\{w_{i,j}:(i,j)\in\mathcal{P}_2(\mathcal{H})\}$ such that~$w_{i,j}$ and~$w_{i',j'}$ with $(i,j)>_L(i',j')$ are adjacent if and only if either
\begin{itemize}
    \item $P_{i,j}$ and~$P_{i',j'}$ intersect, or
    \item $P_{i',j'}$ contains a vertex in $\Gamma_{i'',j''}$ for some $(i'',j'')\in\Sigma_{i,j}$.
\end{itemize}
For each $(i,j)\in\mathcal{P}_2(\mathcal{H})$, $P_{i,j}$ has an internal vertex and hence $\abs{E(P_{i,j})}\geq2$.
Note that~$P_{i,j}$ intersects at most two $P_{i',j'}$ with $(i',j')<_L(i,j)$ and that $\abs{\Gamma_{i'',j''}}\leq2$ for each $(i'',j'')\in\Sigma_{i,j}$.
Thus, $A_2(\mathcal{H})$ is $6$-degenerate.

We show that if $\mathcal{P}_2(\mathcal{H})$ is large, then~$G$ has an induced packing of~$k$ $(\ell,S)$-cycles.

\begin{lemma}\label{lem:aux3}
    Let $(G,S)$ be a good pair and let~$\mathcal{H}$ be a maximal $(\ell,S)$-frame in~$G$.
    If
    \[
        \abs{\mathcal{P}_2(\mathcal{H})}\geq20160000\cdot k^3(2k-1)^2
    \]
    for some positive integer~$k$, then~$G$ has an induced packing of~$k$ $(\ell,S)$-cycles.
\end{lemma}
\begin{proof}
    Since~$A_2(\mathcal{H})$ is $6$-degenerate and has at least
    \[
        20160000\cdot k^3(2k-1)^2=7\cdot2880000\cdot k^3(2k-1)^2
    \]
    vertices, it has an independent set~$I$ of size $2880000\cdot k^3(2k-1)^2$.
    Let $\mathcal{C}_0:=\{P_{i,j}:w_{i,j}\in I\}$.
    By the definition of the fragility and~\ref{item:ear3}, for every $(i,j)\in\mathcal{P}_2(\mathcal{H})$, $P_{i,j}$ is not a detached $\mathcal{H}$-earring.
    By the construction of~$I$, the graphs in~$\mathcal{C}_0$ are pairwise vertex-disjoint.
    In addition, by \cref{lem:argue distance}, they are pairwise at distance at least~$\ell$ in~$\mathcal{H}$.
    This obviously implies that~$\mathcal{C}_0$ is an induced packing in~$\mathcal{H}$.
    We will use the following claim.

    \begin{claim}\label{clm:aux3-1}
        For a pair $(p,q)\in\mathcal{P}(\mathcal{H})$, if~$I$ has a subset~$I'$ of size $12(2k-1)$ such that for all distinct $w_{i,j},w_{i',j'}\in I'$, $\abs{\Sigma_{i,j}}=\abs{\Sigma_{i',j'}}=2$ and $\Sigma_{i,j}\cap\Sigma_{i',j'}\subseteq\{(p,q)\}$, then~$G$ has an induced packing of~$k$ $(\ell,S)$-cycles.
    \end{claim}
    \begin{subproof}
        By \cref{lem:jump3}, $I'$ has a subset~$I''$ of size $4k-2$ such that $\{P_{i,j}:w_{i,j}\in I''\}$ is an induced packing in~$G$.
        If there are at least~$k$ vertices $w_{i,j}\in I''$ such that $(i,j)$ satisfies~\ref{item:aux3-1}, then we are done.
        Thus, we may assume that this is not the case.
        Let~$J$ be the set of vertices $w_{i,j}\in I''$ such that $(i,j)$ satisfies~\ref{item:aux3-3}.
        Note that $\abs{J}\geq3k-2$.
        For each $w_{i,j}\in J$, we denote by~$u_{i,j}$, $u'_{i,j}$, and~$v_{i,j}$ the vertices taking the roles of~$u$, $u'$, and~$v$ in~\ref{item:aux3-3}, respectively.

        Let~$H$ be the digraph with vertex set~$J$ such that there is an arc from~$w_{i,j}$ and~$w_{i',j'}$ if and only if $(i,j)\neq(i',j')$ and~$P_{i',j'}$ contains a vertex in~$\Gamma_{r,s}$ for $(r,s):=\sigma(v_{i,j})$.
        As $\abs{J}\geq3k-2$, by \cref{lem:digraph}, $H$ has an independent set~$J'$ of size~$k$.
        For each $w_{i,j}\in J'$, let $P'_{i,j}$ be the cycle obtained by concatenating $u_{i,j}v_{i,j}u'_{i,j}$ and the $(u_{i,j},u'_{i,j})$-path of $P_{i,j}-\Gamma_{i,j}$, which is an $(\ell,S)$-cycle.

        We show that $\mathcal{C}_1:=\{P'_{i,j}:w_{i,j}\in J'\}$ is an induced packing in~$G$.
        Towards a contradiction, suppose that~$G$ has an edge~$ab$ between distinct~$P'_{i,j}$ and~$P'_{i',j'}$ in~$\mathcal{C}_1$, where $a\in V(P'_{i,j})$ and $b\in V(P'_{i',j'})$.
        As $\dist_\mathcal{H}(P_{i,j},P_{i',j'})\geq\ell$, $ab$ is not an edge of~$\mathcal{H}$ and hence is a chord of~$\mathcal{H}$.
        Without loss of generality, we may assume that $(i_1,j_1):=\sigma(a)\leq_L\sigma(b)=:(i_2,j_2)$.
        Since $\{P_{i,j}:w_{i,j}\in I''\}$ is an induced packing in~$G$, either $a=v_{i,j}$ or $b=v_{i',j'}$.
        If $a=v_{i,j}$, then $\sigma(a)=\sigma(b)$, because otherwise $b\in Z_{i_1,j_1}$.
        This contradicts that~$J'$ is an independent set of~$H$.
        Hence, $a\neq v_{i,j}$ and $b=v_{i',j'}$.
        Thus, by~\ref{item:aux3-3}, we have $u'_{i',j'}\in\Gamma_{i_2,j_2}$.
        Without loss of generality, we may assume that $u'_{i',j'}=b_{i_2,j_2}$.
        Note that $\sigma(b_{i_2,j_2})=(i',j')$.

        Suppose that $\sigma(a)=\sigma(b)$.
        Since~$P_{i,j}$ and~$P_{i',j'}$ are vertex-disjoint, we have $\sigma(a)\neq(i,j)$ and hence $a\in\Gamma_{i,j}$.
        Then~$P_{i',j'}$ contains a vertex in $\Gamma_{i_1,j_1}$, contradicting that~$I$ is an independent set of $A_2(\mathcal{H})$.
        Hence, $\sigma(a)\neq\sigma(b)$.
        Note that~$a$ is an admissible vertex of~$H_{i_2,j_2-1}$, because otherwise $b\in Z_{i_2,j_2-1}$.
        Thus, $\sigma(a)=(i,j)$.
        We apply \cref{prop:induced2} for~$\mathcal{H}$ and~$ab$.
        Since~$P_{i_2,j_2}$ is not a detached $\mathcal{H}$-earring, \cref{prop:induced2}\ref{item:prop2-2} does not hold.
        By the construction of~$J'$, \cref{prop:induced2}\ref{item:prop2-3} does not hold.
        Thus, \cref{prop:induced2}\ref{item:prop2-1} holds for some $x\in\Gamma_{i_2,j_2}$.
        As $\dist_\mathcal{H}(P_{i,j},P_{i',j'})\geq\ell$ and $\sigma(a)=(i,j)$, we have $x=a_{i_2,j_2}$.
        By the construction of~$J'$, $a_{i_2,j_2}$ is not in~$P_{i,j}$.
        Thus, both~$a$ and~$b$ are in $B_\mathcal{H}(a_{i_2,j_2},\ell-1)=B_{H_{i_2,j_2}}(a_{i_2,j_2},\ell-1)$.
        As $\sigma(a)<_L\sigma(b)$, we have $a\in B_{H_{i_2,j_2-1}}(a_{i_2,j_2},\ell-1)$.
        Since~$a$ and~$a_{i_2,j_2}$ are admissible vertices of~$H_{i_2,j_2-1}$, by \cref{lem:nearby} for $\mathcal{H}:=H_{i_2,j_2-1}$, we have $\sigma(a_{i_2,j_2})=\sigma(a)=(i,j)$, contradicting that~$J'$ is an independent set of~$H$.
    \end{subproof}
    
    Let~$M$ be a maximal subset of~$I$ such that $\Sigma_{i,j}\cap\Sigma_{i',j'}=\emptyset$ for all distinct $w_{i,j},w_{i',j'}\in M$.
    By \cref{clm:aux3-1}, we may assume that $\abs{M}\leq24k-13$.
    Let $\Sigma_M:=\bigcup_{w_{i,j}\in M}\Sigma_{i,j}$.
    Note that $\abs{\Sigma_M}\leq2\abs{M}$.
    For each $(p,q)\in\Sigma_M$, let $I_{p,q}$ be the subset of~$I$ such that for all distinct $w_{i,j},w_{i',j'}\in I_{p,q}$, $\abs{\Sigma_{i,j}}=\abs{\Sigma_{i',j'}}=2$ and $\Sigma_{i,j}\cap\Sigma_{i',j'}=\{(p,q)\}$.
    By \cref{clm:aux3-1}, for each $(p,q)\in\Sigma_M$, we may assume that $\abs{I_{p,q}}\leq24k-13$.
    Then the number of distinct $\Sigma_{i,j}$ for $w_{i,j}\in I$ is at most
    \[
        \abs{\Sigma_M}+\abs{\Sigma_M}\cdot(24k-13)\leq2(24k-13)(24k-12)<2(24k-12)^2=288(2k-1)^2.
    \]
    As $\abs{I}=2880000\cdot k^3(2k-1)^2$, $I$ has a subset~$J$ of size $10000k^3$ such that $\Sigma_{i,j}=\Sigma_{i',j'}\neq\emptyset$ for all $w_{i,j},w_{i',j'}\in J$.
    We remark that
    \[
        10000k^3\geq\max\{16(k+8)^2+16,(2k+5)(26k-19)^2\}
    \]
    for every $k\geq1$.
    Thus, by \cref{lem:jump1} or \cref{lem:jump2}, $G$ has an induced packing of~$k$ $(\ell,S)$-cycles.
\end{proof}

We now prove \cref{prop:final}.

\begin{proof}[Proof of \cref{prop:final}]
    By \cref{lem:aux2,lem:aux3}, we may assume that
    \begin{align*}
        \rho_1&:=\abs{\mathcal{P}_1(\mathcal{H})}<3k(6k-5),\\
        \rho_2&:=\abs{\mathcal{P}_2(\mathcal{H})}<20160000\cdot k^3(2k-1)^2<3\cdot10^7\cdot k^3(2k-1)^2.
    \end{align*}
    We are going to construct a subgraph~$\mathcal{H}'$ of~$\mathcal{H}^-$ such that every set of pairwise vertex-disjoint cycles of~$\mathcal{H}'$ is an induced packing in~$G$.
    Let $\mathcal{P}_3(\mathcal{H})$ be the set of pairs $(i,j)\in\mathcal{P}(\mathcal{H})\setminus\mathcal{P}_2(\mathcal{H})$ with $j\geq2$ such that $\Gamma_{i,j}\subseteq\bigcup_{(i,j)\in\mathcal{P}_2(\mathcal{H})}V(P_{i,j})$.
    If
    \[
        \rho_3:=\abs{\mathcal{P}_3(\mathcal{H})}\geq\left(\rho_2+\frac{\rho_2(\rho_2-1)}{2}\right)\cdot10000k^3,
    \]
    then $\mathcal{P}_3(\mathcal{H})$ has a set~$I$ of size $10000k^3$ such that $\Sigma_{i,j}=\Sigma_{i',j'}$ for all $(i,j),(i',j')\in I$.
    Thus, by \cref{lem:jump1} or \cref{lem:jump2}, $G$ has an induced packing of~$k$ $(\ell,S)$-cycles.
    Hence, we may assume that
    \[
        \rho_3<\left(\rho_2+\frac{\rho_2(\rho_2-1)}{2}\right)\cdot10000k^3\leq\rho_2^2\cdot10000k^3<10^{19}\cdot k^9(2k-1)^4.
    \]
    For each $(i,j)\in\mathcal{P}_2(\mathcal{H})$, we define two subpaths~$P^1_{i,j}$ and~$P^2_{i,j}$ of~$P_{i,j}$ as follows:
    \begin{itemize}
        \item if~$P_{i,j}$ is fragile, then let~$P^1_{i,j}$ and~$P^2_{i,j}$ be the two subpaths of~$P_{i,j}$ between~$\Gamma_{i,j}$ and $\{a^*_{i,j},b^*_{i,j}\}$ which are edge-disjoint from $P^*_{i,j}$, and
        \item otherwise let $P^1_{i,j}:=P_{i,j}$ and let~$P^2_{i,j}$ be a null graph.
    \end{itemize}
    Let $\mathcal{P}_4(\mathcal{H})$ be the set of pairs $(p,q)\in\mathcal{P}(\mathcal{H})\setminus(\mathcal{P}_2(\mathcal{H})\cup\mathcal{P}_3(\mathcal{H}))$ with $q\geq2$.
    For each $(i,j)\in\mathcal{P}_2(\mathcal{H})$, let~$r_{i,j}$ and~$r'_{i,j}$ be the number of pairs $(p,q)\in\mathcal{P}_4(\mathcal{H})$ such that~$P^1_{i,j}$ and~$P^2_{i,j}$ contain a vertex in $\Gamma_{p,q}$, respectively.
    Without loss of generality, for every $(i,j)\in\mathcal{P}_2(\mathcal{H})$, we may assume that $r_{i,j}\geq r'_{i,j}$.

    Let~$\mathcal{H}'$ be the graph obtained from~$\mathcal{H}^-$ by removing every vertex~$v$ such that either
    \begin{enumerate}[label=(R\arabic*)]
        \item\label{rule1} $\deg_\mathcal{H}(v)=4$, or
        \item\label{rule2} for some detached $\mathcal{H}$-earring~$P_{i,1}$, $v$ is an end of~$Q_i$ or~$P_{i,2}$, or
        \item\label{rule3} $v=a_{i,j}$ for some $(i,j)\in\mathcal{P}_1(\mathcal{H})$ with $j\geq2$, or
        \item\label{rule4} $v\in V(P^2_{i,j})$ for some $(i,j)\in\mathcal{P}_2(\mathcal{H})$,
    \end{enumerate}    
    and then recursively removing vertices of degree at most~$1$.
    Note that~$\mathcal{H}'$ is a subcubic graph without pendant vertices, unless it is a null graph.
    We will use the following claim.
    
    \begin{claim}\label{clm:final}
        If~$\mathcal{H}'$ has vertex-disjoint cycles~$C_1$ and~$C_2$, then $\{C_1,C_2\}$ is an induced packing in~$G$.
    \end{claim}
    \begin{subproof}
        Towards a contradiction, suppose that~$G$ has an edge~$ab$ between~$C_1$ and~$C_2$.
        Without loss of generality, we may assume that $a\in V(C_1)$, $b\in V(C_2)$, and $(p,q):=\sigma(a)\leq_L\sigma(b)=:(r,s)$.

        Suppose that $\dist_\mathcal{H}(a,b)\leq\ell-1$.
        Note that~$a$ and~$b$ are branch vertices of~$\mathcal{H}$.
        By \cref{cor:short branch}, there is a unique $(i,j)\in\mathcal{P}(\mathcal{H})$ such that $\{a,b\}=\{a_{i,j},b_{i,j}\}$.
        By \cref{obs:structure}\ref{item:order}, $B_{H_{i,j}}(\Gamma_{i,j},\ell-1)=B_\mathcal{H}(\Gamma_{i,j},\ell-1)$.
        Thus, $Q$ is a path of~$H_{i,j}$.
        If~$P_{i,j}$ is an $\mathcal{H}$-detour, then by \cref{lem:strict1}, $Q$ is either~$P_{i,1}$ or~$Q_i$ and hence each of~$C_1$ and~$C_2$ contains both~$a_{i,j}$ and~$b_{i,j}$, a contradiction.
        Thus, $P_{i,j}$ is a regular $\mathcal{H}$-ear.
        By \ref{rule3}, we have $\dist_{H_{i,j-1}}(a_{i,j},b_{i,j})\geq\ell+1$, and therefore $Q=P_{i,j}$.
        By \cref{cor:chord} and~\ref{rule2}, $P_{p,q}$ is fragile, and~$a_{i,j}$ and~$b_{i,j}$ are in distinct components of~$P^{\operatorname{side}}_{p,q}$.
        Thus, $ab$ is not an edge of~$\mathcal{H}$ and hence is a chord of~$\mathcal{H}$.
        As $\dist_{H_{i,j-1}}(a_{i,j},b_{i,j})\geq\ell+1$, $P_{p,q}+ab$ has an $(\ell,S)$-cycle.
        Thus, $(p,q)\in\mathcal{P}_2(\mathcal{H})$, contradicting~\ref{rule4}.
        
        Hence, $\dist_\mathcal{H}(a,b)\geq\ell$.
        Note that~$ab$ is a chord of~$\mathcal{H}$.
        We apply \cref{cor:induced3} for~$\mathcal{H}$, $\{C_1,C_2\}$, and~$ab$.
        By~\ref{rule2}, neither \ref{item:packing1} nor~\ref{item:packing4} holds.
        As $\dist_\mathcal{H}(a,b)\geq\ell$, if one of \ref{item:packing2}, \ref{item:packing3}, and \ref{item:packing5} holds, then $(p,q)$ or $(r,s)$ is in $\mathcal{P}_2(\mathcal{H})$, contradicting~\ref{rule4}.
        Thus, \ref{item:packing6} holds.
        Then there is $(i,j)\in\mathcal{P}(\mathcal{H})$ with $j\geq2$ such that~$H_{i,j}$ has an $(a_{i,j},b_{i,j})$-path~$Q$ of length at most $\ell-1$ and for each $h\in[2]$, there are $x_h\in\Gamma_{i,j}$ and $y_h\in\{a,b\}$ with $y_h\in B_{\mathcal{H}-E(Q)}(x_h,\ell-1)\subseteq V(C_h)$.
        Without loss of generality, we may assume that $x_1=a_{i,j}$, $x_2=b_{i,j}$, $y_1=a$, and $y_2=b$.
        
        Since~$a_{i,j}$ and~$a$ are admissible vertices of~$H_{i,j-1}$ with $\dist_{H_{i,j-1}}(a_{i,j},a)\leq\ell-1$, by \cref{lem:nearby}, $\sigma(a_{i,j})=\sigma(a)=(p,q)$.
        Similarly, $\sigma(b_{i,j})=\sigma(b)=(r,s)$.
        By~\ref{rule3}, we have $\dist_{H_{i,j-1}}(a_{i,j},b_{i,j})\geq\ell+1$, and therefore $\abs{E(P_{i,j})}\leq\ell-1$.
        By \cref{cor:chord} and~\ref{rule2}, $\sigma(a_{i,j})=\sigma(b_{i,j})=(p,q)$ and $P_{p,q}$ is fragile.
        By \cref{prop:induced1} for~$\mathcal{H}$ and~$ab$, $a$ and~$b$ are in distinct components of $P^{\operatorname{side}}_{p,q}$.
        As $\dist_\mathcal{H}(a,b)\geq\ell$, we have $(p,q)\in\mathcal{P}_2(\mathcal{H})$, contradicting~\ref{rule4}.
    \end{subproof}

    We show that~$\mathcal{H}'$ has more than~$s_k$ branch vertices.
    For each $i\in[4]$, let~$r_i$ be the number of branch vertices of~$\mathcal{H}^-$ removed by (R$i$).
    Note that $r_1+r_2+r_3\leq4\rho_1$.
    For each $(p,q)\in\mathcal{P}_4(\mathcal{H})$, $\Gamma_{p,q}$ contains at most one vertex in $\bigcup_{(i,j)\in\mathcal{P}_2(\mathcal{H})}V(P_{i,j})$ as $(p,q)\notin\mathcal{P}_2(\mathcal{H})$.
    Thus, we have
    \[
        \rho_4:=\abs{\mathcal{P}_4(\mathcal{H})}=\sum_{(i,j)\in\mathcal{P}_2(\mathcal{H})}(r_{i,j}+r'_{i,j})\geq\sum_{(i,j)\in\mathcal{P}_2(\mathcal{H})}2r'_{i,j},
    \]
    and therefore
    \[
        \sum_{(i,j)\in\mathcal{P}_2(\mathcal{H})}2r'_{i,j}\leq\rho_4\leq\frac{\abs{V_{\geq3}(\mathcal{H}^-)}}{2}.
    \]
    For each $(i,j)\in\mathcal{P}_2(\mathcal{H})$, let~$m_{i,j}$ be the number of vertices in $V(P^2_{i,j})\setminus\Gamma_{i,j}$ having degree~$4$ in~$\mathcal{H}$ and let~$p_{i,j}$ be the number of vertices of $\bigcup_{(i,j)\in\mathcal{P}_3(\mathcal{H})}\Gamma_{i,j}$ contained in~$P^2_{i,j}$.
    When we remove $V(P^2_{i,j})$ from~$\mathcal{H}^-$, we lose at most $3m_{i,j}+2p_{i,j}+2r'_{i,j}+5$ branch vertices of~$\mathcal{H}^-$, where~$5$ bounds the number of branch vertices of~$\mathcal{H}$ which are removed only by the removal of the ends of~$P^2_{i,j}$.
    Thus, by~\ref{rule4}, $\mathcal{H}^-$ loses at most
    \[
        \sum_{(i,j)\in\mathcal{P}_2(\mathcal{H})}(3m_{i,j}+2p_{i,j}+2r'_{i,j}+5)
        \leq3\rho_1+4\rho_3+\frac{\abs{V_{\geq3}(\mathcal{H}^-)}}{2}+5\rho_2
    \]
    branch vertices.
    If we remove a branch vertex~$v$ of~$\mathcal{H}^-$ and recursively remove vertices of degree at most~$1$, then we lose at most $\deg_{\mathcal{H}^-}(v)+1$ branch vertices of~$\mathcal{H}^-$ including~$v$.
    Thus, to construct~$\mathcal{H}'$, we have removed at most
    \begin{align*}
        &5r_1+4r_2+4r_3+\left(3\rho_1+4\rho_3+\frac{\abs{V_{\geq3}(\mathcal{H}^-)}}{2}+5\rho_2\right)\\
        &\leq23\rho_1+5\rho_2+4\rho_3+\frac{\abs{V_{\geq3}(\mathcal{H}^-)}}{2}\\
        &<69k(6k-5)+15\cdot10^7\cdot k^3(2k-1)^2+4\cdot10^{19}\cdot k^9(2k-1)^4+\frac{\abs{V_{\geq3}(\mathcal{H}^-)}}{2}\\
        &<5\cdot10^{19}\cdot k^9(2k-1)^4+\frac{\abs{V_{\geq3}(\mathcal{H}^-)}}{2}
    \end{align*}
    branch vertices of~$\mathcal{H}^-$.
    Hence, we have
    \[
        \abs{V_{\geq3}(\mathcal{H}')}>\frac{\abs{V_{\geq3}(\mathcal{H}^-)}}{2}-5\cdot10^{19}\cdot k^9(2k-1)^4\geq s_k.
    \]

    By \cref{thm:simonovitz}, $\mathcal{H}'$ has~$k$ pairwise vertex-disjoint cycles.
    By \cref{lem:all cycles}, they are $(\ell,S)$-cycles.
    Thus, by \cref{clm:final}, $G$ has an induced packing of~$k$ $(\ell,S)$-cycles.
\end{proof}

\section{\texorpdfstring{$(\ell,S)$-subframes}{(l,S)-subframes}}\label{sec:subframe}

In this section, we introduce $(\ell,S)$-subframes in graphs, which will serve as the final tool for proving \cref{thm:main1}.
Let~$G$ be a graph and let~$S$ be a subset of $V(G)$.
An \emph{$(\ell,S)$-central path} of~$G$ is a path~$F$ of~$G$ such that
\begin{itemize}
    \item $F$ is a shortest path of~$G$ between its ends,
    \item $G$ has no $(\ell,S)$-ear of~$F$, and
    \item every $(\ell,S)$-cycle of~$G$ contains at least two vertices of~$F$.
\end{itemize}
Throughout this section, we assume that~$G$ has an $(\ell,S)$-central path.
We will often use the following observation, derived from the second condition of the definition of an $(\ell,S)$-central path.

\begin{observation}\label{obs:central ears}
    Let~$G$ be a graph.
    For a set $S\subseteq V(G)$, let $F$ be an $(\ell,S)$-central path of~$G$ and let~$P$ be an $F$-path containing a vertex in~$S$ with ends~$a$ and~$b$.
    Then $\abs{E(P)}+\dist_F(a,b)\leq\ell-1$.
\end{observation}

We now define an $(\ell,S)$-subframe~$F_t$ in~$G$ as follows.
We first consider an $(\ell,S)$-central path~$F_1$ as an $(\ell,S)$-subframe in~$G$.
We denote by~$a_1$ and~$b_1$ the ends of~$F_1$.
Let $F^-_1:=F_1$.
Suppose that we have found an $(\ell,S)$-subframe~$F_i$ in~$G$ for some $i\geq1$ with~$F^-_i$.
Let
\begin{align*}
    Y_i&:=B_{F_i}(V_{\geq3}(F_i),\ell-1),\\
    Z_i&:=B_{G-(V(F_i)\setminus Y_i)}(Y_i,1),\\
    G_i&:=G-Z_i.
\end{align*}
An \emph{admissible pair} of~$F_i$ is a pair $(R,R')$ of internally disjoint paths of~$G_i$ satisfying the following conditions:
\begin{itemize}
    \item $R$ is a strict $S$-ear of~$F_i$, and
    \item $R'$ is an $(\ell,S)$-ear of $F^-_i\oplus R$ such that for the base~$Q$ of~$R$ in~$F_i$, every component of $Q\cap R'$ contains an end of~$Q$.
\end{itemize}
Note that~$R'$ has a unique subpath, denoted by $\varphi(R')$, which is an $F_i$-path.
We inductively define a larger $(\ell,S)$-subframe~$F_{i+1}$ in~$G$ together with~$F^-_{i+1}$.
We recall the definitions of $\mathcal{L}_0(H)$ and $\mathcal{L}_1(H)$ given at the beginning of~\cref{sec:frame}.
\begin{enumerate}[label=(\Alph*), start=5]
    \item\label{item:sub1} If~$G_i$ has an admissible pair of~$F_i$, then let $(R_{i+1},R'_{i+1})$ be such a pair with minimum $\abs{E(R'_{i+1})}$, let $F_{i+1} :=F_i\cup R_{i+1}\cup R'_{i+1}$, and let $F^-_{i+1}:=(F^-_i\oplus R_{i+1})\cup R'_{i+1}$.
    We denote by~$a_{i+1}$ and~$b_{i+1}$ the ends of~$R_{i+1}$, and by~$a'_{i+1}$ and~$b'_{i+1}$ the ends of $\varphi(R'_{i+1})$.
    We say that $i+1$ is \emph{pair-type}.
    \item\label{item:sub2} If~\ref{item:sub1} does not hold and~$G_i$ has a strict $S$-ear of~$F_i$ disjoint from $B_{F_i}(\mathcal{L}_1(F^-_i),\ell-1)$, then let $R_{i+1}$ be such an ear whose base in~$F_i$ is closest to one of the degree-$1$ vertices of $\mathcal{L}_0(F^-_i)$ in~$F_i$, let $F_{i+1}:=F_i\cup R_{i+1}$, and let $F^-_{i+1}:=F^-_i\oplus R_{i+1}$.
    We say that $i+1$ is \emph{path-type}.
\end{enumerate}

We say that an $(\ell,S)$-subframe~$F_t$ is \emph{maximal} if~$G_t$ has neither an admissible pair of~$F_t$ nor~$R_{t+1}$ satisfying~\ref{item:sub2}.
A vertex~$v$ of~$F_t$ is \emph{admissible} if it is not contained in~$Y_t$.
Note that for all $i,j\in[t]$ with $i\leq j$, $Y_i$ is a subset of~$Y_j$.
We denote by $\sigma_{F_t}(v)$ the smallest integer~$i$ such that~$v$ is a vertex of~$F_i$.
We may omit the subscript if it is clear from the context.

We first show that every cycle of~$F^-_t$ is an $(\ell,S)$-cycle.

\begin{lemma}\label{lem:all cycles-sub}
    Let~$G$ be a graph.
    For a set $S\subseteq V(G)$, let~$F_t$ be an $(\ell,S)$-subframe in~$G$.
    Then every cycle of~$F^-_t$ is an $(\ell,S)$-cycle.
\end{lemma}
\begin{proof}
    We proceed by induction on~$t$.
    The statement obviously holds for $t=1$.
    Thus, we may assume that $t\geq2$.
    Let~$C$ be a cycle of~$F^-_t$.
    Suppose that~$t$ is pair-type and~$C$ contains an edge of~$R'_t$.
    Since every internal vertex of~$R'_t$ has degree~$2$ in~$F^-_t$, $R'_t$ is a subpath of~$C$.
    Then by the definition of an $(\ell,S)$-ear, $C$ is an $(\ell,S)$-cycle.

    Hence, we may assume that either~$t$ is path-type, or~$C$ does not contain an edge of~$R'_t$.
    By the inductive hypothesis, we may assume that~$C$ is not a cycle of~$F^-_{t-1}$.
    This implies that~$C$ contains an edge of~$R_t$, because $E(F^-_t)\setminus E(F^-_{t-1})\subseteq E(R_t\cup R'_t)$.
    Since every internal vertex of~$R_t$ has degree~$2$ in~$F^-_t$, $R_t$ is a subpath of~$C$.
    Since~$R_t$ is a strict $S$-ear, $C$ contains a vertex in~$S$.
    Since both~$a_t$ and~$b_t$ are admissible vertices of~$F_{t-1}$, they are at distance at least~$\ell$ from every branch vertex of~$F_{t-1}$.
    Thus, $C$ has length at least~$2\ell$, and therefore it is an $(\ell,S)$-cycle.

    This completes the proof by induction.
\end{proof}

We show that every $F_t$-path of~$G_t$ containing a vertex in~$S$ is an $F_1$-path.

\begin{lemma}\label{lem:F1path}
    Let~$G$ be a graph.
    For a set $S\subseteq V(G)$, let~$F_t$ be an $(\ell,S)$-subframe in~$G$.
    Then every $F_t$-path of~$G_t$ containing a vertex in~$S$ is an $F_1$-path.
    Consequently, for each $i\in[t]\setminus\{1\}$, the vertex set of $R_i\cup(a_iF_1b_i)$ is a subset of~$Y_i$.
\end{lemma}
\begin{proof}
    Let~$P$ be an $F_t$-path of~$G_t$ containing a vertex in~$S$ with ends~$a$ and~$b$.
    Towards a contradiction, suppose that~$a$ or~$b$ is not in~$F_1$.
    By the definition of an $(\ell,S)$-subframe, extending~$P$ from each of its ends along the paths in~$F_t$ towards~$F_1$ yields an $F_1$-path~$P'$ containing~$P$ as a subpath.
    Since~$a$ and~$b$ are admissible vertices of~$F_t$, they are at distance at least~$\ell$ from every branch vertex of~$F_t$.
    Since~$a$ or~$b$ is not in~$F_1$, $P'$ has length at least $\ell+1$.
    Thus, $P'$ is an $(\ell,S)$-ear of~$F_1$, a contradiction.

    Hence, for each $i\in[t]\setminus\{1\}$, $R_i$ is an $F_1$-path.
    Since~$R_i$ is a strict $S$-ear, by \cref{obs:central ears}, $\abs{E(R_i)}+\dist_{F_1}(a_i,b_i)\leq\ell-1$.
    Thus, we have $V(R_i\cup(a_iF_1b_i))\subseteq Y_i$.
\end{proof}

We show that if~$G_t$ has a short $F^-_t$-path between vertices~$a$ and~$b$, then $\sigma(a)=\sigma(b)$.

\begin{lemma}\label{lem:short path-sub}
    Let~$G$ be a graph.
    For a set $S\subseteq V(G)$, let~$F_t$ be an $(\ell,S)$-subframe in~$G$.
    If~$G_t$ has an $F^-_t$-path of length at most $\ell-1$ between vertices~$a$ and~$b$, then $\sigma(a)=\sigma(b)$.
\end{lemma}
\begin{proof}
    Let~$P$ be an $F^-_t$-path of length at most $\ell-1$ between~$a$ and~$b$.
    Towards a contradiction, suppose that $\sigma(a)\neq\sigma(b)$.
    Without loss of generality, we may assume that $\sigma(a)<\sigma(b)=:i$.
    Since~$b$ is an admissible vertex of~$F_t$, by \cref{lem:F1path}, it is not in $R_i\cup(a_iF_1b_i)$.
    Thus, $i$ is pair-type and~$b$ is an internal vertex of $\varphi(R'_i)$.
    Let~$a''_i$ and~$b''_i$ be the ends of~$R'_i$.
    Since~$R'_i$ is an $(\ell,S)$-ear of $F^-_{i-1}\oplus R_i$, without loss of generality, we may assume the following:
    \begin{itemize}
        \item if~$R'_i$ contains a vertex in~$S$, then $a''_iR'_ib$ contains the vertex, and
        \item otherwise either $a\in S$ or~$a$ and~$a''_i$ are in distinct components of $(F^-_{i-1}\oplus R_i)-S$.
    \end{itemize}
    Let~$P'$ be the path obtained by concatenating $a''_iR'_ib$ and~$P$.
    Note that~$P'$ is an $(F^-_{i-1}\oplus R_i)$-path in~$G_{i-1}$ as~$P$ is an $F^-_t$-path in~$G_t$.
    By the assumption, $P'$ is an $S$-ear of $F^-_{i-1}\oplus R_i$.
    Since~$b$ is an admissible vertex of~$F^-_t$, we have $\dist_{R'_i}(b,\{a'_i,b'_i\})\geq\ell>\abs{E(P)}$.
    Thus, we have
    \[
        \abs{E(R'_i)}>\abs{E(P')}\geq\abs{E(P)}+\dist_{R'_i}(b,a''_i)\geq\abs{E(P)}+\dist_{R'_i}(b,\{a'_i,b'_i\})\geq\ell+1.
    \]
    Therefore, $P'$ is an $(\ell,S)$-ear of $F^-_{i-1}\oplus R_i$ in~$G_{i-1}$ shorter than~$R'_i$, contradicting~\ref{item:sub1}.
\end{proof}

In the remainder of this section, we assume that~$G$ has no $(\ell,S)$-cycle of length at most $3(\ell-1)$.
The following lemma shows that if~$i$ is pair-type, then $\varphi(R'_i)$ is disjoint from~$S$.

\begin{lemma}\label{lem:pairtype1}
    Let $(G,S)$ be a good pair and let~$F_t$ be an $(\ell,S)$-subframe in~$G$.
    If~$G_t$ has an admissible pair $(R,R')$ of~$F_t$, then $\varphi(R')$ is disjoint from~$S$.
\end{lemma}
\begin{proof}
    Suppose not.
    By \cref{lem:F1path}, both~$R$ and $P':=\varphi(R')$ are $F_1$-paths.
    Let~$P$ be the base of~$R$ in~$F_1$ and let~$a'$ and~$b'$ be the ends of~$P'$.
    By \cref{obs:central ears}, both $\abs{E(P\cup R)}$ and $\abs{E(P')}+\dist_{F_1}(a',b')$ are at most $\ell-1$.
    Since~$a'$ and~$b'$ are admissible vertices of~$F_t$, they are at distance at least~$\ell$ from every branch vertex of~$F_t$.
    Thus, $a'F_1b'$ is a path of $F^-_t$.
    If~$P$ is disjoint from $a'F_1b'$, then $P'=R'$.
    Thus, by concatenating~$P'$ and $a'F_1b'$, we obtain a cycle of $(F^-_t\oplus R)\cup R'$ of length at most $\ell-1$, contradicting that~$R'$ is an $(\ell,S)$-ear of $F^-_t\oplus R$.
    Thus, we may assume that~$P$ intersects $a'F_1b'$.
    Then $F_1\cup R\cup R'$ contains an $(\ell,S)$-cycle of length at most $2(\ell-1)$, a contradiction.
\end{proof}

We show that every $(\ell,S)$-ear of~$F^-_t$ in~$G_t$ is disjoint from~$S$ and of length at least~$\ell$.

\begin{lemma}\label{lem:ear property2}
    Let $(G,S)$ be a good pair and let~$F_t$ be an $(\ell,S)$-subframe in~$G$.
    Then every $(\ell,S)$-ear of~$F^-_t$ in~$G_t$ is disjoint from~$S$ and of length at least~$\ell$.
\end{lemma}
\begin{proof}
    Let~$P$ be an $(\ell,S)$-ear of~$F^-_t$ in~$G_t$ with ends~$a$ and~$b$.
    We first show that~$P$ is disjoint from~$S$.
    Suppose not.
    By \cref{lem:F1path}, both~$a$ and~$b$ are in~$F_1$.
    By \cref{obs:central ears}, we have $\abs{E(P)}+\dist_{F_1}(a,b)\leq\ell-1$.
    Since~$a$ and~$b$ are admissible vertices of~$F_t$, they are at distance at least~$\ell$ from every branch vertex of~$F_t$.
    Thus, $aF_1b$ is a path of $F^-_t$.
    Hence, $\abs{E(P)}+\dist_{F^-_t}(a,b)\leq\ell-1$, contradicting that~$P$ is an $(\ell,S)$-ear of~$F^-_t$.

    We now show that~$P$ has length at least~$\ell$.
    Suppose not.
    By \cref{lem:short path-sub}, we have $i:=\sigma(a)=\sigma(b)$.
    We first consider the case that $i=1$.
    By the definition of an $(\ell,S)$-central path, we have $\dist_{F_1}(a,b)\leq\abs{E(P)}\leq\ell-1$.
    Since~$a$ and~$b$ are admissible vertices of~$F_t$, they are at distance at least~$\ell$ from every branch vertex of~$F_t$.
    Thus, $aF_1b$ is a path of~$F^-_t$.
    By concatenating~$P$ and $aF_1b$, we obtain a cycle of $F^-_t\cup P$ of length at most $2(\ell-1)$ which is an $(\ell,S)$-cycle by the definition of an $(\ell,S)$-ear, a contradiction.

    We now consider the case that $i>1$.
    Since~$a$ and~$b$ are admissible vertices of~$F_t$, by \cref{lem:F1path}, neither~$a$ nor~$b$ is in $R_i\cup(a_iF_1b_i)$.
    Thus, $i$ is pair-type and both~$a$ and~$b$ are internal vertices of $\varphi(R'_i)$.
    By \cref{lem:pairtype1}, $\varphi(R'_i)$ is disjoint from~$S$.
    Since no internal vertex of $\varphi(R'_i)$ is in~$F_1$, by \cref{lem:F1path}, $\varphi(R'_i)$ is a subpath of~$F^-_t$.
    Thus, $a$ and~$b$ are in the same component of $F^-_t-S$.
    This contradicts that~$P$ is an $(\ell,S)$-ear of $F^-_t$ as~$P$ is disjoint from~$S$.
\end{proof}

The following lemma shows that if~$i$ is pair-type, then $\varphi(R'_i)$ has length at least~$\ell$.

\begin{lemma}\label{lem:pairtype2}
    Let $(G,S)$ be a good pair and let~$F_t$ be an $(\ell,S)$-subframe in~$G$.
    If~$G_t$ has an admissible pair $(R,R')$ of~$F_t$, then $\varphi(R')$ has length at least~$\ell$.
\end{lemma}
\begin{proof}
    Suppose that $P':=\varphi(R')$ has length at most $\ell-1$.
    By \cref{lem:F1path}, $R$ is an $F_1$-path.
    Let~$P$ be the base of~$R$ in~$F_1$ and let~$a'$ and~$b'$ be the ends of~$P'$.
    By \cref{lem:short path-sub}, we have $i:=\sigma(a')=\sigma(b')$.
    We first consider the case that $i=1$.
    Since~$a'$ and~$b'$ are admissible vertices of~$F_t$, they are at distance at least~$\ell$ from every branch vertex of~$F_t$.
    Since~$F_1$ is a shortest $(a_1,b_1)$-path of~$G$, we have $\abs{E(a'F_1b')}\leq\abs{E(P')}\leq\ell-1$.
    Thus, $a'F_1b'$ is a path of~$F^-_t$.
    If~$P$ is disjoint from $a'F_1b'$, then by concatenating~$P'$ and $a'F_1b'$, we obtain a cycle of $(F^-_t\oplus R)\cup R'$ of length at most $2(\ell-1)$ which is an $(\ell,S)$-cycle as~$R'$ is an $(\ell,S)$-ear of $F^-_t\oplus R$, a contradiction.
    Thus, we may assume that~$P$ intersects $a'F_1b'$.
    Then $F_1\cup R\cup R'$ contains an $(\ell,S)$-cycle of length at most $3(\ell-1)$, a contradiction.

    We now consider the case that $i>1$.
    Since~$a'$ and~$b'$ are admissible vertices of~$F_t$, by \cref{lem:F1path}, neither~$a'$ nor~$b'$ is in $R_i\cup(a_iF_1b_i)$.
    Thus, $i$ is pair-type and both~$a'$ and~$b'$ are internal vertices of~$\varphi(R'_i)$.
    By \cref{lem:pairtype1}, $\varphi(R'_i)$ is disjoint from~$S$.
    Since no internal vertex of $\varphi(R'_i)$ is in~$F_1$, by \cref{lem:F1path}, $\varphi(R'_i)$ is a subpath of~$F^-_t$.
    Thus, $a'$ and~$b'$ are in the same component of $F^-_t-S$.
    Since~$R$ and~$\varphi(R'_i)$ are vertex-disjoint, they are still in the same component of $(F^-_t\oplus R)-S$, contradicting that~$R'$ is an $(\ell,S)$-ear of $F^-_t\oplus R$.
\end{proof}

We now characterise all chords of an $(\ell,S)$-subframe.

\begin{proposition}\label{prop:subframe1}
    Let $(G,S)$ be a good pair and let~$F_t$ be an $(\ell,S)$-subframe in~$G$.
    If~$F_t$ has a chord~$ab$, then both~$a$ and~$b$ are in $B_{F_t}(x,\ell-1)$ for some branch vertex~$x$ of~$F_t$.
\end{proposition}
\begin{proof}
    Without loss of generality, we may assume that $i:=\sigma(a)\leq\sigma(b)=:j$.
    We divide into two cases according to whether $i=j$.
    
    \medskip
    \noindent
    \textbf{Case 1.} $i=j$.

    Note that $i\neq1$, because~$F_1$ is a shortest $(a_1,b_1)$-path of~$G$.
    By \cref{obs:central ears} and \cref{lem:F1path}, if both~$a$ and~$b$ are internal vertices of~$R_i$, then the statement holds.
    Thus, we may assume that~$i$ is pair-type and~$a$ or~$b$ is not an internal vertex of $R_i$.
    
    Suppose first that both~$a$ and~$b$ are internal vertices of $\varphi(R'_i)$.
    We may assume that for each $x\in\{a'_i,b'_i\}$, $a$ or~$b$ is not in $B_{R'_i}(x,\ell-1)$, because otherwise the statement holds for the vertex~$x$.
    Let~$R''$ be the path obtained from~$R'_i$ by replacing $aR'b$ with~$ab$.
    By the previous assumption, $R''$ has length at least $\ell+1$.
    By \cref{lem:pairtype1}, $\varphi(R'_i)$ is disjoint from~$S$.
    Thus, $R''$ is an $(\ell,S)$-ear of $F^-_{i-1}\oplus R_i$ shorter than~$R'_i$, contradicting~\ref{item:sub1}.

    We now suppose that exactly one of~$a$ and~$b$, say~$a$, is an internal vertex of~$R_i$.
    Since~$R_i$ is a strict $S$-ear, it has an end, say~$a_i$, such that $a_iR_ia$ together with~$ab$ is an $F_i$-path of~$G_i$ containing a vertex in~$S$.
    By \cref{lem:F1path}, the path is an $F_1$-path, contradicting that~$b$ is not a vertex of~$F_1$.

    \medskip
    \noindent
    \textbf{Case 2.} $i\neq j$.

    Suppose first that~$b$ is an internal vertex of~$R_j$.
    Since~$R_j$ is a strict $S$-ear, it has an end, say~$a_j$, such that $a_jR_jb$ together with~$ab$ is an $F_j$-path of~$G_j$ containing a vertex in~$S$.
    By \cref{lem:F1path}, the path is an $F_1$-path.
    Then by \cref{obs:central ears}, the statement holds for $x:=a_j$.

    We now suppose that~$j$ is pair-type and~$b$ is an internal vertex of $\varphi(R'_j)$.
    Note that~$a$ is an admissible vertex of $F_{j-1}$, because otherwise $b\in Z_{j-1}$.
    We divide into two subcases according to $\dist_{R'_j}(b,\{a'_j,b'_j\})$.

    \medskip
    \noindent
    \textbf{Case 2-1.} $\dist_{R'_j}(b,\{a'_j,b'_j\})=1$.

    Without loss of generality, we may assume that~$b$ is adjacent in~$R'_j$ to~$a'_j$.
    We may assume that $a\notin B_{F_t}(a'_j,\ell-1)$, because otherwise the statement holds for $x:=a'_j$.
    Since~$a'_j$ and~$a$ are admissible vertices of~$F_{j-1}$, they are vertices of~$F^-_{j-1}$ by \cref{lem:F1path}.
    Thus, $a'_jba$ is an $F^-_{j-1}$-path of~$G_{j-1}$.
    By \cref{lem:short path-sub} with~$F_{j-1}$ and $a'_jba$, we have $h:=\sigma(a'_j)=\sigma(a)$.
    Note that~$h$ is pair-type and both~$a'_j$ and~$a$ are internal vertices of $\varphi(R'_h)$.
    Let~$R''$ be the path obtained from~$R'_h$ by replacing $a'_jR'_ha$ with $a'_jba$.
    As $a\notin B_{F_t}(a'_j,\ell-1)$, $R''$ is shorter than~$R'_h$.
    Since~$a'_j$ and~$a$ are admissible vertices of~$F_{j-1}$, we have $\dist_{R'_j}(\{a'_j,a\},\{a'_h,b'_h\})\geq\ell$.
    Thus, $R''$ has length at least $2(\ell+1)$.
    By \cref{lem:pairtype1}, $\varphi(R'_h)$ is disjoint from~$S$.
    Thus, $R''$ is an $(\ell,S)$-ear of $F^-_{h-1}\oplus R_h$ shorter than~$R'_h$, contradicting~\ref{item:sub1}.

    \medskip
    \noindent
    \textbf{Case 2-2.} $\dist_{R'_j}(b,\{a'_j,b'_j\})\geq2$.

    Let~$a''_j$ and~$b''_j$ be the ends of~$R'_j$.
    Since $R'_j$ is an $(\ell,S)$-ear of $F^-_{j-1}\oplus R_j$, without loss of generality, we may assume the following:
    \begin{itemize}
        \item if $R'_j$ contains a vertex in~$S$, then~$a''_jR'_jb$ contains the vertex, and
        \item otherwise either $a\in S$ or~$a$ and~$a''_j$ are in distinct components of $(F^-_{j-1}\oplus R_j)-S$.
    \end{itemize}
    Let~$P'$ be the path obtained by concatenating $a''_jR'_jb$ and~$ab$.
    Note that~$P'$ is an $(F^-_{j-1}\oplus R_j)$-path in~$G_{j-1}$ as~$a$ is an admissible vertex of~$F_{j-1}$.
    By the assumption, $P'$ is an $S$-ear of $F^-_{j-1}\oplus R_j$.
    As $\dist_{R'_j}(b,\{a'_j,b'_j\})\geq2$, $P'$ is shorter than~$R'_j$.
    By~\ref{item:sub1}, $P'$ is not an $(\ell,S)$-ear of $F^-_{j-1}\oplus R_j$.
    Thus,
    \[
        \abs{E(a''_jR'_jb)}+\dist_{F_t}(a''_j,a)<\abs{E(P')}+\dist_{F^-_{j-1}\oplus R_j}(a''_j,a)\leq\ell-1,
    \]
    and therefore the statement holds for $x:=a''_j$.

    \medskip
    This completes the proof.
\end{proof}

In the following proposition, we show that given a maximal $(\ell,S)$-subframe, one can find a large induced packing of $(\ell,S)$-cycles or a small set of vertices whose closed neighbourhood meets all $(\ell,S)$-cycles of~$G$.

\begin{proposition}\label{prop:subframe2}
    There exists a function $g(k,\ell)=\mathcal{O}(\ell\cdot k\log k+\ell^3\cdot k)$ such that for all integers $k\geq1$ and $\ell\geq3$, every graph~$G$, and every set $S\subseteq V(G)$, if~$G$ admits an $(\ell,S)$-central path and has no $(\ell,S)$-cycle of length at most $3(\ell-1)$, then one can find, in $\abs{V(G)}^{\mathcal{O}(\ell)}$ time, either an induced packing of~$k$ $(\ell,S)$-cycles or a set~$X$ of at most $g(k,\ell)$ vertices such that $G-B_G(X,1)$ has no $(\ell,S)$-cycle.
\end{proposition}
\begin{proof}
    Let
    \[
        g(k,\ell):=(19\ell-15)s_k+\ell(\ell-1)(6\ell+1)k+29\ell-23
    \]
    Let~$F_t$ be a maximal $(\ell,S)$-subframe in~$G$ and let $F^-$ be the graph obtained from~$F^-_t$ by recursively removing degree-$1$ vertices.
    Note that~$F^-$ is a subcubic graph without pendant vertices.
    Suppose first that $F^-$ has at least $s_k$ branch vertices.
    By \cref{thm:simonovitz}, $F^-$ has~$k$ vertex-disjoint cycles.
    By \cref{lem:all cycles-sub}, they are $(\ell,S)$-cycles.
    By \cref{prop:subframe1}, they form an induced packing in~$G$.

    We now suppose that $F^-$ has less than~$s_k$ branch vertices.
    Note that~$F^-_t$ has at most two more branch vertices than~$F^-$.
    Then $\mathcal{L}_0(F^-_t)$ has at most
    \[
        \frac{3(s_k-1)}{2}+4=\frac{3s_k+5}{2}
    \]
    components.
    By~\ref{item:sub2}, each component of $\mathcal{L}_0(F^-_t)$ contains at most two $R_i$'s for $i\geq2$ as subgraphs.
    For such $R_i$, by \cref{lem:F1path}, we have $\abs{B_{F_t}(R_i,\ell-1)}\leq3(\ell-1)$.
    Since the maximum degree of~$F_t$ is at most~$4$, we have
    \begin{align*}
        \abs{Y_t}
        &\leq(4\ell-3)(s_k+1)+6(\ell-1)\cdot\frac{3s_k+5}{2}\\
        &\leq(13\ell-12)s_k+19\ell-18.
    \end{align*}
    Let $X_0$ be the set obtained from~$Y_t$ by adding $B_{F^-_t}(v,\ell-1)$ for each end~$v$ of $\mathcal{L}_1(F^-_t)$.
    Since $\mathcal{L}_1(F^-_t)$ has at most $(3s_k+5)/2$ components, we have
    \[
        \abs{X_0}\leq\abs{Y_t}+(2\ell-1)(3s_k+5)\leq(19\ell-15)s_k+29\ell-23.
    \]

    Let $G_0:=G_t-X_0$.
    As $Y_t\subseteq X_0$, by \cref{lem:F1path}, $F_t-X_0$ is an induced subgraph of~$F^-_t$ and each component of~$F^-_t$ contains exactly one component of $F_t-X_0$.
    Let~$\mathcal{C}_0$ be the set of all components of $F_t-X_0$ and let~$\mathcal{C}_1$ be the set of components of $F_t-X_0$ which are subpaths of $\mathcal{L}_1(F^-_t)$.
    By \cref{lem:F1path}, every component in~$\mathcal{C}_1$ is a subpath of~$F_1$.
    We will use the following claim.
    \begin{claim}\label{clm:disconnection}
        $G_0$ has no $F_t$-path between a component in~$\mathcal{C}_1$ and another component in~$\mathcal{C}_0$.
    \end{claim}
    \begin{subproof}
        Towards a contradiction, suppose that~$G_0$ has an $F_t$-path~$P$ between a component $C\in\mathcal{C}_1$ and another component in~$\mathcal{C}_0$.
        Let~$C'$ be the component of $\mathcal{L}_1(F^-_t)$ having~$C$ as a subpath.
        Note that both ends of~$C'$ are contained in~$S$.
        By the definition of an $(\ell,S)$-subframe, extending~$P$ from each of its ends along the paths in~$F_t$ towards~$F_1$ yields an $F_1$-path~$P'$ containing~$P$ as a subpath.
        Let~$v$ and~$v'$ be the ends of~$P'$ with $v\in V(C)$.
        As $v'\notin V(C')$, either $\{v,v'\}\cap S\neq\emptyset$, or~$v$ and~$v'$ are in distinct components of $F^-_t-S$.
        By the construction of~$X_0$, we have $\dist_{F_1}(v,v')\geq\ell-1$.
        Thus, $P'$ is an $(\ell,S)$-ear of~$F_1$, a contradiction.
    \end{subproof}

    Let~$C$ be an $(\ell,S)$-cycle of~$G_0$.
    Note that~$C$ intersects a component in~$\mathcal{C}_1$, because it has either a vertex in $V(F_1)\cap S$ or a strict $S$-ear of~$F_1$.
    By \cref{clm:disconnection}, $C$ does not intersect any other component in~$\mathcal{C}_0$.
    Thus, the base~$Q$ of~$C$ in~$F_t$ is a subpath of~$F_1$.
    Let~$a_C$ and~$b_C$ be the ends of~$Q$ such that~$a_C$ is closer to~$a_1$ than~$b_C$ in~$F_1$.
    Let $L(C)$, $M(C)$, and $R(C)$ be the subpaths of~$Q$ such that
    \begin{itemize}
        \item $L(C)$ and $R(C)$ are the components of $C\cap F_1$ containing~$a_C$ and~$b_C$, respectively, and
        \item $M(C)$ is the minimal subpath of~$Q$ containing all edges of~$Q$ not in $L(C)\cup R(C)$.
    \end{itemize}
    We denote by~$a'_C$ and~$b'_C$ the ends of~$M(C)$ such that~$a'_C$ is closer to~$a_1$ than~$b'_C$ in~$F_1$.
    We denote by $\eta(C)$ the number of subpaths of~$C$ which are $F_1$-paths.

    We show that for every $(\ell,S)$-cycle~$C$ of~$G_0$, $M(C)$ is not a null graph.
    Suppose not.
    Then $L(C)=R(C)=Q$.
    Thus, $C$ has a unique subpath~$P$ which is an $F_1$-path.
    Since~$F_1$ is a shortest $(a_1,b_1)$-path, we have $\abs{E(Q)}\leq\abs{E(P)}$.
    Since~$G$ has no $(\ell,S)$-cycle of length at most $3(\ell-1)$, we have $\abs{E(P)}>3(\ell-1)/2$.
    By \cref{lem:F1path}, $P$ is disjoint from~$S$.
    Thus, $Q$ contains a vertex in~$S$, and therefore~$P$ is an $(\ell,S)$-ear of~$F_1$, a contradiction.

    We greedily choose $(\ell,S)$-cycles $C_1,\ldots,C_h$ of~$G_0$ as follows: for each $i\in[h]$, among all $(\ell,S)$-cycles~$C$ of~$G_0$ whose $M(C)$ is vertex-disjoint from~$M(C_j)$ for every $j\in[i-1]$, $C_i$ is one with lexicographically smallest $(\abs{E(C_i)},\eta(C_i),\dist_{F_1}(a_1,a'_{C_i}),\dist_{F_1}(a_1,b'_{C_i}))$.
    We will use the following three claims.

    \begin{claim}\label{clm:medium}
        For every $i\in[h]$, $M(C_i)$ has length less than $(\ell-1)(\ell-2)$.
    \end{claim}
    \begin{subproof}
        Let~$v$ be a vertex in $V(C_i)\cap S$.
        Since $M(C_i)$ is not a null graph, it has a subpath~$P$ which is a $C_i$-path.
        Let~$a$ and~$b$ be the ends of~$P$ and let~$P'$ be a shortest $(a,b)$-subpath of~$C_i$ containing~$v$.
        Since~$F_1$ is a shortest $(a_1,b_1)$-path, we have $\abs{E(P)}\leq\abs{E(P')}$.
        By the choice of~$C_i$, $P\cup P'$ is not an $(\ell,S)$-cycle.
        Thus, $\abs{E(P\cup P')}\leq\ell-1$.
        Therefore, $P$ has length at most $(\ell-1)/2$ and both~$a$ and~$b$ are in $B_{C_i}(v,\ell-3)$.
        This implies that $M(C_i)$ has at most $2\ell-5$ subpaths which are $C_i$-paths, because each vertex of~$M(C_i)$ can be an end of at most two such $C_i$-paths.
        Therefore, we have $\abs{E(M(C_i))}\leq(2\ell-5)(\ell-1)/2<(\ell-1)(\ell-2)$.
    \end{subproof}

    For every $i\in[h]$, let $L'(C_i)$ and $R'(C_i)$ be the subpaths of~$C_i$ which are $F_1$-paths with ends~$a_{C_i}$ and~$b_{C_i}$, respectively.
    
    \begin{claim}\label{clm:middle-S}
        For every $i\in[h]$, if $L(C_i)$ has length at least $\ell/2-1$, then $L(C_i)$ and $L'(C_i)$ are disjoint from~$S$.
        The analogous statement holds for $R(C_i)$ and $R'(C_i)$.
    \end{claim}
    \begin{subproof}
        Since the proof is symmetric, we prove the claim only for $L(C_i)$ and $L'(C_i)$.
        Let~$b$ be the end of $L'(C_i)$ other than~$a_{C_i}$.
        Note that $b\neq a'_{C_i}$.
        Since~$F_1$ is a shortest $(a_1,b_1)$-path, $L'(C_i)$ is longer than $L(C_i)$.
        Then
        \[
            \abs{E(L'(C_i))}+\dist_{F_1}(a_{C_i},b)\geq\abs{E(L(C_i))}+\abs{E(L'(C_i))}+1\geq\ell.
        \]
        Since~$F_1$ has no $(\ell,S)$-ear, neither $L(C_i)$ nor $L'(C_i)$ contains a vertex in~$S$.
    \end{subproof}

    \begin{claim}\label{clm:wing}
        For an $F_1$-path~$P$ of~$G_t$, there are at most $2\ell$ integers $i\in[h]$ such that $M(C_i)$ is a subpath of the base of~$P$ in~$F_1$.
    \end{claim}
    \begin{subproof}
        Let~$Q$ be the base of~$P$ in~$F_1$.
        Towards a contradiction, suppose that there is an increasing sequence $i_1,\ldots,i_{2\ell+1}$ of integers in~$[h]$ such that for each $j\in[2\ell+1]$, $M(C_{i_j})$ is a subpath of~$Q$.
        Note that $B_{F_1}(M(C_{i_{\ell+1}}),2\ell)$ is a subpath of~$Q$.
        This implies that~$Q$ has length at least $4\ell+1$.
        Since~$P$ is not an $(\ell,S)$-ear of~$F_1$, $B_{F_1}(M(C_{i_{\ell+1}}),2\ell)$ is disjoint from~$S$.
        Then by \cref{clm:middle-S}, $C_{i_{\ell+1}}$ has an $F_1$-path~$P'$ containing a vertex in~$S$.
        Note that the ends of~$P'$ are not in~$P$.
        
        If $P$ contains an internal vertex of~$P'$, then $P\cup P'$ has an $F_1$-path~$P''$ containing a vertex in~$S$ whose base has length at least $2\ell$, contradicting \cref{obs:central ears}.
        Thus, $P$ and~$P'$ are vertex-disjoint.
        Since~$a$ and~$b$ are admissible vertices of~$F_t$, they are vertices of~$F^-_t$.
        Since~$Q$ is a subpath of a component in~$\mathcal{C}_1$, either $\{a,b\}\cap S\neq\emptyset$, or~$a$ and~$b$ are in distinct components of $F^-_t-S$.
        As $\abs{E(Q)}\geq4\ell+1$, $P$ is an $(\ell,S)$-ear of $F^-_t\oplus P'$.
        Thus, $(P',P)$ is an admissible pair of~$F_t$, contradicting the maximality of~$F_t$.
    \end{subproof}

    We now find either an induced packing of~$k$ $(\ell,S)$-cycles or a desired set~$X$.
    Suppose first that $h\geq(6\ell+1)k$.
    For each $j\in[k]$, let $i_j:=2\ell(3j-2)+j$.
    By \cref{clm:wing}, if the base of $C_{i_j}$ intersects $M(C_r)$ for some $r\in[h]$, then $i_j-2\ell\leq r\leq i_j+2\ell$, because each of $L(C_{i_j})$ and $R(C_{i_j})$ contains at most $2\ell$ such paths $M(C_r)$.
    Thus, for all $j,j'\in[k]$ with $j<j'$, the bases of~$C_{i_j}$ and~$C_{i_{j'}}$ are vertex-disjoint.
    If $\{C_{i_j},C_{i_{j'}}\}$ is not an induced packing in~$G$, then $C_{i_j}\cup C_{i_{j'}}$ has an $F_1$-path~$P$ of~$G_0$ between the bases of~$C_{i_j}$ and~$C_{i_{j'}}$.
    Thus, there are at least $2\ell$ integers $r\in[h]$ such that $M(C_r)$ is a subpath of the base of~$P$ in~$F_1$, contradicting \cref{clm:wing}.
    Therefore, $\{C_{i_j}:j\in[k]\}$ is an induced packing of $k$ $(\ell,S)$-cycles in~$G$.
    
    We now suppose that $h<(6\ell+1)k$.
    Let~$X$ be the set obtained from~$X_0$ by adding $B_{F_1}(M(C_i),\ell-1)$ for every $i\in[h]$.
    By \cref{clm:medium}, we have
    \[
        \abs{B_{F_1}(M(C_i),\ell-1)}<(\ell-1)(\ell-2)+2(\ell-1)=\ell(\ell-1)
    \]
    for each $i\in[h]$.
    Thus, we have
    \[
        \abs{X}<\abs{X_0}+\ell(\ell-1)(6\ell+1)k=g(k,\ell).
    \]
    
    Towards a contradiction, suppose that $G-B_G(X,1)$ has an $(\ell,S)$-cycle~$C$.
    Since~$C$ is disjoint from~$X$, by the greedy choice of $C_1,\ldots,C_h$, $M(C)$ contains $A:=B_{F_1}(M(C_i),\ell-1)$ for some $i\in[h]$.
    Thus, $C$ has an $F_1$-path~$P$ as a subpath whose base in~$F_1$ contains all vertices in~$A$.
    Since~$P$ is not an $(\ell,S)$-ear of~$F_1$, $A$ is disjoint from~$S$.
    By \cref{clm:middle-S}, $C_i$ has a strict $S$-ear~$P'$ of~$F_1$ as a subpath.
    Then $(P',P)$ is an admissible pair of~$F_t$, contradicting the maximality of~$F_t$.
    Hence, $G-B_G(X,1)$ has no $(\ell,S)$-cycles, as desired.
\end{proof}

\section{A proof of the main theorem}\label{sec:proof}

We now prove \cref{thm:main1}.

\mainone*

\begin{proof}
    We proceed by induction on~$k$.
    For $k=1$, the statement holds by taking~$f(1,\ell)$ as~$1$.
    Thus, we may assume that~$k\geq2$.
    If~$G$ has no $(\ell,S)$-cycle, then the statement holds by taking~$X$ as an empty set.
    Thus, we may assume that~$G$ has an $(\ell,S)$-cycle.
    Let $g(k,\ell)$ be the function of \cref{prop:subframe2} and let
    \begin{align*}
        h(k):=&10^{20}\cdot k^9(2k-1)^4+2s_k,\\
        f(k,\ell):=&(4\ell-3)(13h(k)+24k(6k-5))+(k-1)g(k,\ell).
    \end{align*}
    As $g(k,\ell)=\mathcal{O}(\ell\cdot k\log k+\ell^3\cdot k)$, we have $f(k,\ell)=\mathcal{O}(\ell\cdot k^{13}+\ell^3\cdot k^2)$.
    We remark that for every $k\geq2$, it holds that $f(k-1,\ell)+3(\ell-1)\leq f(k,\ell)$.
    
    Suppose that~$G$ has an $(\ell,S)$-cycle~$C$ of length at most $3(\ell-1)$.
    Let $G':=G-B_G(C,1)$.
    By the inductive hypothesis for $k-1$, $G'$ has either an induced packing~$\mathcal{C}$ of $k-1$ $(\ell,S)$-cycles or a set~$X'$ of at most $f(k-1,\ell)$ vertices such that $G'-B_{G'}(X',1)$ has no $(\ell,S)$-cycle.
    In the former, $\mathcal{C}\cup\{C\}$ is an induced packing in~$G$.
    Thus, we may assume that the latter holds.
    Let $X:=X'\cup V(C)$.
    Note that $G-B_G(X,1)=G'-B_{G'}(X',1)$ and
    \[
        \abs{X}\leq f(k-1,\ell)+3(\ell-1)\leq f(k,\ell).
    \]
    Thus, $X$ is a desired set, and therefore the statement holds.

    Hence, we may assume that~$G$ has no $(\ell,S)$-cycle of length at most $3(\ell-1)$, that is, $(G,S)$ is a good pair.
    Let~$\mathcal{H}$ be a maximal $(\ell,S)$-frame in~$G$ and let~$t$ be the depth of~$\mathcal{H}$.
    Let~$Y$ be the set of
    \begin{itemize}
        \item all branch vertices of~$\mathcal{H}$,
        \item the ends of every component of $\mathcal{L}_1(\mathcal{H}^-)$, and
        \item the ends of~$Q_i$ for every detached $\mathcal{H}$-earring~$P_{i,1}$.
    \end{itemize}
    
    We first bound the size of~$Y$.
    By \cref{lem:aux2}, we may assume that $\abs{\mathcal{P}_1(\mathcal{H})}<3k(6k-5)$.
    By \cref{prop:final}, we may assume that~$\mathcal{H}^-$ has less than $h(k)$ branch vertices.
    Since $\mathcal{H}^-$ has maximum degree at most~$4$, the number of components of $\mathcal{L}_0(\mathcal{H}^-)$ is at most
    \[
        \frac{4\abs{V_{\geq3}(\mathcal{H}^-)}}{2}+\abs{\mathcal{P}_1(\mathcal{H})}<2h(k)+3k(6k-5).
    \]
    By~\ref{item:ear4}, each component of $\mathcal{L}_0(\mathcal{H}^-)$ has at most two $\mathcal{H}$-detours as subpaths.
    Thus, we have
    \[
        \abs{V_{\geq3}(\mathcal{H})}\leq\abs{V_{\geq3}(\mathcal{H}^-)}+4(2h(k)+3k(6k-5))<9h(k)+12k(6k-5),
    \]
    and therefore
    \begin{align*}
        \abs{Y}
        &<9h(k)+12k(6k-5)+2(2h(k)+3k(6k-5))+2\abs{\mathcal{P}_1(\mathcal{H})}\\
        &<13h(k)+24k(6k-5).
    \end{align*}

    We now find a desired set~$X$.
    Let
    \begin{align*}
        X_0&:=B_\mathcal{H}(Y,\ell-1),\\
        Z&:=B_{G-(V(\mathcal{H})\setminus X_0)}(X_0,1).
    \end{align*}
    Since~$\mathcal{H}$ has maximum degree at most~$4$, we have $\abs{X_0}\leq(4\ell-3)\abs{Y}$.
    Note that $Y_{t,\mu(t)}\subseteq X_0$ and hence $Z_{t,\mu(t)}\subseteq Z$.
    Thus, $G^*:=G-Z$ is an induced subgraph of $G_{t,\mu(t)}$.
    In addition, by \cref{lem:strict1}, $\mathcal{H}-X_0$ is an induced subgraph of~$\mathcal{H}^-$ and each component of~$\mathcal{H}^-$ contains exactly one component of $\mathcal{H}-X_0$.
    Let~$\mathcal{C}_0$ be the set of all components of $\mathcal{H}-X_0$ and let~$\mathcal{C}_1$ be the set of components of $\mathcal{H}-X_0$ which are subpaths of $\mathcal{L}_1(\mathcal{H}^-)$.
    We will use the following claim.

    \begin{claim}\label{clm:main1}
        $G^*$ has no $\mathcal{H}$-path between a component in~$\mathcal{C}_1$ and another component in~$\mathcal{C}_0$.
    \end{claim}
    \begin{subproof}
        Towards a contradiction, suppose that~$G^*$ has an $\mathcal{H}$-path~$P$ between a component $C\in\mathcal{C}_1$ and another component $C'\in\mathcal{C}_0$.
        Let~$v$ and~$v'$ be the ends of~$P$ and let~$C''$ be the component of $\mathcal{L}_1(\mathcal{H}^-)$ having~$C$ as a subpath.
        Note that both ends of~$C''$ are contained in~$S$.
        Since~$C'$ is vertex-disjoint from~$C''$, either $\{v,v'\}\cap S\neq\emptyset$, or~$v$ and~$v'$ are in distinct components of $\mathcal{H}^--S$.
        By the construction of~$X_0$, we have $\dist_\mathcal{H}(v,v')\geq2\ell$.
        Thus, $P$ is an $(\ell,S)$-ear of~$\mathcal{H}^-$, contradicting the maximality of~$\mathcal{H}$.
    \end{subproof}

    Let~$C$ be an $(\ell,S)$-cycle of~$G^*$.
    Note that~$C$ intersects a component in~$\mathcal{C}_1$, because it has either a vertex in $V(\mathcal{H})\cap S$, or a subpath which is a strict $S$-ear of~$\mathcal{H}$.
    By \cref{clm:main1}, $C$ does not intersect any other component in~$\mathcal{C}_0$.
    Let $L_1,\ldots,L_i$ be the components of $\mathcal{H}-X_0$ intersecting $(\ell,S)$-cycles.
    For each $j\in[i]$, let $G_j$ be the component of~$G^*$ having~$L_j$ as a subpath.
    By \cref{clm:main1}, $G_j\neq G_{j'}$ for all distinct $j,j'\in[i]$.
    Thus, we may assume that $i<k$, because each~$G_j$ has an $(\ell,S)$-cycle.

    We show that for each $j\in[i]$, $L_j$ is an $(\ell,S)$-central path of~$G_j$.
    By \ref{item:ear1}--\ref{item:ear3}, $L_j$ is a shortest path between its ends.
    By the maximality of~$\mathcal{H}$, $G_j$ has no $(\ell,S)$-ear of~$L_j$, and every $(\ell,S)$-cycle of~$G_j$ contains at least two vertices of~$L_j$.
    Thus, $L_j$ is an $(\ell,S)$-central path of~$G_j$.

    For each $j\in[i]$, we now apply \cref{prop:subframe2}.
    We may assume that for every $j\in[i]$, we have found a set $X_j$ of at most $g(k,\ell)$ vertices such that $G_j-B_{G_j}(X_j,1)$ has no $(\ell,S)$-cycle.
    Let
    \[
        X:=X_0\cup\bigcup_{j=1}^iX_j.
    \]
    Then $G-B_G(X,1)$ has no $(\ell,S)$-cycle.
    Note that
    \[
        \abs{X}\leq(4\ell-3)\abs{Y}+(k-1)g(k,\ell)<f(k,\ell).
    \]
    This completes the proof.
\end{proof}

\section{Open problems}\label{sec:last}

We conclude the paper with several open problems.
The first one is about $S$-cycles.
It has been shown that the classic Erd\H{o}s--P\'{o}sa theorem can be extended to $S$-cycles with an asymptotically optimal function~\cite{PW2012}.

\begin{theorem}[Pontecorvi and Wollan~\cite{PW2012}]
    There exists a function $f(k)=\mathcal{O}(k\log k)$ such that for every positive integer~$k$, every graph~$G$, and every set $S\subseteq V(G)$, $G$ has either~$k$ pairwise vertex-disjoint $S$-cycles, or a set~$X$ of at most $f(k)$ vertices such that $G-X$ has no $S$-cycles.
\end{theorem}

By setting $\ell=3$, \cref{thm:main1} becomes an Erd\H{o}s--P\'{o}sa theorem for induced packings of $S$-cycles.
We remark that for $\ell=3$, \cref{thm:main1} can be improved with a function~$f(k)$ of order $\mathcal{O}(k^5)$.
The main reason is that every strict $S$-ear of a $(3,S)$-frame~$\mathcal{H}$ becomes a $(3,S)$-ear of~$\mathcal{H}$, so that we can simplify a large portion of our proofs.
We conjecture that the order $\mathcal{O}(k\log k)$ works for an Erd\H{o}s--P\'{o}sa theorem for induced packings of $S$-cycles as well.

\begin{conjecture}
    There exists a function $f(k)=\mathcal{O}(k\log k)$ such that for every positive integer~$k$, every graph~$G$, and every set $S\subseteq V(G)$, one can find, in polynomial time, either an induced packing of~$k$ $S$-cycles or a set~$X$ of at most $f(k)$ vertices such that $G-B_G(X,1)$ has no $S$-cycle.
\end{conjecture}

Long cycles are extensively studied in the literature on the Erd\H{o}s–P\'{o}sa theorem.
Mousset, Noever, \v{S}kori\'{c}, and Weissenberger~\cite{MoussetNSW17} showed that an Erd\H{o}s--P\'{o}sa theorem holds for long cycles with an asymptotically optimal function.

\begin{theorem}[Mousset, Noever, \v{S}kori\'{c}, and Weissenberger~\cite{MoussetNSW17}]\label{thm:EPlong}
    There exists a function $f(k,\ell)=\mathcal{O}(\ell\cdot k+k\log k)$ such that for all integers $k\geq1$ and $\ell\geq3$, every graph~$G$ has either~$k$ pairwise vertex-disjoint cycles of length at least~$\ell$, or a set~$X$ of at most $f(k,\ell)$ vertices such that $G-X$ has no cycle of length at least~$\ell$.
\end{theorem}

Recently, Ahn, Gollin, Huynh, and Kwon~\cite{coarseEParxiv} showed an Erd\H{o}s--P\'{o}sa theorem for induced packings of long cycles.

\begin{theorem}[Ahn, Gollin, Huynh, and Kwon~\cite{coarseEParxiv}]\label{thm:longcoarse}
    There exist functions $f(k,\ell) = \mathcal{O}(\ell\cdot k\log k)$ and $g(k)=\mathcal{O}(k\log k)$ such that for all integers $k\geq1$ and $\ell\geq3$, every graph~$G$ has either an induced packing of~$k$ cycles of length at least~$\ell$, or two sets~$X_1$ and~$X_2$ of vertices with $\abs{X_1}\leq f(k,\ell)$ and $\abs{X_2}\leq g(k)$ such that neither ${G-B_G(X_1,1)}$ nor ${G-B_G(X_2,\ell)}$ has cycles of length at least~$\ell$.
\end{theorem}

We remark that the function $f(k,\ell)$ in \cref{thm:longcoarse} is asymptotically tight when~$\ell$ is a fixed constant, because \cref{thm:longcoarse} for any fixed~$\ell$ implies the classic Erd\H{o}s--P\'{o}sa theorem.

Bruhn, Joos, and Schaudt~\cite{BruhnJS18} showed an Erd\H{o}s--P\'{o}sa theorem for $(\ell,S)$-cycles as follows.

\begin{theorem}[Bruhn, Joos, and Schaudt~\cite{BruhnJS18}]\label{thm:longS}
    There exists a function $f(k,\ell)=\mathcal{O}(\ell\cdot k\log k)$ such that for all integers $k\geq1$ and $\ell\geq3$, every graph~$G$, and every set $S\subseteq V(G)$, one can find, in $\abs{V(G)}^{\mathcal{O}(\ell)}$ time, either~$k$ pairwise vertex-disjoint $(\ell,S)$-cycles, or a set~$X$ of at most $f(k,\ell)$ vertices such that $G-X$ has no $(\ell,S)$-cycle.
\end{theorem}

Inspired by \cref{thm:EPlong,thm:longcoarse,thm:longS}, we conjecture that the order $\mathcal{O}(\ell\cdot k+k\log k)$ works for an Erd\H{o}s--P\'{o}sa theorem for induced packings of $(\ell,S)$-cycles.

\begin{conjecture}
    There exists a function $f(k,\ell)=\mathcal{O}(\ell\cdot k+k\log k)$ such that for all integers $k\geq1$ and $\ell\geq3$, every graph~$G$, and every set $S\subseteq V(G)$, one can find, in $\abs{V(G)}^{\mathcal{O}(\ell)}$ time, either an induced packing of~$k$ $(\ell,S)$-cycles or a set~$X$ of at most $f(k,\ell)$ vertices such that $G-B_G(X,1)$ has no $(\ell,S)$-cycle.
\end{conjecture}

Since our proof relies on applying \cref{thm:Dilworth} several times, obtaining a significantly better bound would likely require a different approach.

We can impose constraints not only on cycles but also on packings.
For an integer $d\geq0$, a \emph{distance-$d$ packing} of cycles in~$G$ is a set of pairwise vertex-disjoint cycles such that~$G$ has no path of length at most~$d$ between distinct cycles in the packing.
Hence, distance-$0$ packings are usual packings, distance-$1$ packings are induced packings.
Recently, Dujmović, Joret, Micek, and Morin~\cite{DujmovicJMM} showed an Erd\H{o}s--P\'{o}sa theorem for distance-$d$ packings of cycles.

\begin{theorem}[Dujmović, Joret, Micek, and Morin~\cite{DujmovicJMM}]
    There is a function $f(k)=\mathcal{O}(k^{18}\log^{18}k)$ such that for all integers $k\geq1$ and $d\geq0$, every graph~$G$ has either a distance-$d$ packing of~$k$ cycles, or a set~$X$ of at most $f(k)$ vertices such that $G-B_G(X,19d)$ has no cycles.
\end{theorem}

We conjecture that an analogue statement holds for $(\ell,S)$-cycles.

\begin{conjecture}
    There exist functions $f(k,\ell)=\mathcal{O}(\ell\cdot k+k\log k)$ and $C(\ell,d)=\mathcal{O}(\ell+d)$ such that for all integers $k\geq1$, $\ell\geq3$, and $d\geq0$, every graph~$G$, and every set $S\subseteq V(G)$, $G$ has either a distance-$d$ packing of~$k$ $(\ell,S)$-cycles, or a set~$X$ of at most $f(k,\ell)$ vertices such that $G-B_G(X,C(\ell,d))$ has no $(\ell,S)$-cycles.
\end{conjecture}

\providecommand{\bysame}{\leavevmode\hbox to3em{\hrulefill}\thinspace}
\providecommand{\MR}{\relax\ifhmode\unskip\space\fi MR }
\providecommand{\MRhref}[2]{%
  \href{http://www.ams.org/mathscinet-getitem?mr=#1}{#2}
}
\providecommand{\href}[2]{#2}

\end{document}